\documentclass[9pt,a4paper,reqno]{amsart}

\usepackage{amsmath}
\usepackage{mathptmx} 
\usepackage{amssymb}
\usepackage{amsthm}
\usepackage{txfonts}
\usepackage{microtype}
\usepackage[usenames,dvipsnames]{xcolor}
\usepackage{graphicx}
\usepackage{enumitem}
\usepackage{dsfont}
\usepackage{hyperref}
\usepackage[utf8]{inputenc}
\usepackage[arrow, matrix, curve]{xy}
\usepackage{multicol}
\usepackage{mathrsfs}

\theoremstyle{plain}
\newtheorem{theorem}{Theorem}[section]
\newtheorem{lemma}[theorem]{Lemma}
\newtheorem{proposition}[theorem]{Proposition}

\theoremstyle{definition}
\newtheorem{definition}{Definition}[section]

\theoremstyle{remark}
\newtheorem{remark}{Remark}

\allowdisplaybreaks

\title[On the convergence to equilibrium of unbounded observables]{On the convergence to equilibrium of unbounded observables under a family of intermittent  interval maps}

\thanks{The first two authors were supported by the German Research Foundation (DFG) grant 
\textit{Renewal Theory and Statistics of Rare Events in Infinite Ergodic Theory} ({Gesch{\"a}ftszeichen} KE 1440/2-1).}

\author{J.~Kautzsch}
\address{Fachbereich 3 Mathematik, Universit\"at Bremen, Bibliothekstr. 1, 28359 Bremen, Germany}
\email{kautzsch@math.uni-bremen.de}

\author{M.~Kesseb\"ohmer}
\address{Fachbereich 3 Mathematik, Universit\"at Bremen, Bibliothekstr. 1, 28359 Bremen, Germany}
\email{mhk@math.uni-bremen.de}
                      
\author{T.~Samuel}
\address{Fachbereich 3 Mathematik, Universit\"at Bremen, Bibliothekstr. 1, 28359 Bremen, Germany}
\email{tony@math.uni-bremen.de}

\date{\today}

\begin{document}

\maketitle

\begin{abstract}
We consider a family $\{ T_{r} \colon [0, 1] \circlearrowleft \}_{r \in [0, 1]}$ of Markov interval maps interpolating between the Tent map $T_{0}$ and the Farey map $T_{1}$.  Letting $\mathcal{P}_{r}$ denote the Perron-Frobenius operator of $T_{r}$, we show, for $\beta \in [0, 1]$ and $\alpha \in (0, 1)$, that the asymptotic behaviour of the iterates of $\mathcal{P}_{r}$ applied to observables with a singularity at $\beta$ of order $\alpha$ is dependent on the structure of the $\omega$-limit set of $\beta$ with respect to $T_{r}$.    Having a singularity it seems that such observables do not fall into any of the function classes on which convergence to equilibrium has been previously shown.
\end{abstract}

\section{Introduction}

Expanding maps of the unit interval have been widely studied in the last decades and the associated transfer operators have proven to be of vital importance in solving problems concerning the statistical behaviour of the underlying interval maps \cite{B:2000,C:1996,M:1991}.  

In recent years an increasing amount of  interest has developed in maps which are expanding everywhere except on an unstable fixed point (that is, an indifference fixed point) at which trajectories are considerably slowed down.  This leads to an interplay of chaotic and regular dynamics, a characteristic of intermittent systems \cite{PM:1980,S:1988}.  From an ergodic theory viewpoint, this phenomenon leads to any absolutely continuous invariant measure having infinite mass.  Therefore, standard methods of ergodic theory cannot be applied in this setting; indeed it is well known that Birkhoff's ergodic theorem does not hold under these circumstances, see for instance \cite{JA:1997,JA:1981}.

We consider a family $\{ T_{r} \colon [0, 1] \circlearrowleft \}_{r \in [0, 1]}$ of Markov interval maps interpolating between the Tent map $T_{0}$ and the Farey map $T_{1}$.  These interpolating maps, we believe, were first defined in \cite{EIK:06,GI:2005}, and have since attracted much attention.  For $r \in [0, 1]$, the map $T_{r} \colon [0, 1] \circlearrowleft$ is defined by
\begin{align*}
T_{r}(x) \coloneqq
\begin{cases}
\displaystyle{\frac{(2-r)\cdot x}{1-r \cdot x}} & \text{if} \; 0\leq x \leq 1/2,\\[1em]
\displaystyle{\frac{(2-r) \cdot (1-x)}{1-r+r \cdot x}} & \text{if} \; 1/2 < x \leq 1.
\end{cases}
\end{align*}
For $r \in [0, 1)$, many properties of these maps are given in \cite{EIK:06,GI:2005} and due to the piecewise monotonicity of each $T_{r}$, for $r \in [0, 1)$, several results about the associated Perron-Frobenius operator $\mathcal{P}_{r}$, can be deduced from, for instance, \cite{B:2000,K:1984}.  These latter results can not be applied to the Perron-Frobenius operator $\mathcal{P}_{1}$ of the Farey map $T_{1}$, since any absolutely continuous $T_{1}$-invariant measure is infinite, whereas, for $r \in [0, 1)$, there exists a unique absolutely continuous $T_{r}$-invariant probability  measure $\mu_{r}$.  (See Section \ref{sec:notation} for the definition of $\mathcal{P}_{r}$.)  However, recent advancements have been made on the asymptotic behaviour of $\mathcal{P}_{1}$, see \cite{KKSS:2014,MT:2011}.

For $r \in [0, 1)$, from the results of \cite{K:1984} it can be deduced that the essential spectral radius of $\mathcal{P}_{r}$ restricted to the Banach space of functions of bounded variation is equal to $1/(2 - r)$.  Moreover, in \cite{GI:2005},  for $r \in [0, 1]$, a Hilbert space of analytic functions which is left invariant by each $\mathcal{P}_{r}$ is constructed, and the spectrum of each $\mathcal{P}_{r}$ restricted to this Hilbert space is studied.  Here we extend and complement results of \cite{B:2000,HK:1982,K:1984,R:1983} on the convergence to equilibrium in one-dimensional systems.  In particular, it has been shown, for various classes of regular functions (such as functions of bounded variation and Lipschitz continuous, H\"older continuous, piecewise H\"older continuous and $C^{1+\epsilon}$ functions), that if $f$ belongs to one of these classes then, for $r \in [0, 1)$, uniformly on $[0, 1]$, we have that
\begin{align}\label{eq:intro_convergence}
\lim_{n \to \infty} \mathcal{P}_{r}^{n}(f) = \int f \, \mathrm{d}\lambda \cdot h_{r}.
\end{align}
Here $\lambda$ denotes the one-dimensional Lebesgue measure and $h_{r} \coloneqq \mathrm{d} \mu_{r} / \mathrm{d}\lambda$.  Using arguments similar to those given in \cite{T:2000} one can also prove the above convergence for proper Riemann integrable functions.  Applying arguments similar to those presented in \cite{KKSS:2014,MT:2011}, one can also show that, if $f$ belongs to a certain class of regular functions, then uniformly on compact subsets of $(0, 1]$
\begin{align*}
\lim_{n \to \infty} \ln(n) \cdot \mathcal{P}_{1}^{n}(f) = \int f \, \mathrm{d}\lambda \cdot h_{1}.
\end{align*}
One of our main contributions to this theory is given in Theorem~\ref{thm:Convergence_Equilibrum} where we show that the convergence given in \eqref{eq:intro_convergence} also holds for improper Riemann integrable functions with a finite number of singularities and that the type of convergence depends on the structure of the $\omega$-limit set of the singularities with respect to $T_{r}$, for $r \in [0, 1)$.

We also study the case when $r = 1$, for which any absolutely continuous invariant measure has infinite mass. Thaler \cite{T:2000} was the first to discern the asymptotics of the Perron-Frobenius operator of a class of interval maps preserving an infinite measure.  This class of maps, to which the Farey map does not belong, have become to be known as Thaler maps.  In an effort to generalise this work, by combining  renewal theoretical arguments and functional analytic techniques, a new approach to estimate the decay of correlation of a dynamical system was achieved by Sarig \cite{S:2002}.  Subsequently, Gou\"{e}zel \cite{G:2004,G:2005,G:2011} generalised these methods.  Using these ideas and employing the methods of Garsia and Lamperti \cite{GL:1969}, Erickson \cite{E:1970} and Doney \cite{D:1997}, recently Melbourne and Terhesiu \cite{MT:2011} proved a landmark result on the asymptotic rate of convergence of the `return time operator' (see Section \ref{sec:Section_4_3_1}) and showed that these result can be applied to Gibbs-Markov maps, Thaler maps, AFN maps, and Pomeau-Manneville maps.  Thus, the question which naturally arises is, whether this asymptotic rate can be related to the asymptotic rate of convergence of iterates of the transfer operator itself and hence the Perron-Frobenius operator. This was already partially deduced in \cite{KKSS:2014,MT:2011},  namely, for a specific class observables which are bounded.  In this article we present a proof of this result for the Farey map (Theorems~\ref{thm:MT:2001:sec10} and~\ref{prop:jk2011:th7}) and moreover show that this class of observables can be extended (Theorem~\ref{thm:Infinite_1}).  Indeed we compute the asymptotic behaviour of the iterates of the Perron-Frobenius operator $\mathcal{P}_{r}$ acting on an observable with a finite number of singularities, and show that the type of convergence depends on the structure of the $\omega$-limit set, with respect to $T_{1}$, of the singularities.

Let us take the opportunity to say a few words on the proofs of our main theorems.  The proofs of our results for $r \in [0,1)$ rely on arguments from ergodic theory, for instance those which can be found in \cite{B:2000,K:1984,R:1983}, together with the principle of bounded distortion.  For the case $r = 1$ more sophisticated methods are required.  Indeed we use results of \cite{MT:2011} which are based on operator renewal techniques which require Banach spaces with certain properties (see Page \pageref{page:refR1-R5}).  To obtain refined results on the set of points of non-convergence, that is to show it is of Hausdorff dimension zero, it is important to choose a Banach space which distinguishes functions point-wise.

We remark that from an ergodic theory point of view the Farey map is of great interest since it is expanding everywhere except at the indifferenced fixed point where it has (right) derivative one. This makes the Farey map a simple model of physical phenomenon such as intermittency \cite{PM:1980}.  Further, from the viewpoint of number theory, the Farey map encodes the continued fraction algorithm as well as the Riemann zeta function.  In particular, it has an induced version topologically conjugate to the Gauss map \cite{M:1991}.  Also, several models of statistical mechanics have been considered in recent years in connection to the Farey map and continued fractions \cite{FKO:2003,PO:1999,K:1993,K:1998,LR:1996}.

Finally, we would like to acknowledge that this work has arisen out of our attempts to understand and generalise the work of \cite{MT:2011,T:2000}.

\subsection{Outline}
\mbox{ }

In the following section we present essential definitions and state various preliminary results.  In Section~\ref{sec:main} we formally state our results.  Several further definitions and preliminary results are given in Section~\ref{sec:preliminaries}.  We divide this section into three parts.  In the first part we present some properties of functions of bounded variation, the second part contains preliminaries for the case when $r \in [0, 1)$ and the third part contains preliminaries for the case when $r = 1$.  In this latter case, namely when $r= 1$, we present two key results (Theorems~\ref{thm:MT:2001:sec10} and \ref{prop:jk2011:th7}). These results provide mild conditions under which the asymptotic behaviour of iterates of the Farey transfer operator $\widehat{T}_{1}$ (and hence the Perron-Frobenius operator $\mathcal{P}_{1}$) can be deduced from the asymptotic behaviour of the first return time operators.  Although, Theorem~\ref{thm:MT:2001:sec10} appears in \cite{MT:2011}, recently a counterexample was given in \cite{KKSS:2014} which shows that this result does not hold in the full generality as stated in \cite{MT:2011}.  Thus, here we present a full proof of this result.  Further, in the case that $r = 1$, we will make use of \cite[Theorem 2.1]{MT:2011} for which we require the existence of a Banach space with certain properties.  Such a Banach space is described in Proposition~\ref{prop:rmk1}.  Analogous results in an $\mathcal{L}^{1}$ setting are abundant in the current literature, the Banach space considered here differs in that it distinguishes functions point-wise and so at the end of this article (Section~\ref{sec:appendix}) we include a full proof.  In Section~\ref{sec:proofs} we give the proofs of our main results, Theorems~\ref{thm:Convergence_Equilibrum}, \ref{thm:Infinite_1} and~\ref{thm:Infinite_2}.

\subsection{Notation}
\mbox{ }

The natural numbers will be denoted by $\mathbb{N}$, the real numbers by $\mathbb{R}$ and the complex numbers by $\mathbb{C}$.  We will also use the symbol $\mathbb{N}_{0}$ to denote the set of non-negative integers, $\mathbb{R}^{+}$ to denote the set of positive real numbers and $\overline{\mathbb{R}}$ to denote the extended real numbers, namely $\overline{\mathbb{R}} = \mathbb{R} \cup \{ \pm \infty\}$.

Following convention, we use the symbol $\sim$ between the elements of two sequences of real or complex numbers $(b_{n})_{n \in \mathbb{N}}$ and $(c_{n})_{n \in \mathbb{N}}$ to mean that the sequences are asymptotically equivalent, namely that $\lim_{n \to +\infty} b_{n}/c_{n} = 1$, and we use the Landau notation $b_{n} = \mathfrak{o}(c_{n})$ if $\lim_{n \to +\infty} b_{n}/c_{n} = 0$.  The same notation is used between two $\mathbb{R}$-valued or $\mathbb{C}$-valued function $f$ and $g$; that is, if $\lim_{x \to + \infty} f (x)/g(x) = 0$, then we write $f =  \mathfrak{o}(g)$.  

\section{Central definitions}\label{sec:notation}

For $r \in [0, 1]$, the map $T_{r}$ has two fixed points, one at zero and one at $1 - (3 - \sqrt{9 - 4r})/(2r)$.  The inverse branches $f_{r, 0}, f_{r, 1} \colon [0, 1] \circlearrowleft$ of $T_{r}$ are given by
\begin{align*}
f_{r, 0}(x) \coloneqq \frac{x}{2 - r + r \cdot x}
\quad \text{and} \quad
f_{r, 1}(x) \coloneqq \frac{1 + (1 - r) \cdot (1 - x)}{2 - r + r \cdot x}.
\end{align*}
In \cite{GI:2005,KS:2008} it was shown that the absolutely continuous invariant measure $\mu_{r}$ of $T_{r}$ is given by 
\begin{align*}
h_{r}(x)
\coloneqq \frac{\mathrm{d}\mu_{r}}{\mathrm{d}\lambda}(x) =
\begin{cases}
1 & \text{if} \; r = 0,\\
\displaystyle{\frac{-r}{\ln(1-r)}\frac{1}{1-r+r \cdot x}} & \text{if} \; r \in (0, 1),\\
\displaystyle{1/x} & \text{if} \; r = 1.
\end{cases}
\end{align*}
We let $\mathcal{L}_{r}^{1}([0, 1])$ denote the Banach space of equivalence classes $[f]$ of functions, where for each representative $f \colon [0, 1] \to \mathbb{C}$ of $[f]_{r}$,
\begin{align*}
\lVert f \rVert_{r, 1} \coloneqq \int \lvert f \rvert \; \mathrm{d}\mu_{r} < +\infty,
\end{align*}
and where $f, g$ belong to the same equivalence class, if and only if, $\lVert f - g \rVert_{r, 1} = 0$.  Throughout, following convention, we write $f \in \mathcal{L}_{r}^{1}([0, 1])$ to mean a function $f \colon [0, 1] \to \mathbb{C}$ which belongs to an equivalence class of $\mathcal{L}_{r}^{1}([0, 1])$.

For $r \in [0, 1]$, the \textit{Perron-Frobenius operator} $\mathcal{P}_{r}  \colon \mathcal{L}^{1}_{0}([0, 1]) \circlearrowleft$ of $T_{r}$ is defined, for $f \in \mathcal{L}^{1}_{0}([0, 1])$, by
\begin{align}\label{eq:alternative_form_Pr}
\mathcal{P}_{r}(f) = f_{r, 0}' \cdot f \circ f_{r, 0} + f_{r, 1}' \cdot f \circ f_{r, 1}.
\end{align}
Here $f_{r, 0}'$ and $f_{r, 1}'$ denote the derivative of the contractions $f_{r, 0}$ and $f_{r, 1}$ respectively.  Note, the domain of definition of $\mathcal{P}_{r}$ can be extended to any well-defined $\mathbb{C}$-valued or $\overline{\mathbb{R}}$-valued function.  In \cite{GI:2005,KS:2008} it has been shown that $h_{r}$ is the unique fixed point function of $\mathcal{P}_{r}$, namely that $\mathcal{P}_{r}(h_{r}) = h_{r}$, and so
\begin{align*}
\mathcal{P}_{r}(f) \coloneqq \frac{\mathrm{d}\nu_{f} \circ {T_{r}}^{-1}}{\mathrm{d}\lambda},
\quad \text{where} \quad
\nu_{f}(A) \coloneqq \int \mathds{1}_{A} \cdot f \, \mathrm{d}\lambda,
\quad \text{for all Borel sets $A \subset [0, 1]$.}
\end{align*}

Two important function spaces which we will use are defined below.
\begin{enumerate}
\item The space $\mathrm{BV}(0,1)$ which is defined to be the set of right-continuous functions $f \colon [0, 1] \to \mathbb{C}$ such that the norm $\lVert f \rVert_{\mathrm{BV}} \coloneqq V_{[0, 1]}(f) + \lVert f \rVert_{\infty}$ is finite.  Here $V_{[0, 1]}(f)$ denotes the variance of $f$, see Section~\ref{sec:f_of_BV} for the definition and properties of the variance of a function, and $\lVert f \rVert_{\infty}$ denotes the supremum of $\lvert f \rvert$ and is defined by $\lVert f \rVert_{\infty} \coloneqq \sup \{ \lvert f(x) \rvert \colon x \in [0, 1] \}$.
\item The space $\mathfrak{U}_{\beta, \alpha}$ is defined for $\alpha \in (0, 1)$ and $\beta \in [0, 1]$, and where $v \colon [0, 1] \to \overline{\mathbb{R}}$ belongs to $\mathfrak{U}_{\beta, \alpha}$ if and only if
\begin{enumerate}
\item $\lim_{x \uparrow \beta} v(x) = \lim_{x \downarrow \beta} v(x) = + \infty$,
\item for any compact subset $K \subset [0, 1] \setminus \{ \beta \}$, we have that $v \cdot \mathds{1}_{K} \in \mathrm{BV}(0,1)$, where for a set $A \subset [0, 1]$ we let $\mathds{1}_{A} \colon [0, 1] \to \mathbb{R}$ denote the \textit{characteristic function on $A$}, namely
\begin{align*}
\mathds{1}_{A}(x) \coloneqq \begin{cases}
1& \text{if} \; x \in A,\\
0& \text{otherwise},
\end{cases}
\end{align*}
\item there exists a connected open neighbour $U \subset [0, 1]$  of $\beta$, under the (Euclidean) subspace topology, and two   constants $C_{1}, C_{2}$ such that $C_{1} \lvert \beta - x \rvert^{-\alpha} \leq v(x) \leq C_{2} \lvert \beta - x \rvert^{-\alpha}$, for all $x \in U$.
\end{enumerate}
\end{enumerate}
Note conditions (b) and (c) immediately imply that if $v \in \mathfrak{U}_{\beta, \alpha}$, then $v$ is improper Riemann integrable.  Moreover, without loss of generality, throughout we assume that $v$ is positive.

Define the \textit{$\omega$-limit set of $\beta \in [0,1]$} with respect to $T_{r}$ to be the set of accumulation points of the orbit $( T_{r}^{n}(\beta) )_{n \in \mathbb{N}_{0}}$ and denote it by
\begin{align*}
\Omega_{r}(\beta)\coloneqq \bigcap_{k\in \mathbb{N}_{0}}\overline{\{T_{r}^{\ell}(\beta):\ell \geq k\}}.
\end{align*}
We say that a point $x \in [0, 1]$ is \textit{pre-periodic with respect to $T_{r}$} if there exist $m \in \mathbb{N}$ and $n \in \mathbb{N}_{0}$ such that 
\begin{align}\label{eq:pre-periodic}
T_{r}^{n + k}(x) = T_{r}^{n + m + k}(x),
\end{align}
for all $k \in \mathbb{N}_{0}$.  Indeed, for $r \in [0, 1]$, we have that $1 - (3 - \sqrt{9 - 4 \cdot r})/(2 \cdot r)$ is pre-periodic with respect to $T_{r}$.  For a given pre-periodic point $x$ with respect to $T_{r}$, we define the \textit{period length} of $x$ to be the minimal $m$ such that the equality in \eqref{eq:pre-periodic} holds.

In the case when $r = 1$, as mentioned above, the map $T_{1}$ is the celebrated Farey map which encodes the continued fraction expansion algorithm.  A \textit{continued fraction expansion} of an irrational $\beta \in [0, 1]$ is denoted by $[0; a_{1}, a_{2}, \dots]$ where
\begin{align*}
\beta = \cfrac{1}{a_{1} + \cfrac{1}{a_{2} + \dots}}
\end{align*}
and $a_{n} \in \mathbb{N}$, for all $n \in \mathbb{N}$.  A \textit{continued fraction expansion} of a rational $\beta \in [0, 1]$ is denoted by $[0; a_{1}, a_{2}, \dots, a_{k}]$ where
\begin{align*}
\beta = \cfrac{1}{a_{1} + \cfrac{1}{ \dots + \cfrac{1}{a_{k}}}}
\end{align*}
and $a_{n} \in \mathbb{N}$, for all $n \in \{ 1, 2, \dots, k \}$.  If there exist $m \in \mathbb{N}_{0}$ and $n \in \mathbb{N}$ such that $a_{m + k} = a_{m + k + n + 1}$, for all $k \in \mathbb{N}$, then we write $\beta = [0; a_{1}, a_{2}, \dots, a_{m}, \overline{a_{m + 1}, a_{m+2}, \dots, a_{m + n}} ]$.

For $\beta \in [0, 1]$, we let $p_{n} = p_{n}(\beta)$ and $q_{n} = q_{n}(\beta)$ be defined recursively by
\begin{align}\label{eq:definition_p_q}
p_{-1} \coloneqq 1, \quad
q_{-1} \coloneqq 0, \quad
p_{0}  \coloneqq 0,\quad 
q_{0} \coloneqq 1, \quad
p_{n} \coloneqq a_{n}p_{n-1} + p_{n-2},
\quad \text{and} \quad
q_{n} \coloneqq a_{n} q_{n-1} + q_{n-2}.
\end{align}
Note, for $n \in \mathbb{N}$, that 
\begin{align*}
\frac{p_{n}}{q_{n}} = [0; a_{1}, a_{2}, \dots, a_{n}]
\quad \text{and} \quad
p_{n -1} \cdot q_{n} - p_{n} \cdot q_{n-1} = 1,
\end{align*}
and that if $\beta = [0; a_{1}, a_{2}, \dots, a_{n}]$ is rational then we set $a_{m} = 0$ for all $m > n$.  Given an $\alpha \in (0,1)$ we say that an irrational $\beta = [0; a_{1}, a_{2}, \dots] \in [0, 1]$ is of \textit{intermediate $\alpha$-type} if and only if there exists an $\epsilon > 0$, such that
\begin{align*}
\sum_{n = 1}^{+\infty} \sum_{k = 1}^{a_{n}} (t_{n, j})^{-2\cdot(1-\alpha) + \epsilon} < + \infty.
\end{align*}
where $s_{n, j}/t_{n, j} = [0; a_{1}, \dots, a_{n-1}, j]$ and where $s_{n, j}, t_{n, j} \in \mathbb{N}$ are co-prime.  (Using the terminology from continued fraction expansion one refers to $s_{n,j}/t_{n,j}$ as an \textit{intermediate approximant to $\beta$}.)  We also note the following.
\begin{enumerate}
\item If $\beta$ is pre-periodic, or more generally, if the continued fraction entries $a_{i}$ of $\beta$ are bounded, then $\beta$ is of intermediate $\alpha$-type, for all $\alpha \in (0, 1)$.  
\item If $\alpha < 1/2$, then every irrational $\beta$, is of intermediate $\alpha$-type.
\item It follows from the results of \cite{KS:2012} that
\begin{align*}
\qquad
\mathrm{dim}_{\mathcal{H}} (\{ \beta \in [0, 1] \colon \beta \; \text{is of intermediate $\alpha$-type for all} \; \alpha \in (0, 1) \}) = 1.
\end{align*}
Here and throughout we will denote the Hausdorff dimension of a set $A \subset \mathbb{R}$ by $\mathrm{dim}_{\mathcal{H}}(A)$, see \cite{F:2014} for the definition and further details on the Hausdorff dimension of a set.
\end{enumerate}
For more on continued fraction expansions we refer the reader to \cite{DK:2002,YK:1997}.

\section{Main results}\label{sec:main}

\subsection{The case $\mathbf{r \in [0, 1)}$}

\begin{theorem}\label{thm:Convergence_Equilibrum}
For $r \in [0, 1)$, if $\alpha \in (0, 1)$ and $\beta \in [0, 1]$, then, for each $v \in \mathfrak{U}_{\beta, \alpha}$, we have that
\begin{align}\label{eq:convergence}
\lim_{n \to \infty} \mathcal{P}_{r}^{n}(v) 
=
\int v \, \mathrm{d}\lambda \cdot h_{r},
\end{align}
uniformly on compact subsets of $[0, 1] \setminus \Omega_{r}(\beta)$ and point-wise outside a set with Hausdorff dimension equal to zero.  If $\beta \in [0, 1]$ is pre-periodic with respect to $T_{r}$ and has period length strictly greater than one, then on the finite set $\Omega_{r}(\beta)$ we have that
\begin{align*}
\qquad \quad
\liminf_{n \to +\infty} \mathcal{P}_{r}^{n}(v)
= \int v \, \mathrm{d}\lambda \cdot h_{r}
\quad \text{and} \quad
\limsup_{n \to +\infty} \mathcal{P}_{r}^{n}(v)
= +\infty.
\end{align*}
In the case that $\beta \in [0, 1]$ is pre-periodic with respect to $T_{r}$ and has period length equal to one then on the singleton $\Omega_{r}(\beta)$ we have that the limit in \eqref{eq:convergence} is equal to $+\infty$.
\end{theorem}

\begin{remark}\label{rmk:generalisation}
We believe that Theorem~\ref{thm:Convergence_Equilibrum} holds for more general of systems, namely for any piecewise $C^{1 + \epsilon}$ Markov interval map $T \colon [0, 1] \circlearrowleft$.  The proof of such a result should follow in the same manner as those set out below.
\end{remark}

\subsection{The case $\mathbf{r =1}$}

\begin{theorem}\label{thm:Infinite_1}
If $\alpha \in (0, 1)$ and if $\beta \in (0, 1]$ is either rational or irrational of intermediate $\alpha$-type, then, for each $v \in \mathfrak{U}_{\beta, \alpha}$,  we have that
\begin{align}\label{eq:convergence_r=1}
\lim_{n \to \infty} \, \ln(n) \cdot \mathcal{P}_{1}^{n}(v) 
=
\int_{[0, 1]} v \, \mathrm{d}\lambda \cdot h_{1},
\end{align}
uniformly on compact subsets of $(0, 1) \setminus \Omega_{1}(\beta)$ and point-wise outside a set with Hausdorff dimension equal to zero.  If $\beta \in (0, 1]$ is pre-periodic with respect to $T_{1}$ and has period length strictly greater than one, then on the finite set $\Omega_{1}(\beta)$ we have that
\begin{align}\label{eq:non_convergence_r=1}
\qquad \quad
\liminf_{n \to +\infty} \ln(n) \cdot \mathcal{P}_{1}^{n}(v)
= \int v \, \mathrm{d}\lambda \cdot h_{1}
\quad \text{and} \quad
\limsup_{n \to +\infty} \ln(n) \cdot \mathcal{P}_{1}^{n}(v)
= +\infty.
\end{align}
In the case that $\beta \in (0, 1]$ is pre-periodic with respect to $T_{1}$ and has period length equal to one then on the singleton $\Omega_{1}(\beta)$ we have that the limit in \eqref{eq:convergence_r=1} is equal to $+\infty$.
\end{theorem}

\begin{remark}
The $\ln(n)$ term in \eqref{eq:convergence_r=1} and \eqref{eq:non_convergence_r=1} is known as the \textit{wandering rate} of the Farey map $T_{1}$.  Indeed this term is well defined for any interval map $T \colon [0, 1] \circlearrowleft$ and for the maps we are concerned with it is given by
\begin{align*}
w_{n}(T_{r}) \coloneqq \mu_{r} \left( \bigcup_{k=0}^{n-1} T_{r}^{-k}([1/2, 1])\right).
\end{align*}
Indeed from this definition one sees that for $r \in [0, 1)$ we have that $w_{n}(T_{r}) \sim 1$ and for $r = 1$ we have that $w_{n}(T_{r}) \sim \ln(n)$.
\end{remark}

\begin{remark}
We highlight an interesting difference between Theorems~\ref{thm:Convergence_Equilibrum} and~\ref{thm:Infinite_1}, which is a result of the Farey map having an indifference fixed point at zero.  In the case that $r \in [0, 1)$, $\alpha \in (0, 1)$, $\beta$ is an $r$-rational (see Section~\ref{sec:preliminaries}) and $v \in \mathfrak{U}_{\beta, \alpha}$, we have that
\begin{align*}
\lim_{n \to \infty} \, \mathcal{P}_{r}^{n}(v)(0) = +\infty
\end{align*}
whereas, for $r = 1$, $\alpha \in (0, 1)$, $\beta$ is a rational number and $v \in \mathfrak{U}_{\beta, \alpha}$, we have that
\begin{align*}
\lim_{n \to \infty} \ln(n) \cdot \mathcal{P}_{1}^{n}(v)(0) = 0.
\end{align*}
(Note that the points $0, 1/2$ and $1$ are $r$-rationals for all $r \in [0, 1]$.)
\end{remark}

\begin{remark}
In the case that one replaces the norm $\lVert \cdot \rVert_{\infty}$ by the essential supremum norm in the definition of $\mathrm{BV}(0, 1)$, and hence $\mathfrak{U}_{\beta, \alpha}$, the limit in \eqref{eq:convergence_r=1} holds uniformly Lebesgue almost everywhere on compact subsets of $(0, 1)\setminus \Omega_{1}(\beta)$ and point-wise Lebesgue almost everywhere on $(0, 1)$.
\end{remark}

In the following theorem, for the observable $v_{\beta, \alpha}(x) = \lvert \beta - x \rvert^{-\alpha}$, we demonstrate that on the set of exceptional points where the equality in \eqref{eq:convergence_r=1} does not hold, the values of the limit inferior and limit superior depend on the diophantine properties of $\beta$.

\begin{theorem}\label{thm:Infinite_2}
\mbox{ }
\begin{enumerate}[label=(\alph*)]
\item\label{3.3.a} There exist non-periodic $\beta$ and $\varrho \in (0, 1]$ both with bounded continued fraction entries but such that, on the one hand, if $\alpha \in (0, 1)$, then on $\Omega_{1}(\beta)$
\begin{align*}
\lim_{n \to +\infty} \ln(n) \cdot \mathcal{P}_{1}^{n}(v_{\beta, \alpha})
=
\int v_{\beta, \alpha} \, \mathrm{d}\lambda \cdot h_{1},
\end{align*}
and on the other hand, if $\alpha  \in (0, 1/2)$, then on $\Omega_{1}(\varrho)$
\begin{align*}
\lim_{n \to \infty} \ln(n) \cdot \mathcal{P}_{1}^{n}(v_{\varrho, \alpha})
=
\int v_{\varrho, \alpha} \, \mathrm{d}\lambda \cdot h_{1};
\end{align*}
otherwise, if $\alpha \in (1/2, 1)$, then on $\Omega_{1}(\varrho)$
\begin{align*}
\qquad \quad
\liminf_{n \to +\infty} \ln(n) \cdot \mathcal{P}_{1}^{n}(v_{\varrho, \alpha})
= \int v_{\varrho, \alpha} \, \mathrm{d}\lambda \cdot h_{1}
\quad \text{and} \quad
\limsup_{n \to +\infty} \ln(n) \cdot \mathcal{P}_{1}^{n}(v_{\varrho, \alpha})
= +\infty.
\end{align*}

\item\label{3.3.b}  Let $\alpha \in (0, 1)$ and let $\beta = [0; a_{1}, a_{2}, \dots] \in (0, 1]$ be of intermediate $\alpha$-type such that 
\begin{align*}
\displaystyle{\lim_{n \to +\infty} a_{n} = + \infty}.
\end{align*}
Fix $k \in \mathbb{N}$ and let $l = l(k) \coloneqq \min \{ i \in \mathbb{N} \colon a_{m} \geq k \; \text{for all} \; m \geq i  \}$.  For all $j \geq l$, set $n_{k, j} \in \mathbb{N}$ to be the unique integer satisfying $T_{1}^{n_{k, j}}(\beta) = [0; k, a_{j+1}, a_{j+2}, \dots]$ and set
\begin{align*}
\mathscr{S}_{k, j} \coloneqq \frac{(a_{j + 1})^{\alpha} \cdot \ln(n_{k, j})}{(q_{j})^{2\cdot(1-\alpha)}},
\end{align*}
where $q_{n}$ is as defined in \eqref{eq:definition_p_q}.  If $\displaystyle{\limsup_{j \to \infty}  \mathscr{S}_{k, j} = 0}$, then
\begin{align*}
\lim_{n \to +\infty} \ln(n) \cdot \mathcal{P}_{1}^{n}(v_{\beta, \alpha})(1/k)
=
\int v_{\beta, \alpha} \, \mathrm{d}\lambda \cdot h_{1};
\end{align*}
otherwise,
\begin{align*}
\qquad \quad
\liminf_{n \to +\infty} \ln(n) \cdot \mathcal{P}_{1}^{n}(v_{\beta, \alpha})(1/k)
= \int v_{\beta, \alpha} \, \mathrm{d}\lambda \cdot h_{1}
\;\; \text{and} \;\;
\limsup_{n \to +\infty} \ln(n) \cdot \mathcal{P}_{1}^{n}(v_{\beta, \alpha})(1/k)
> \int v_{\beta, \alpha} \, \mathrm{d}\lambda.
\end{align*}
(Note that in this case $\Omega_{1}(\beta) = \{ 1/n \colon n \in \mathbb{N} \} \cup \{ 0 \}$.)
\end{enumerate}
\end{theorem}

\section{Preliminaries}\label{sec:preliminaries}

We let $\Sigma \coloneqq \{ 0, 1 \}$, $\Sigma^{n} \coloneqq \{ 0, 1 \}^{n}$, for $n \in \mathbb{N}$, and let $\Sigma^{\mathbb{N}}$ denote the set of all infinite words over the alphabet $\Sigma$.  For $\beta \in [0, 1]$ we let $\omega_{r}(\beta)$ denote the infinite word $(\omega_{r, 1}(\beta), \omega_{r, 2}(\beta) \dots) \in \Sigma^{\mathbb{N}}$, where 
\begin{align*}
\omega_{r, n}(\beta) \coloneqq
\begin{cases}
0 & \text{if} \; T_{r}^{n-1}(\beta) \leq 1/2,\\
1 & \text{otherwise}.
\end{cases}
\end{align*}
Unless otherwise stated, let $n \in \mathbb{N}$ be fixed.  For $\omega = (\omega_{1}, \omega_{2}, \dots) \in \Sigma^{\mathbb{N}}$, we set $\omega\vert_{n} \coloneqq (\omega_{1}, \dots, \omega_{n}) \in \Sigma^{n}$ and, for $\varphi = (\varphi_{1}, \varphi_{2}, \dots, \varphi_{n}) \in \Sigma^{n}$, we set
\begin{align*}
f_{r, \varphi} \coloneqq f_{r, \varphi_{1}} \circ \cdots \circ f_{r, \varphi_{n}}
\quad \text{and} \quad
[\varphi]_{r} \coloneqq f_{r, \varphi}([0, 1]).
\end{align*}
The set $[\varphi]_{r}$ is referred to as a \textit{cylinder set of length $n$} with respect to $T_{r}$.  We let $\omega_{r}^{\pm}(\beta)\lvert_{n} \in \Sigma^{n}$ denote unique finite words such that
\begin{align*}
[\omega_{r}^{+}(\beta)\lvert_{n}]_{r} \cap [\omega_{r}(\beta)\lvert_{n}]_{r} \neq \emptyset,
\quad
[\omega_{r}^{-}(\beta)\lvert_{n}]_{r} \cap [\omega_{r}(\beta)\lvert_{n}]_{r} \neq \emptyset
\end{align*}
and such that either one of the following sets of inequalities hold,
\begin{align*}
f_{\omega_{r}^{-}(\beta)\lvert_{n}}(x) \leq f_{\omega_{r}(\beta)\lvert_{n}}(x) < f_{\omega_{r}^{+}(\beta)\lvert_{n}}(x)
\quad \text{or} \quad
f_{\omega_{r}^{-}(\beta)\lvert_{n}}(x) < f_{\omega_{r}(\beta)\lvert_{n}}(x) \leq f_{\omega_{r}^{+}(\beta)\lvert_{n}}(x),
\end{align*}
for all $x \in (0, 1)$.  Note that in the case when there exists $\omega \in \Sigma^{m}$, for some $m \in \mathbb{N}$, such that either $f_{r, \omega}(0) = \beta$ or $f_{r, \omega}(1) = \beta$, then it can occur that $\omega_{r}^{+}(\beta)\vert_{m} = \omega_{r}(\beta)\vert_{m}$ or that $\omega_{r}^{-}(\beta)\vert_{m} = \omega_{r}(\beta)\vert_{m}$.  We call such points \textit{$r$-rationals}.  (Note, if $r = 1$, then the set of $r$-rationals is precisely the set of rational numbers in the closed unit interval $[0, 1]$.)  For ease of notation, we set
\begin{align}\label{eq_tail_set}
\mathfrak{W}_{r, n}(\beta) \coloneqq \{ \omega_{r}^{-}(\beta)\vert_{n}, \omega_{r}(\beta)\vert_{n}, \omega_{r}^{+}(\beta)\vert_{n} \}
\quad \text{and} \quad
[W_{r, n}(\beta)] =  [\omega_{r}^{-}(\beta)\vert_{n}]_{r} \cup [\omega_{r}(\beta)\vert_{n}]_{r} \cup [\omega_{r}^{+}(\beta)\vert_{n}]_{r}.
\end{align}

\begin{lemma}\label{lem:adjacent_cylinders}
Let $r \in [0, 1]$ and $n \in \mathbb{N}$ be fixed.  If $\omega = (\omega_{1}, \omega_{2}, \dots, \omega_{n})$ and $\nu = (\nu_{1}, \nu_{2}, \dots, \nu_{n})$ denote two distinct elements of $\Sigma^{n}$, with $[\omega]_{r} \cap [\nu]_{r} \neq \emptyset$, then there exists a unique $i \in \{ 1, 2, \dots, n \}$ such that $\omega_{i} \neq \nu_{i}$ and $\omega_{j} = \nu_{j}$ for all $j \in \{ 1, 2, \dots, n \} \setminus \{ i \}$.
\end{lemma}

\begin{proof}
For $n = 1$ we have that $[(0)]_{r} = f_{r, 0}([0, 1]) = [0, 1/2]$ and $[(1)]_{r} = f_{r, 1}([0, 1]) = [1/2, 1]$.  We now proceed by induction on $n$.  Suppose the statement is true for some $n \in \mathbb{N}$.  Let $\omega = (\omega_{1}, \omega_{2}, \dots, \omega_{n+1})$ and $\nu = (\nu_{1}, \nu_{2}, \dots, \nu_{n+1})$ denote two distinct elements of $\Sigma^{n+1}$, with $[\omega]_{r} \cap [\nu]_{r} \neq \emptyset$.  We have two cases to consider, namely, if there exists an $\xi \in \Sigma^{n}$ such that $[\omega]_{r} \cup [\nu]_{r} = [\xi]_{r}$, or not.

In the case that there exists an $\xi = (\xi_{1}, \dots, \xi_{n}) \in \Sigma^{n}$ with $[\omega]_{r} \cap [\nu]_{r} = [\xi]_{r}$, then, by construction, either
\begin{enumerate}
\item $\omega = (\xi_{1}, \xi_{2}, \dots, \xi_{n}, 0)$ and $\nu = (\xi_{1}, \xi_{2}, \dots, \xi_{n}, 1)$, or
\item $\omega = (\xi_{1}, \xi_{2}, \dots, \xi_{n}, 1)$ and $\nu = (\xi_{1}, \xi_{2}, \dots, \xi_{n}, 0)$,
\end{enumerate}
in which case the result follows.

In the case that there does not exist an $\xi \in \Sigma^{n}$ with $[\omega]_{r} \cap [\nu]_{r} = [\xi]_{r}$, then, by construction, there exist $\xi = (\xi_{1}, \xi_{2}, \dots, \xi_{n}), \eta = (\eta_{1}, \eta_{2}, \dots, \eta_{n}) \in \Sigma^{n}$ such that $[\xi]_{r} \cap [\eta]_{r} \neq \emptyset$, $[\omega]_{r} \subset [\xi]_{r}$ and $[\omega]_{r} \subset [\eta]_{r}$. Therefore, by the inductive hypothesis, we have that either $f_{r, \xi}$ is order preserving and $f_{r, \eta}$ is order reversing, or $f_{r, \xi}$ is order reversing and $f_{r, \eta}$ is order preserving.  Assuming the former of these two cases, by construction we have that $\omega = (\xi_{1}, \dots, \xi_{n}, 1)$ and $\nu = (\eta_{1}, \dots, \eta_{n}, 1)$, in which case the result follows.  In the remaining case, namely that $f_{r, \xi}$ is order reversing and $f_{r, \eta}$ is order preserving, by construction we have that $\omega = (\xi_{1}, \dots, \xi_{n}, 0)$ and $\nu = (\eta_{1}, \dots, \eta_{n}, 0)$, which concludes the proof.
\end{proof}

\begin{definition}\label{def:tail}
Given $r \in [0, 1]$, $\alpha \in (0, 1)$ and $\beta \in [0, 1]$, we define the \textit{$r$-tail of the observable $v_{\beta, \alpha} \colon x \mapsto \lvert x - \beta \rvert^{-\alpha}$} by 
\begin{align}\label{eq:vnr}
v_{n, r} = v_{\beta, \alpha, n, r} \coloneqq \mathcal{P}_{r}^{n}(v_{\beta, \alpha} \cdot \mathds{1}_{[W_{r, n}(\beta)]}) = 
\begin{cases}
\displaystyle{\sum_{\omega \in \mathfrak{W}_{r, n}(\beta)} \lvert f_{r, \omega}'(x) \rvert \cdot v_{\beta, \alpha} \circ f_{r, \omega}} & \text{if} \; r \in [0, 1),\\
\lvert f_{r, \omega_{1}(\beta)\vert_{n}}'(x) \rvert \cdot v_{\beta, \alpha} \circ f_{r, \omega_{1}(\beta)\vert_{n}} & \text{if} \; r = 1.
\end{cases}
\end{align}
Further, for $r \in [0, 1]$, $\alpha \in (0, 1)$, $\beta \in [0, 1]$, $n \in \mathbb{N}$ and $\eta > 0$ set
\begin{align*}
A_{n, r, \eta} \coloneqq 
\begin{cases}
\{ x \in [0, 1] \colon v_{n, r}(x) > \eta \} & \text{if} \; r \in [0, 1),\\
\{ x \in [0, 1] \colon \ln(n) \cdot v_{n, r}(x) > \eta \} & \text{if} \; r =1.
\end{cases}
\end{align*}
\end{definition}

\subsection{Functions of bounded variation}\label{sec:f_of_BV}
\mbox{ }

Let $[a, b]$ be a compact interval in $\mathbb{R}$.  The variation of a function $f \colon [a, b] \to \mathbb{C}$ is defined to by
\begin{align*}
V_{[a, b]}(f) \coloneqq \sup_{P} \left\{ \sum_{k = 1}^{n} \lvert f(x_{k}) - f(x_{k -1}) \rvert \right\}.
\end{align*}
Here the supremum is taken over finite partitions $P \coloneqq \{ I_{i} = [x_{i -1}, x_{i}] \colon i \in \{ 1, 2, \dots, n\} \}$, where $a \coloneqq x_{0} < x_{1} < \dots < x_{n - 1} < x_{n} \coloneqq b$, is a chain of points belonging to $[a, b]$, for some $n \in \mathbb{N}$.

Below we state various properties of functions of bounded variation, which we will require in the sequel: Proposition~\ref{Prop:propertiesBV} is concerned with $\mathbb{R}$-valued functions and Proposition \ref{Prop:propertiesBV-Complex} is concerned with $\mathbb{C}$-valued functions.

\begin{proposition}[{\cite[Chapter 2]{BG:1997}}]\label{Prop:propertiesBV}
Let $f, g \in \mathcal{L}^{1}_{\lambda}([a, b])$ be two $\mathbb{R}$-valued functions of bounded variation.
\begin{enumerate}
\item \label{BV:norm} The supremum norm $\lVert f \rVert_{\infty}$ of $f$ is finite.
\item \label{BV:value} For $x \in [a, b]$ we have that $\lvert f(x) \rvert \leq V_{[a, b]}(f) + \lVert f \rVert_{0, 1}/(b - a)$.
\item \label{BV:Lin+prod} The sum, difference and product of two functions of bounded variation is again of bounded variation, and moreover,
\begin{align*}
\qquad\quad
V_{[a, b]}(f \pm g) \leq V_{[a, b]}(f) + V_{[a, b]}(g)
\quad \text{and} \quad
V_{[a, b]}(f \cdot g) \leq V_{[a, b]}(g) \cdot \lVert f \rVert_{\infty} + V_{[a, b]}(f) \cdot \lVert g \rVert_{\infty}.
\end{align*}
\item \label{BV:Split} If $c \in (a, b)$, then $f$ is of bounded variation on the intervals $[a, c]$ and $[c, d]$ and moreover, $V_{[a, b]}(f) = V_{[a, c]}(f) + V_{[c, b]}(f)$.
\item \label{BV:difference} The function $f$ (and $g$) has a representation as the difference of two non-decreasing functions.
\item \label{BV:diff} A function of bounded variation is differentiable Lebesgue almost everywhere.
\item \label{BV:alt} Letting $\tau = \tau_{[a, b]} \coloneqq \{ \psi \in C^{1}([a, b]) \colon \lVert \psi \rVert_{\infty} \leq 1 \; \textup{and} \; \psi(0) = \psi(1) = 0 \}$, we have that
\begin{align*}
V_{[a, b]}(f) = \sup_{\psi \in \tau} \int f \cdot \psi' \mathrm{d}\lambda.
\end{align*}
\end{enumerate}
\end{proposition}

\begin{proposition}[{\cite[p.\ 74 f.]{F:2001}}]\label{Prop:propertiesBV-Complex}
Let $f, g \in \mathcal{L}^{1}_{\lambda}([a, b])$ be two $\mathbb{C}$-valued functions of bounded variation.
\begin{enumerate}
\item \label{BV:Complex1} The supremum norm $\lVert f \rVert_{\infty}$ of $f$ is finite.
\item \label{BV:Complex2} The sum, difference and product of two functions of bounded variation is of bounded variation.
\item \label{BV:Complex3} A $\mathbb{C}$-valued function is of bounded variation, if and only if its real and imaginary parts are of bounded variation.  In particular, if $f = \mathfrak{Re}(f) + i \mathfrak{Im}(f)$, then
\begin{align*}
\max\{ V_{[a, b]}(\mathfrak{Re}(f)), V_{[a, b]}(\mathfrak{Im}(f)) \} \leq V_{[a, b]}(f) \leq V_{[a, b]}(\mathfrak{Re}(f)) + V_{[a, b]}(\mathfrak{Im}(f))
\end{align*}
and hence $\max\{ \lVert \mathfrak{Re}(f) \rVert_{\mathrm{BV}}, \lVert \mathfrak{Im}(f) \rVert_{\mathrm{BV}} \} \leq \lVert f \rVert_{\mathrm{BV}}$.
\end{enumerate}
\end{proposition}

The next proposition follows from \cite[p.\ 74]{F:2001} together with a standard continuity argument.

\begin{proposition}[{\cite[p.\ 74]{F:2001}}]\label{prop:Banach}
The space $\mathrm{BV}(0, 1)$ equipped with the norm $\lVert \cdot \rVert_{\mathrm{BV}}$ is a Banach space.
\end{proposition}

For further details concerning functions of bounded variation see \cite[Chapter 2.3]{BG:1997}, \cite[Section 224]{F:2001} and \cite[Section 2.3]{KL:2005}.

\subsection{Auxiliary results for the case $\mathbf{r \in [0, 1)}$}\label{sec:preliminaries_r_0_1}

\subsubsection{Bounded distortion}

\begin{lemma}[{\cite[Lemma 3.2]{KK:2012} Bounded Distortion}]\label{lem:BD}
Let $r \in [0, 1)$ be fixed.  There exists a sequence $(\varrho_{n})_{n \in \mathbb{N}_{0}}$, dependent on $r$, with $\varrho_{n} > 0$ for each $n \in \mathbb{N}_{0}$ and $\lim_{n \to +\infty} \varrho_{n} = 1$, such that, for all $m, n \in \mathbb{N}_{0}$, $\omega \in \Sigma^{m}$, $\varphi \in \Sigma^{n}$ and $x, y \in [\omega]_{r}$, we have that 
\begin{align*}
\varrho_{m}^{-1} \leq \left\lvert \frac{f_{r, \varphi}'(x)}{f_{r, \varphi}'(y)} \right\rvert \leq \varrho_{m}.
\end{align*}
\end{lemma}

(Here $\Sigma^{0}$ denotes the set containing the empty set and $f_{r, \emptyset}$ denotes the identity function $[0,1] \ni x \mapsto x$.)

\begin{lemma}\label{lem:BD_neighbouring_cylinders}
Let $n \in \mathbb{N}$ be fixed.  If $\omega = (\omega_{1}, \omega_{2}, \dots, \omega_{n})$ and $\nu = (\nu_{1}, \nu_{2}, \dots, \nu_{n})$ denote two distinct elements of $\Sigma^{n}$, with $[\omega] \cap [\nu] \neq \emptyset$, then there exists a positive constant $K$ such that, for all $x, y \in [0, 1]$,
\begin{align*}
K^{-1} \leq \left\lvert \frac{f'_{r, \omega}(x)}{f'_{r, \nu}(y)} \right\rvert \leq K.
\end{align*}
\end{lemma}

\begin{proof}
This is a consequence of the chain rule and Lemmata~\ref{lem:adjacent_cylinders} and~\ref{lem:BD}.
\end{proof}

\subsubsection{Classical results on convergence to equilibrium}

\begin{theorem}[{\cite{B:2000,C:1996,K:1984,R:1983}}]\label{thm:Finite_Ergodic}
For $r \in [0, 1)$ there exist constants $M = M(r) > 0$ and $p = p(r) \in (0, 1)$ such that
\begin{align*}
\left\lVert \mathcal{P}^{n}_{r}(f) - \int f \, \mathrm{d}\lambda \cdot h_{r} \right\rVert_{\mathrm{BV}} \leq M \cdot p^{n} \cdot \lVert f \rVert_{\mathrm{BV}},
\end{align*}
 for all $f \in \mathrm{BV}(0, 1)$.
\end{theorem}

\begin{lemma}\label{lem:bv_part}
For $r \in [0, 1)$, $\alpha \in (0, 1)$, $\beta \in [0, 1]$ and $v \in \mathfrak{U}_{\beta, \alpha}$ we have that
\begin{align*}
\lim_{n \to \infty} \mathcal{P}_{r}^{n}(v \cdot \mathds{1}_{[0, 1] \setminus [W_{r, n}(\beta)]}) = \int v \, \mathrm{d}\lambda \cdot h_{r},
\end{align*}
uniformly on $[0, 1]$.
\end{lemma}

\begin{proof}
Let $N \in \mathbb{N}$ be fixed.  By Theorem~\ref{thm:Finite_Ergodic}, since $v \cdot \mathds{1}_{[0, 1] \setminus [W_{r}(\beta)\vert_{N}])} \in \mathrm{BV}(0, 1)$, we have that
\begin{align*}
\lim_{n \to +\infty} \mathcal{P}^{n}_{r}(v \cdot \mathds{1}_{[0, 1] \setminus [W_{r}(\beta)\vert_{N}])})
= \int v \cdot \mathds{1}_{[0, 1] \setminus [W_{r}(\beta)\vert_{N}])} \, \mathrm{d}\lambda \cdot h_{r}
\end{align*}
uniformly on $[0, 1]$.  We will shortly show, with the aid of Lemmeta~\ref{lem:adjacent_cylinders} and~\ref{lem:BD}, that, uniformly on $[0, 1]$, there exists a positive constant $K \in \mathbb{R}$ such that for all $x \in [0, 1]$
\begin{align}\label{eq:middle_BV}
\lim_{n \to +\infty} \mathcal{P}^{n}_{r}(v \cdot \mathds{1}_{[W_{r}(\beta)\vert_{N}] \setminus [W_{r}(\beta)\vert_{n}])})(x)
\leq K \sum_{k = N}^{+\infty} (2-r)^{-k(1 - \alpha)}.
\end{align}
As $v$ is improper Riemann integrable and as $\lim_{N \to +\infty} \lambda([\omega_{r}(\beta)\vert_{N}]) = 0$, we have that  
\begin{align*}
\lim_{N \to +\infty} \int v \cdot \mathds{1}_{[0, 1] \setminus [W_{r}(\beta)\vert_{N}])} \, \mathrm{d}\lambda = \int v \, \mathrm{d}\lambda
\end{align*}
and by the properties of geometric series we have that
\begin{align*}
\lim_{N \to +\infty} \sum_{k = N}^{+\infty} (2-r)^{-k} = 0.
\end{align*}
Thus assuming the inequalities given in \eqref{eq:middle_BV}, since $\mathcal{P}_{r}$ is a positive linear operator and since $N$ was chosen arbitrarily, the result follows.

We now show the inequalities stated in \eqref{eq:middle_BV}.  Let $U \subset [0, 1]$ be an open set and let $C_{2}$ be a constant such that Conditions (c) in the definition of $\mathfrak{U}_{\beta, \alpha}$ is satisfied.  Let $n > N \geq 2$ with $[W_{r}(\beta)\vert_{N}] \subseteq U$ be fixed.  For all $x \in [0, 1]$, we have that
\begin{align}\label{eq:derivative_bounds}
(2-r)/4 \leq f_{r, 0}'(x), f_{r, 1}'(x) \leq 1/(2-r).
\end{align}
This in tandem with Lemmata~\ref{lem:BD} and \ref{lem:BD_neighbouring_cylinders} and the mean value theorem, gives that there exists a positive constant $\varrho \in \mathbb{R}$ such that the following chain of inequalities hold, for all $x \in [0, 1]$.
\begin{align*}
\mathcal{P}^{n}_{r}(v \cdot \mathds{1}_{[W_{r}(\beta)\vert_{N}] \setminus [W_{r}(\beta)\vert_{n}])})(x)
\;\; &=  \sum_{\substack{\omega \in \Sigma^{n} \setminus \mathfrak{W}_{r, n}(\beta)\\ [\omega] \subseteq [W_{r}(\beta)\vert_{N}]}} \lvert f_{r, \omega}'(x) \rvert \cdot v \circ f_{r, \omega}(x)\\
&\leq \sum_{\substack{\omega \in \Sigma^{n} \setminus \mathfrak{W}_{r, n}(\beta)\\ [\omega] \subseteq [W_{r}(\beta)\vert_{N}]}}  \varrho \cdot \lambda([\omega]) \cdot \sup \{ v(y) \colon y \in [\omega] \}\\
&\leq  \sum_{k = N+1}^{n} \sum_{\substack{\omega \in \Sigma^{k} \setminus \mathfrak{W}_{r, k}(\beta)\\ [\omega] \subseteq [W_{r}(\beta)\vert_{k-1}]}}  \varrho \cdot \lambda([\omega]) \cdot \sup \{ v(y) \colon y \in [\omega] \}\\
&\leq \sum_{k = N+1}^{n} \sum_{\substack{\omega \in \Sigma^{k} \setminus \mathfrak{W}_{r, k}(\beta)\\ [\omega] \subseteq [W_{r}(\beta)\vert_{k-1}]}} \varrho \cdot C_{2} \cdot \lambda([\omega]) \cdot \sup \{ \lvert y - \beta \rvert^{-\alpha} \colon y \in [\omega] \}\\
&\leq \sum_{k = N+1}^{n} 
2 \cdot \varrho^{2} \cdot C_{2} \cdot \left( \frac{4^{\alpha} \cdot \lambda([\omega_{r}^{-}(\beta)\vert_{k-1}])^{1 - \alpha}}{(2-r)^{1 + \alpha}}
+
\frac{4^{\alpha} \cdot \lambda([\omega_{r}^{+}(\beta)\vert_{k-1}])^{1-\alpha}}{(2-r)^{1+ \alpha}} \right)\\
&\leq \sum_{k = N+1}^{n} 
 \varrho^{2} \cdot C_{2}  \cdot 4^{1 + \alpha} (2-r)^{-(1+ \alpha) - (k - 1)(1-\alpha)}
\end{align*}
This completes the proof.
\end{proof}

\begin{remark}
In the case when one is in the situation of Remark~\ref{rmk:generalisation}, that is when one considers a piecewise $C^{1 + \epsilon}$ Markov interval map $T \colon [0, 1] \circlearrowleft$, a similar result to Lemma~\ref{lem:bv_part} holds true.  Specifically, one can show that for an compact interval $[a ,b]$ of the open interval $(0, 1)$, one has that
\begin{align}\label{eq:generalisation_bv_part}
\lim_{n \to \infty} \mathcal{P}^{n}(v \cdot \mathds{1}_{[0, 1] \setminus [W_{r, n}(\beta)]}) = \int v \, \mathrm{d}\lambda \cdot h_{r},
\end{align}
uniformly on $[a, b]$.  (Here $\mathcal{P}$ denotes the Perron-Frobenius operator of $T$.)  One approaches this by first showing the results for the end points of $a$ and  $b$.  This is obtained by a similar arguments to those presented above, however, instead of using Lemma~\ref{lem:BD_neighbouring_cylinders}, one uses the observation that there exists a positive constant $K$ such that
\begin{align*}
K^{-1} \cdot \min \{ a, 1 - a \} \cdot \lvert g_{n}(0) - g_{n}(1) \rvert \leq
\lvert g_{n}(a) - g_{n}(0) \rvert,
\lvert g_{n}(a) - g_{n}(1) \rvert
\leq
K \cdot \max \{ a, 1 - a \} \cdot \lvert g_{n}(0) - g_{n}(1) \rvert,
\end{align*}
where $g_{n}$ denotes an inverse branch of $T^{n}$.  This follows from an application of the principle of bounded variation and the chain rule.  The result stated in \eqref{eq:generalisation_bv_part} will then follow for all $z \in [a, b]$ by monotonicity, and thus the convergence at $z$ only depends on $a$ and $b$, yielding uniform convergence on the interval $[a, b]$.
\end{remark}

\subsubsection{Convergence of the $r$-tail}

\begin{lemma}\label{lem:tail_r_neq_1}
For $r \in [0, 1)$, $\alpha \in (0, 1)$, $\beta \in [0, 1]$, $n \in \mathbb{N}$ and $\eta > 0$, we have that
\begin{align*}
\mathrm{dim}_{\mathcal{H}}\left(\limsup_{n \to +\infty} A_{n, r, \eta}\right) = 0,
\end{align*}
where $A_{n, r, \eta}$ is as defined in Definition~\ref{def:tail}.
\end{lemma}

\begin{proof}
Set $z = T_{r}^{n}(\beta)$ and observe that $z$ is the unique real number in $[0, 1]$ with $f_{r, \omega_{r}(\beta)\vert_{n}}(z) = \beta$.  By the mean value theorem there exists $u \in (0, 1)$ such that
\begin{align*}
\lvert \beta - f_{r, \omega_{r}(\beta)\vert_{n}}(x)\rvert
= \lvert f_{r, \omega_{r}(\beta)\vert_{n}}(z) - f_{r, \omega_{r}(\beta)\vert_{n}}(x)\rvert
= \lvert x - z \rvert \cdot \lvert f_{r, \omega_{r}(\beta)\vert_{n}}'(u)\rvert
= \lvert x - T_{r}^{n}(\beta) \rvert \cdot \lvert f_{r, \omega_{r}(\beta)\vert_{n}}'(u)\rvert.
\end{align*}
Further, by construction, we have that $\lvert \beta - f_{r, \omega_{r}^{\pm}(\beta)\vert_{n}}(x)\rvert \geq \lvert \beta - f_{r, \omega_{r}(\beta)\vert_{n}}(x)\rvert$.  This in tandem with \eqref{eq:derivative_bounds} and Lemmata~\ref{lem:BD} and~\ref{lem:BD_neighbouring_cylinders}, yields the following set inclusions.
\begin{align*}
A_{n, r, \eta}
= \{ x \in [0, 1] \colon v_{n, r}(x) > \eta \}
&=  \left\{ x \in [0, 1] \colon \textstyle{\sum_{\omega \in \mathfrak{W}_{r, n}(\beta)} \lvert f_{r, \omega}'(x) \rvert \cdot v_{\beta, \alpha} \circ f_{r, \omega}} > \eta \right\}\\
&=  \left\{ x \in [0, 1] \colon \textstyle{\sum_{\omega \in \mathfrak{W}_{r, n}(\beta)} \lvert f_{r, \omega}'(x) \rvert \cdot \lvert x - T_{r}^{n}(\beta) \rvert^{-\alpha} \cdot \lvert f_{r, \omega_{r}(\beta)\vert_{n}}'(u)\rvert^{-\alpha}} > \eta \right\}\\
&\subseteq  \left\{ x \in [0, 1] \colon \lvert x - T_{r}^{n}(\beta) \rvert < (2-r)^{(1 - 1/\alpha) \cdot n} \cdot (3 \cdot \eta \cdot K)^{1/\alpha} \right\}\\
&= B\left(T_{r}^{n}(\beta),(2-r)^{(1 - 1/\alpha) \cdot n} \cdot (3 \cdot \eta \cdot K)^{1/\alpha}\right)
\end{align*}
(Here and throughout we denote by $B(y, l)$, the open Euclidean ball centred at $y$ of radius $l$.)  Hence, given $\delta > 0$, there exists a natural number $M = M(\delta) \in \mathbb{N}$ such that
\begin{align*}
\left\{ 
B\left(T_{r}^{n}(\beta),(2-r)^{(1 - 1/\alpha) \cdot n} \cdot (3 \cdot \eta \cdot K)^{1/\alpha}\right)
\colon
n \geq M \; \text{and} \; n \in \mathbb{N}
\right\}
\end{align*}
is an open $\delta$-cover of $\limsup_{n \to +\infty} A_{n, r, \eta}$.  Therefore, for $s > 0$ and $\delta > 0$, letting $\mathcal{H}_{\delta}^{s}$ denote the $\delta$-approximation to the $s$-dimensional Hausdorff measure, we have that 
\begin{align*}
\mathcal{H}_{\delta}^{s}\left( \limsup_{n \to +\infty} A_{n, r, \eta} \right)
&\leq \sum_{n = M}^{+\infty} \lambda \left( B\left(T_{r}^{n}(\beta),(2-r)^{(1 - 1/\alpha) \cdot n} \cdot (3 \cdot \eta \cdot K)^{1/\alpha}\right)\right)^{s}\\
&\leq \sum_{n = M}^{+\infty} (2-r)^{(1 - 1/\alpha) \cdot s \cdot n} \cdot (3 \cdot \eta \cdot K)^{s/\alpha}\\
&= \frac{(3 \cdot \eta \cdot K)^{s/\alpha} \cdot (2 - r)^{(1 - 1/\alpha) \cdot s \cdot M}}{1 - (2 - r)^{(1- 1/\alpha) \cdot s}}.
\end{align*}
Since $\alpha \in (0, 1)$, this latter quantity is finite for all $s > 0$ and $\delta > 0$, and so $\mathcal{H}^{s}( \limsup_{n \to +\infty} A_{n, r, \eta})$ is finite for all $s > 0$.  This yields that  $\mathrm{dim}_{\mathcal{H}} ( \limsup_{n \to +\infty} A_{n, r, \eta} ) = 0$ as required.  (Here $\mathcal{H}^{s}$ denotes the $s$-dimensional Hausdorff measure.)
\end{proof}

\pagebreak

\subsection{Auxiliary results for the case $\mathbf{r =1}$}

\subsubsection{Infinite ergodic theory revisited}\label{sec:Section_4_3_1}
\mbox{ }

The \textit{transfer operator} $\widehat{T}_{1} \colon \mathcal{L}_{1}^{1}([0, 1]) \circlearrowleft$ of $T_{1}$ is defined by 
\begin{align*}
\widehat{T}_{1}(f) = \frac{\mathcal{P}_{1}(f \cdot h_{1})}{h_{1}}.
\end{align*}
Namely $\widehat{T}_{1}$ is the dual operator of $T_{1}$ with respect to $\mu_{1}$; that is the positive linear operator satisfying
\begin{align*}
\widehat{T}_{1}(f) \coloneqq \frac{\mathrm{d}\nu_{1, f} \circ {T_{1}}^{-1}}{\mathrm{d}\mu_{1}},
\quad \text{where} \quad
\nu_{1, f}(A) \coloneqq \int \mathds{1}_{A} \cdot f \, \mathrm{d}\mu_{1},
\quad \text{for all Borel sets $A \subset [0, 1]$.}
\end{align*}
Note, the domain of definition of $\widehat{T}_{1}$ can be extended to any well-defined real-valued function.

Let $Y \subset [0, 1]$ be such that $\mu_{1}(Y)$ is positive and finite.  For each $n \in \mathbb{N}$, define the \textit{return time operator} $T^{(n)}_{Y} \colon \mathcal{L}^{1}_{1}([0, 1]) \circlearrowleft$ by
\begin{align*}
T^{(n)}_{Y}(f) \coloneqq \mathds{1}_{Y} \cdot \widehat{T}_{1}^{n}(\mathds{1}_{Y} \cdot f),
\end{align*}
and define the \textit{first return time operator} $R_{n} \colon \mathcal{L}^{1}_{1}([0, 1]) \circlearrowleft$ by
\begin{align*}
R_{n}(f) \coloneqq \mathds{1}_{Y} \cdot \widehat{T}^{n}_{1}(\mathds{1}_{\overline{\{ y \in Y \colon \phi_{Y}(y) = n\}}} \cdot f).
\end{align*}
Here $\phi_{Y}(y)$ denotes the \textit{first return time} of $y \in Y$ given by $\phi_{Y}(y) \coloneqq \inf \{ n \in \mathbb{N} \colon T_{1}^{n}(y) \in Y \}$.  

We let $\mathcal{L}^{\infty}(Y)$ denotes the Banach space of equivalence classes $[f]$ of functions, where for each representative $h \colon [0, 1] \to \mathbb{C}$ of $[f]$, we have that $h$ is a Lebesgue measurable function with $\lVert h \rVert_{\mathcal{L}^\infty} \coloneqq \inf \{ \rVert f \rVert_{\infty} \colon \lambda\{ x \colon f(x) \neq h(x) \} = 0 \} < +\infty$ and with $h$ supported on $Y$.  Here $f, g$ belong to the same equivalence class, if and only if, $\lVert f - g \rVert_{\mathcal{L}^{\infty}} = 0$.  Following convention, we will write $f \in \mathcal{L}^{\infty}([0, 1])$ to mean a function $f \colon [0, 1] \to \mathbb{C}$ which belongs to an equivalence class of $\mathcal{L}^{\infty}([0, 1])$.

Let $\mathcal{B}$, equipped with a norm $\lVert \cdot \rVert_{\mathcal{B}}$, be a Banach space of $\mathbb{C}$-valued functions $f \in \mathcal{L}_{1}^{1}([0, 1])$ with domain $[0, 1]$ that are supported on a subset of $Y$ and which satisfy the following five conditions. \label{page:refR1-R5}
\begin{enumerate}
\item[(R1)] If $f \in \mathcal{B}$, then $f \in \mathcal{L}^{\infty}([0, 1])$ and $R(1)(f) \in \mathcal{B}$, where $R(1) \coloneqq \sum_{n=1}^{+\infty} R_{n}$.
\item[(R2)] The inequality $\lVert f \rVert_{\mathcal{L}^{\infty}} \leq \lVert f \rVert_{\mathcal{B}}$ holds for all $f \in \mathcal{B}$.
\item[(R3)] \textit{The Renewal Equation}: For all $n \in \mathbb{N}$, the operator $R_{n}\vert_{\mathcal{B}}$ is bounded and linear.  Moreover, there exists a constant $C > 0$, such that $\lVert R_{n} \lVert \leq C \cdot \mu_{1}(\{ y \in Y \colon \phi_{Y}(y) = n \})$.
\item[(R4)] \textit{Spectral Gap}: The operator $R(1)$ restricted to $\mathcal{B}$ has a simple isolated eigenvalue at $1$.
\item[(R5)] \textit{Aperiodocity}: For $z \in \mathbb{D} \setminus \{ 1 \}$, the value $1$ is not in the spectrum of $R(z) \coloneqq \sum_{n = 1}^{+\infty} z^{n} \, R_{n} \colon \mathcal{B} \circlearrowleft$.  (Here $\mathbb{D}$ denotes the closed unit ball in $\mathbb{C}$.)
\end{enumerate}

\begin{theorem}[{\cite[Theorem 2.1]{MT:2011}}]\label{thm:MT:2011:thm2.1}
If conditions (R1) to (R5) are satisfied, then the limit
\begin{align*}
\lim_{n \to +\infty} \sup_{f \in \mathcal{B}; \; \lVert f \rVert_{\mathcal{B}} \leq 1} \left\lVert \ln(n) \cdot T^{(n)}_{Y}(f) - \int_{Y} f \, \mathrm{d}\mu \right\rVert_{\mathcal{B}},
\end{align*}
exists and converges to zero.
\end{theorem}

In the following proposition, we give an example of when the conditions (R1) to (R5) are satisfied.   This, we believe is a folklore result, a full proof of the result can be found in the Section~\ref{sec:appendix}.

\begin{proposition}\label{prop:rmk1}
Let $Y = [1/2, 1]$ and let $\mathrm{BV}(Y)$ denote the space of $\mathbb{C}$-valued right-continuous functions with domain $[0, 1]$ that are supported on a subset of $Y$ and which are of bounded variation.  We define, for all $f \in \mathrm{BV}(Y)$, the norm $\lVert f \rVert_{\mathrm{BV}} \coloneqq \lVert f \rVert_{\infty} + V_{Y}(f)$.  The space $\mathrm{BV}(Y)$ is a Banach space (Proposition~\ref{prop:Banach}) and satisfies conditions (R1)~to~(R5).
\end{proposition}

For $k \in \mathbb{N}_{0}$, set 
\begin{align*}
Y_{k} \coloneqq T_{1}^{-k}(Y) \setminus \bigcup_{j = 0}^{k - 1} T_{1}^{-j}(Y).
\end{align*}
Indeed, if $Y = [1/2, 1]$, then $Y_{0} = Y$ and $Y_{k} = [1/(k+2), 1/(k+1))$ for $k \geq 1$.  For each $f \colon [0, 1] \to \mathbb{C}$ with $\lVert f \rVert_{\infty} < \infty$, we let $\widetilde{f}_{k} \coloneqq \mathds{1}_{Y_{k}} \cdot f$ and we write $f \in \mathcal{B}([0, 1])$, if $f \in \mathcal{L}^{1}_{1}([0, 1])$ and $\widehat{T}^{k}_{1}(\widetilde{f}_{k}) \in \mathcal{B}$ for all $k \in \mathbb{N}_{0}$.

By definition, for a measurable function $g \colon [0, 1] \to \mathbb{C}$ with $\lVert g \rVert_{\infty} < +\infty$ and for $f \in \mathcal{L}^{1}_{1}([0, 1])$, we have that
\begin{align*}
\int \widehat{T}_{1}(f) \cdot g \,\mathrm{d}\mu_{1} = \int f \cdot g \circ T_{1} \,\mathrm{d}\mu_{1}.
\end{align*}
Moreover, since $\widehat{T}_{1}(f) = \mathcal{P}_{1}(f \cdot h_{1}) / h_{1}$, the operator $\widehat{T}_{1}$ can be written in terms of the inverse branches of $T_{1}$, namely
\begin{align}\label{eqn:FareyDual}
\widehat{T}_{1}(f)(x) = f_{1, 0}(x) \cdot f \circ f_{1, 1}(x) + f_{1, 1}(x) \cdot f \circ f_{1, 0}(x).
\end{align}
This implies, on $[0, 1]$, for all $n \in \mathbb{N}$ and integers $j > n$, that $\mathds{1}_{Y} \cdot \widehat{T}^{n}_{1}(\widetilde{f}_{j}) = 0$ and $\widehat{T}_{1}^{n}(\widetilde{f}_{n}) = \mathds{1}_{Y} \cdot \widehat{T}_{1}^{n}(\widetilde{f}_{n})$, and hence, that
\begin{align*}
\mathds{1}_{Y} \cdot \widehat{T}_{1}^{n}(f)
= \sum_{j = 0}^{n} \mathds{1}_{Y} \cdot \widehat{T}_{1}^{n-j}( \mathds{1}_{Y} \cdot \widehat{T}_{1}^{j}(\widetilde{f}_{j})).
\end{align*}
See \cite[p. 11]{KS:2008} or \cite[Section 3.3.2]{JK:2011} for further details on the transfer operator $\widehat{T}_{1}$, the Perron Frobenius operator $\mathcal{P}_{1}$ and the equalities given above.

\begin{theorem}[{\cite[Theorem 10.4]{MT:2011}}]\label{thm:MT:2001:sec10}
Let $f \in \mathcal{B}([0, 1])$ be such that $\lVert f \rVert_{\infty} < + \infty$.  If
\begin{align}\label{eq:assumption_MT}
\sum_{k = 0}^{+\infty} \lVert \widehat{T}_{1}^{k}(\widetilde{f}_{k}) \rVert_{\infty} < +\infty,
\end{align}
then on $Y$
\begin{align*}
\lim_{n \to +\infty} \ln(n) \cdot \widehat{T}_{1}^{n}(f) = \int f \, \mathrm{d}\mu.
\end{align*}
\end{theorem}

\begin{remark}\label{rmk:thm:MT:2001:sec10}
If $f \in \mathrm{BV}(0, 1)$, then $f$ satisfies the conditions of Theorem~\ref{thm:MT:2001:sec10}.  To see this observe that, by the identity given in \eqref{eqn:FareyDual},
\begin{align*}
\widehat{T}_{1}^{n}(f \cdot \mathds{1}_{Y_{n}}) = \prod_{k = 0}^{n-1} f_{1,1} \circ f_{1, 0}^{k} \cdot f \circ f_{1,0}^{n}.
\end{align*}
Therefore, since $f$, $f_{1, 0}$ and $f_{1, 1}$ are of bounded variation and the composition and product of functions of bounded is again of bounded variation it follows that $\widehat{T}^{n}(f \cdot \mathds{1}_{Y_{k}}) \in \mathrm{BV}(Y)$.  Moreover, since a function of bounded variation has finite supremum norm, we have that
\begin{align*}
\sum_{k = 0}^{+\infty} \lVert \widehat{T}_{1}^{k}(f \cdot \mathds{1}_{Y_{k}}) \rVert_{\infty}
\leq \sum_{k = 0}^{+\infty} \frac{1}{(k + 1)!} \lVert f \rVert_{\infty} < + \infty.
\end{align*}
\end{remark}

\begin{proof}
We acknowledge that the first part of this proof is inspired by the first paragraph of the proof of \cite[Theorem 10.4]{MT:2011}.

By Theorem~\ref{thm:MT:2011:thm2.1} and Proposition~\ref{prop:rmk1}, we have, for each $n \in \mathbb{N}_{0}$, that there exist $\theta_{n} \colon [0, 1] \to \mathbb{C}$ supported on a subset of $Y$ with $\lVert \theta_{n} \rVert_{\infty} = \mathfrak{o}(1/\ln(n + 2))$ and
\begin{align*}
\mathds{1}_{Y} \cdot \widehat{T}_{1}^{n}(\mathds{1}_{Y} \cdot f) = \frac{1}{\ln(n+2)} \int f \, \mathrm{d}\mu_{1} \cdot \mathds{1}_{Y} + \theta_{n} \cdot f.
\end{align*}
For $n \in \mathbb{N}$ and $j \in \{ 0, 1, 2, \dots, n\}$, set $c_{j, n} \coloneqq \ln(n)/\ln(n-j + 2) -1$.  For all natural numbers $n > 1$, we have on $Y$
\begin{align}\label{eq:MT:2011:10.4}
\begin{aligned}
&\left\lvert \ln(n) \cdot \widehat{T}^{n}_{1}(f) - \int f \, \mathrm{d}\mu_{1} \right\rvert\\
&= \left\lvert \ln(n) \sum_{j = 0}^{n} \mathds{1}_{Y} \cdot \widehat{T}^{n - j}_{1}\left(\mathds{1}_{Y} \cdot \widehat{T}^{j}_{1}(\widetilde{f}_{j})\right)  - \int f \, \mathrm{d}\mu_{1} \right\rvert \\
&\leq \left\lvert 
\ln(n) \sum_{j = 0}^{n} \frac{1}{\ln(n -j + 2)} \int \widehat{T}^{j}_{1}(\widetilde{f}_{j}) \, \mathrm{d}\mu_{1} - \int f \, \mathrm{d} \mu_{1}
 \right\rvert + \ln(n) \sum_{j = 0}^{n} \lVert \theta_{n-j} \rVert_{\infty} \cdot \lVert \mathds{1}_{Y} \cdot \widehat{T}^{j}_{1}(\widetilde{f}_{j})\rVert_{\infty}\\
&\leq \sum_{j = 0}^{n} c_{n, j} \int \rvert \widetilde{f}_{k} \lvert \, \mathrm{d} \mu_{1} + \sum_{j = n+1}^{+\infty} \int \lvert \widetilde{f}_{j} \rvert \, \mathrm{d}\mu_{1} + \ln(n) \sum_{j = 0}^{n} \lVert \theta_{n-j} \rVert_{\infty} \cdot \lVert \mathds{1}_{Y} \cdot \widehat{T}^{j}_{1}(\widetilde{f}_{j})\rVert_{\infty}.
\end{aligned}
\end{align}
We now proceed by showing that the three terms in the final line of \eqref{eq:MT:2011:10.4} each converge to zero as $n$ tends to infinity, for all $x \in Y$.
\begin{enumerate}[label=(\alph*)]
\item Since $\mu_{1}(Y_{j}) = \ln(1 + 1/(j+1)) \sim 1/(j+1)$ and since $f \in \mathcal{L}^{\infty}([0,1])$, there exists a constant $c > 0$ such that $\lVert \widetilde{f}_{j} \rVert_{1, 1} \leq c/ (j + 1)$, for all $j \in \mathbb{N}_{0}$.  For $\epsilon > 0$ if $0 \leq j \leq n - n^{1/(1+\epsilon)} + 2$, then for all $n \in \mathbb{N}$, $\ln(n) / \ln(n - j + 2) \leq 1+\epsilon$, .  Thus, for a given $\epsilon > 0$, we have that
\begin{align*}
&\lim_{n \to +\infty} \sum_{j = 0}^{n} \frac{\ln(n)}{\ln(n-j + 2)} \int \rvert \widetilde{f}_{j} \lvert \, \mathrm{d}\mu_{1}\\
&\leq \lim_{n \to +\infty} \sum_{j = 0}^{n - \lceil n^{1/(1+\epsilon)}\rceil + 1} (1 + \epsilon) \int \rvert \widetilde{f}_{j} \lvert \, \mathrm{d}\mu_{1} +  \lim_{n \to +\infty} \sum_{j = n - \lceil n^{1/(1+\epsilon)} \rceil + 2}^{n} \frac{c \cdot \ln(n)}{j \cdot \ln(n-j + 2)}\\
&\leq (1 + \epsilon) \int \lvert f \rvert \, \mathrm{d} \mu_{1} + \lim_{n \to +\infty} \frac{c}{\ln(2)} \frac{(\lceil n^{1/(1+\epsilon)}\rceil +2 ) \cdot \ln(n)}{(n - n^{1/(1+\epsilon)} + 2)}
= (1 + \epsilon) \int \lvert f \rvert \, \mathrm{d} \mu_{1}.
\end{align*}
Moreover, since for all integers $n > 1$ and $j \in \{2, 3, \dots, n\}$, we have that $\ln(n)/\ln(n-j + 2) > 1$ and since $\lim_{n \to +\infty} \ln(n)/\ln(n - j + 2) = 1$, for $j \in \{ 0, 1 \}$, it follows that,
\begin{align*}
\lim_{n \to +\infty} \sum_{j = 0}^{n} \frac{\ln(n)}{\ln(n-j + 2)} \int \rvert \widetilde{f}_{j} \lvert \, \mathrm{d} \mu_{1}
\geq \lim_{n \to +\infty} \sum_{j = 0}^{n} \int \rvert \widetilde{f}_{j} \lvert \, \mathrm{d} \mu_{1}
= \int \lvert f \rvert \, \mathrm{d} \mu_{1}.
\end{align*}
Hence, we have that
\begin{align*}
\lim_{n \to +\infty} \sum_{j = 0}^{n} c_{n, j} \int \rvert \widetilde{f}_{j} \lvert \, \mathrm{d} \mu = 0.
\end{align*}
\item Since $f \in \mathcal{L}^{1}_{1}([0, 1])$, using the definition of $\widetilde{f}_{j}$, we obtain that the second term in the final line of \eqref{eq:MT:2011:10.4} converges to zero.
\item For $j \in \mathbb{N}_{0}$, the map $f_{1, 1} \circ f_{1, 0}^{j}$ is order reversing and, an inductive argument can be used to show that $f_{1, 1} \circ f_{1, 0}^{j}(x) = (1 + j \cdot x)/( 1+ (j+1) \cdot x)$.  Using the fact that $Y_{k} \subseteq f_{1, 0}^{k} \circ f_{1, 1}((0, 1])$, for $k \in \mathbb{N}$, and the representation of $\widehat{T}_{1}$ given in \eqref{eqn:FareyDual}, an inductive argument yields, for all $j \in \mathbb{N}_{0}$, that
\begin{align*}
\widehat{T}^{j}_{1}(\widetilde{f}_{j}) (x) =
\left(\prod_{k = 0}^{j-1} f_{1, 1} \circ f_{1, 0}^{k}(x) \right) \cdot \widetilde{f}_{j} \circ f_{1, 0}^{j}(x),
\end{align*}
and thus, that
\begin{align}\label{eq:sup_norm_1/j}
\lVert \mathds{1}_{Y} \cdot \widehat{T}^{j}(\widetilde{f}_{j}) \rVert_{\infty}
\leq \left(\prod_{k = 0}^{j-1} \frac{1 + k/2}{1 + (k+1)/2} \right) \, \lVert \widetilde{f}_{j} \rVert_{\infty}
\leq \frac{2}{j + 2} \lVert \widetilde{f}_{j} \rVert_{\infty}
\leq \frac{2}{j + 2}\lVert f \rVert_{\infty}.
\end{align}
Since $\lVert \theta_{n} \rVert_{\infty} = \mathfrak{o}(1/\ln(n + 2))$, given an $\epsilon > 0$, there exists $N_{\epsilon} \in \mathbb{N}$ such that $\lVert \theta_{m} \rVert_{\infty} \leq 2\epsilon/\ln(m)$, for all $m \geq N_{\epsilon}$.  Moreover, the value $\Theta \coloneqq \sup \{ \lVert \theta_{n} \rVert_{\infty} \colon n \in \mathbb{N}_{0} \}$ is positive and finite.  Combining these statements, we have the following inequality.
\begin{align*}
\qquad \qquad
\ln(n) \sum_{j = 0}^{n} \lVert \theta_{n-j} \rVert_{\infty} \cdot \lVert \widehat{T}^{j}_{1}(\widetilde{f}_{j})\rVert_{\infty}
\leq 2 \cdot \epsilon \sum_{j = 0}^{n-N_{\epsilon}} \frac{\ln(n)}{\ln(n - j)} \lVert \widehat{T}^{j}_{1}(\widetilde{f}_{j})\rVert_{\infty} + 2 \cdot \Theta \cdot \lVert f \rVert_{\infty} \cdot \ln(n) \!\! \sum_{j = n - N_{\epsilon}+1}^{n} 1/j.
\end{align*}
Using \eqref{eq:assumption_MT} and \eqref{eq:sup_norm_1/j} a similar argument to that given in (a) yields that
\begin{align*}
\lim_{n \to +\infty} 2 \cdot \epsilon \sum_{j = 0}^{n-N_{\epsilon}} \frac{\ln(n)}{\ln(n - j)}  \lVert \widehat{T}^{j}_{1}(\widetilde{f}_{j})\rVert_{\infty}
\leq 2 \cdot \epsilon \cdot (1 + \epsilon) \sum_{k = 0}^{+\infty} \lVert \widehat{T}^{k}_{1}(\widetilde{f}_{k}) \rVert_{\infty}.
\end{align*}
Thus, for a given $\epsilon > 0$, we have that
\begin{align*}
&\lim_{n \to +\infty} \ln(n) \sum_{j = 0}^{n} \lVert \theta_{n-j} \rVert_{\infty} \cdot \lVert \widehat{T}^{j}_{1}(\widetilde{f}_{j})\rVert_{\infty}\\
&\leq 2 \cdot \epsilon \cdot (1 + \epsilon) \sum_{j = 0}^{+\infty}  \lVert \widehat{T}^{j}_{1}(\widetilde{f}_{j})\rVert_{\infty} + \lim_{n \to +\infty} 2 \cdot \Theta \cdot \lVert f \rVert_{\infty} \cdot \ln(n) \sum_{j = n - N_{\epsilon}+1}^{n} j^{-1}\\
&\leq 2 \cdot \epsilon \cdot (1 + \epsilon) \sum_{j = 0}^{+\infty} \lVert \widehat{T}^{j}_{1}(\widetilde{f}_{j})\rVert_{\infty} + 2 \cdot \Theta \cdot \lVert f \rVert_{\infty} \lim_{n \to +\infty} \ln(n) \cdot \ln\left( \frac{n}{n - N_{\epsilon}}  \right).
\end{align*}
An application of  L'H\^opital's rule yields that
\begin{align*}
\lim_{n \to +\infty} \ln(n) \sum_{j = 0}^{n} \lVert \theta_{n-j} \rvert_{\infty} \cdot \lVert \widehat{T}^{j}_{1}(\widetilde{f}_{j})\rVert_{\infty}
\leq 2 \cdot \epsilon \cdot (1 + \epsilon) \sum_{j = 0}^{\infty} \lVert \widehat{T}^{j}_{1}(\widetilde{f}_{j})\rVert_{\infty}.
\end{align*}
Since $\epsilon$ was chosen arbitrarily, this completes the proof.
\end{enumerate}
\end{proof}

\begin{theorem}\label{prop:jk2011:th7}
If $f \in \mathcal{L}^{1}_{\mu}([0, 1])$ satisfies
\begin{align*}
\ln(n) \cdot \widehat{T}^{n}(f) \rightarrow \int f d\mu_{1}
\end{align*}
uniformly on $Y$, then the same convergence holds on any compact subsets of $(0,1]$. 
\end{theorem}

\begin{proof}
For $g \in \mathcal{L}^{1}_{1}([0, 1])$, $x \in [0,1]$ and $n\in \mathbb{N}$, we have that
\[
\begin{aligned}
(\mathcal{P}_{1}^{n+1}(\varphi \cdot g))(x) = \mathcal{P}_{1} ( (\mathcal{P}^{n}_{1}(\varphi \cdot g) )(x) )
=  \lvert f'_{1, 0}(x) \rvert \cdot (\mathcal{P}^{n}_{1}(\varphi \cdot g) ) (f_{1, 0}(x)) + \lvert f'_{1, 1}(x) \rvert \cdot (\mathcal{P}_{1}^{n}( \varphi \cdot g) ) (f_{1, 1}(x)),
\end{aligned}
\]
and hence
\begin{align}\label{eq:Ext}
(\mathcal{P}_{1}^{n}(\varphi \cdot g))(f_{1, 0}(x)) = 
\frac{(\mathcal{P}_{1}^{n+1}(\varphi \cdot g))(x) - \lvert f'_{1, 1}(x) \rvert \cdot (\mathcal{P}_{1}^{n}(\varphi \cdot g)) (f_{1, 1}(x))}{\lvert f'_{1, 0}(x) \rvert}.
\end{align}
We proceed by induction as follows. The start of the induction is given by the assumption in the theorem. For the inductive step, assume that the statement holds for $\bigcup_{k = 0}^{j} Y_{k}$, for some $j \in \mathbb{N}$.  Consider an arbitrary $y \in Y_{j+1}$, and let $x$ denote the unique element in $Y_{j}$ such that $f_{1, 0}(x) = y$.  Using \eqref{eq:Ext}, the fact that $\widehat{T}_{1}(g) = \mathcal{P}_{1}(h_{1} \cdot g)/h_{1}$ and the inductive hypothesis, we obtain that
\begin{align*}
&\ln(n) \cdot \widehat{T}^{n}_{1}(g)(y)\\
&= \ln(n) \cdot \widehat{T}^{n}_{1}(g)(f_{1, 0}(x))\\
&=\frac{\ln(n) \cdot (\mathcal{P}_{1}^{n}(h_{1} \cdot g ))(f_{1, 0}(x))}{h_{1} (f_{1, 0}(x))}\\
&= \frac{\ln(n) \cdot (\mathcal{P}_{1}^{n+1}(h_{1} \cdot g))(x) - \lvert f'_{1, 1}(x)\rvert \cdot \ln(n) \cdot (\mathcal{P}_{1}^{n}(h_{1} \cdot g))(f_{1, 1}(x))}{h_{1} (f_{1, 0}(x)) \cdot \lvert{f'_{1, 0}(x)} \rvert}\\
&= \frac{1}{h_{1}(f_{1, 0}(x)) \cdot \lvert f_{1, 0}'(x) \rvert} \left( h_{1}(x) \cdot \ln(n) \cdot \widehat{T}_{1}^{n+1}(g)(x) - \lvert f'_{1, 1}(x) \rvert \cdot h_{1}(f_{1, 1}(x)) \cdot \ln(n) \cdot \widehat{T}_{1}^{n}(g)(f_{1, 1}(x))\right)\\
&\sim \frac{h_{1} (x) - h_{1} (f_{1, 1}(x)) \cdot \lvert f'_{1, 1}(x) \rvert}{h_{1} (f_{1, 0}(x)) \cdot \lvert f'_{1, 0}(x)\rvert} \int g \,\mathrm{d} \mu
= \int g \, \mathrm{d}\mu.
\end{align*}
The last equality in the above calculation is a consequence of \eqref{eq:alternative_form_Pr} and the fact that $\mathcal{P}_{1}(h_{1}) = h_{1}$.
\end{proof}

Our next result, Lemma~\ref{lem:bv_part_r=1}, is the analogous result of Lemma~\ref{lem:bv_part} for $r = 1$.  In the proof of this result the following will play an essential role.  For $n \in \mathbb{N}$ and $\beta \in (0, 1]$, We recall that $p_{n} = p_{n}(\beta)$ and $q_{n} = q_{n}(\beta)$ are as defined in \eqref{eq:definition_p_q}, and define $k(n) = k(n, \beta)$, $m(n) = m(n, \beta)$ and $r(n) = r(n, \beta)$ by
\begin{align}\label{eq:def_m_r_k}
\begin{aligned}
k(n) &\coloneqq \max \{ k \in \{ 1, 2, \dots, n\} \colon \omega_{1, k}(\beta) = 1 \},\\
m(n) &\coloneqq \mathrm{card}\{ \ell\in \{ 1, 2, \dots, n \} \colon \omega_{1, \ell}(\beta) = 1 \}
\quad \text{and} \quad\\
r(n) &\coloneqq n - k(n).
\end{aligned}
\end{align}
The following list of properties can be discerned from the given definitions and remarks.
\begin{enumerate}
\item\label{enumi:Prop_1} If $k(n) = n$, then $a_{m(n)} = n - k(n -1)$.
\item\label{enumi:Prop_1.5} If $(b_{m})_{m \in \mathbb{N}}$ is a sequence of positive real numbers, then, for $n \in \mathbb{N}$, we have that 
\begin{align*}
T_{1}([0; b_{1}, b_{2}, \dots, b_{n}]) =
\begin{cases}
[0; b_{1}-1, b_{2}, \dots, b_{n}] & \text{if} \; b_{1} > 1,\\
[0; b_{2}, \dots, b_{n}] & \text{otherwise}.
\end{cases}
\end{align*}
\item\label{enumi:Prop_2} The function $f_{1, \omega_{1}(\beta)\vert_{n}}$ is a M\"obius transformation and for all $x \in [0, 1]$, \!$\displaystyle{\lim_{n \to +\infty} f_{1, \omega_{1}(\beta)\vert_{n}}(x) = \beta}$.
\item\label{enumi:Prop_4} For $n \in \mathbb{N}$, we have that
\begin{align*}
f_{1, \omega_{1}(\beta)\vert_{n}}(0) = \frac{p_{m(n)}}{q_{m(n)}} = [0; a_{1}, a_{2}, \dots, a_{m(n)}]
\end{align*}
and
\begin{align*}
f_{1, \omega_{1}(\beta)\vert_{n}}(1) = \frac{(r(n)+1)\cdot p_{m(n)} + p_{m(n) - 1}}{(r(n)+1) \cdot q_{m(n)} + q_{m(n) - 1}} = [0; a_{1}, a_{2}, \dots, a_{m(n)}, r(n) + 1].
\end{align*}
\end{enumerate}

\begin{lemma}\label{lem:fomegan}
For $n \in \mathbb{N}$ and $\beta \in (0, 1]$, we have that
\begin{align}\label{eq:fomegann}
f_{1, \omega_{1}(\beta)\vert_{n}}(x) = \frac{(r(n) \cdot p_{m(n)} + p_{m(n)-1}) \cdot x + p_{m(n)}}{(r(n) \cdot q_{m(n)} + q_{m(n) - 1}) \cdot x + q_{m(n)}},
\end{align}
where $p_{n} = p_{n}(\beta)$ and $q_{n} = q_{n}(\beta)$ are as defined in \eqref{eq:definition_p_q}.
\end{lemma}

\begin{proof}
The function $f_{1, \omega_{1}(\beta)\vert_{n}}$ is a M\"obius transformation and moreover, a M\"obius transformation is uniquely determined by its values at three distinct points.  Let us consider the case when $\omega_{1, n}(\beta) = 1$.  By definition we have that $r(n) = 0$ and so the function on the RHS of \eqref{eq:fomegann} becomes
\begin{align}\label{eq:fomegann1}
x \mapsto \frac{p_{m(n)-1} \cdot x + p_{m(n)}}{q_{m(n) - 1}\cdot x + q_{m(n)}}.
\end{align}
By Property (\ref{enumi:Prop_4}) given above,
\begin{align*}
0 \mapsto \frac{p_{m(n)}}{q_{m(n)}} = f_{\omega_{1}(\beta)\vert_{n}}(0) \quad \text{and} \quad 1 \mapsto \frac{p_{m(n)-1} + p_{m(n)}}{q_{m(n) - 1} + q_{m(n)}} = f_{\omega_{1}(\beta)\vert_{n}}(1).
\end{align*}
Since $f_{1, \omega_{1}(\beta)\vert_{n}}$ is a contraction, by Banach's fixed point theorem, there exists a unique $x \in [0, 1]$ such that $f_{1, \omega_{1}(\beta)\vert_{n}}(x) = x$.  By Properties (\ref{enumi:Prop_1}) and (\ref{enumi:Prop_1.5}) given above the pre-periodic point
\begin{align*}
[0; \overline{a_{1}, \dots, a_{m(n)}}] \coloneqq [0; a_{1}, \dots, a_{m(n)}, a_{1}, \dots, a_{m(n)}, a_{1}, \dots, a_{m(n)}, \dots, a_{1}, \dots, a_{m(n)}, \dots ]
\end{align*}
is a fixed point of $f_{1, \omega_{1}(\beta)\vert_{n}}$.  Further, by \cite[Exercise 1.3.10]{DK:2002} it follows that the point $[0; \overline{a_{1}, \dots, a_{m(n)}}]$ is a fixed point of the map given in \eqref{eq:fomegann1}.  This completes the proof of the result for when $\omega_{n} = 1$.

The result for the case when $\omega\vert_{n} \neq 1$, follows from the definition of $r(n)$ and the case when $\omega_{n} = 1$, together with the observation that $f_{1, 0}^{n}(x) = x / (1 + n\cdot x)$, for $n \in \mathbb{N}$ and all $x \in [0, 1]$.
\end{proof}

\begin{lemma}\label{lem:bv_part_r=1}
For $\alpha \in (0, 1)$, $\beta \in (0, 1]$ of intermediate $\alpha$-type and $v \in \mathfrak{U}_{\beta, \alpha}$, we have that
\begin{align*}
\lim_{n \to \infty} \ln(n) \cdot \widehat{T}_{1}^{n}(v \cdot \mathds{1}_{[0, 1] \setminus [\omega_{1}(\beta)\vert_{n}]} / h_{1}) = \int v \; \mathrm{d}\lambda
\end{align*}
uniformly on compact subsets of $(0, 1)$.
\end{lemma}

\begin{proof}
Let $K$ be a compact subset of $(0, 1)$ and let $a, b \in (0, 1)$ be such that $K \subseteq [a, b]$.  Let $N \in \mathbb{N}$ be fixed.  By Proposition~\ref{prop:rmk1} and Theorems~\ref{thm:MT:2001:sec10} and~\ref{prop:jk2011:th7} together with Remark~\ref{rmk:thm:MT:2001:sec10}, since the function $v \cdot \mathds{1}_{[0, 1] \setminus [\omega_{1}(\beta)\vert_{N}]}$ is of bounded variation, it follows that
\begin{align*}
\lim_{n \to \infty} \ln(n) \cdot \widehat{T}_{1}^{n}(v \cdot \mathds{1}_{[0, 1] \setminus [\omega_{1}(\beta)\vert_{N}]} / h_{1}) = \int v \cdot \mathds{1}_{[0, 1] \setminus [\omega_{1}(\beta)\vert_{N}]}  \; \mathrm{d}\lambda
\end{align*}
Therefore, by linearity and positivity of the operator $\widehat{T}_{1}$, and since $\lim_{k \to +\infty}\lambda([\omega_{1}(\beta)\vert_{k}]) = 0$, since the observable $v$ is Lebesgue integrable and since $\beta$ is of intermediate $\alpha$-type, it suffices to show that there exists a positive constant $C$ so that 
\begin{align*}
\lim_{n \to +\infty} \ln(n) \cdot \widehat{T}_{1}^{n}(v_{\beta, \alpha} \cdot \mathds{1}_{[\omega_{1}(\beta)\vert_{N}] \setminus [\omega_{1}(\beta)\vert_{n}]} / h_{1})
\leq C \sum_{k = \widetilde{N}}^{+\infty} \sum_{j = 1}^{a_{k}} t_{k, j}^{2\cdot(\alpha - 1) + \epsilon},
\end{align*}
for some given $\epsilon \in (0, 2\cdot(\alpha - 1))$ and where
\begin{enumerate}
\item $t_{n, j}$ is as defined at the end of Section \ref{sec:notation} and
\item $\widetilde{N}$ is the unique integer so that $a_{1} + a_{2} + \dots a_{\widetilde{N}} \leq N < a_{1} + a_{2} + \dots a_{\widetilde{N} + 1}$.
\end{enumerate}
To this end, for each integer $k > 1$, let $\overline{\omega_{1}(\beta)}\vert_{k} \in \Sigma^{k}$ be the unique word of length $k$ such that $[\omega_{1}(\beta)\vert_{k-1}] = [\omega_{1}(\beta)\vert_{k}] \cup [\overline{\omega_{1}(\beta)}\vert_{k}]$.   By Lemma~\ref{lem:fomegan} we have that for all $x \in K$
\begin{enumerate}
\item $\displaystyle{\left\lvert f_{\overline{\omega_{1}(\beta)}\vert_{k}}'(x) \right\rvert \leq a^{-2} \cdot ((r(k) + 1) \cdot q_{m(k)} + q_{m(k)-1} )^{-2}}$,\\
\item if $r(k) + 1 \neq a_{m(k)}$, then 
\begin{align*}
\qquad
\left\lvert \beta - f_{\overline{\omega_{1}(\beta)}\vert_{k}}(x) \right\rvert
&\geq \left\lvert \frac{(r(k)+2) \cdot p_{m(k)} + p_{m(k)-1}}{(r(k)+2) \cdot q_{m(k)} + q_{m(k) - 1}} - \frac{(r(k)+1) \cdot p_{m(k)} + p_{m(k)-1}}{(r(k)+1) \cdot q_{m(k)} + q_{m(k) - 1}} \right\rvert\\
&\geq \frac{1}{2 \cdot ((r(k)+1) \cdot q_{m(k)} + q_{m(k)-1})^{2}},
\end{align*}
\item if $r(k) + 1 = a_{m(k)}$, letting 
\begin{align*}
\qquad
z_{k} = \begin{cases}
b & \text{if} \; m(k) \; \text{is even},\\
a & \text{if} \; m(k) \; \text{is odd},
\end{cases}
\end{align*}
then
\begin{align*}
\qquad
\left\lvert \beta - f_{\overline{\omega_{1}(\beta)}\vert_{k}}(x) \right\rvert
&\geq \left\lvert \frac{(r(k)+1) \cdot p_{m(k)} + p_{m(k)-1}}{(r(k)+1) \cdot q_{m(k)} + q_{m(k) - 1}} - \frac{(r(k) \cdot p_{m(k)} + p_{m(k)-1})\cdot x + p_{m(k)}}{(r(k) \cdot q_{m(k)} + q_{m(k) - 1}) \cdot x + q_{m(k)}} \right\rvert\\
&\geq \frac{1 - z_{k}}{((r(k)+1) \cdot q_{m(k)} + q_{m(k) - 1})^{2}}.
\end{align*}
\end{enumerate}
Since $1/h_{1}$ is of bounded variation, we have by Proposition~\ref{prop:rmk1} and Theorems~\ref{thm:MT:2001:sec10} and~\ref{prop:jk2011:th7} together with Remark~\ref{rmk:thm:MT:2001:sec10}, that there exists a positive constant $C'$, so that for all $k \in \mathbb{N}$ and $x \in K$
\begin{align*}
\widehat{T}_{1}^{k}(1/h_{1})(x) \leq \frac{C'}{\ln(k+1)}.
\end{align*}
Noting that $t_{m(k),r(k)+1} = (r(k) +1)\cdot q_{m(k)} + q_{m(k)-1}$ and, letting $\epsilon$ be such that 
\begin{align*}
\sum_{n = 1}^{+\infty} \sum_{k = 1}^{a_{n}} t_{n, j}^{-2\cdot(1-\alpha) + \epsilon} < + \infty,
\end{align*}
we have that
\begin{align*}
&\lim_{n \to +\infty} \ln(n) \cdot \widehat{T}^{n}_{1}(v_{\beta, \alpha} \cdot \mathds{1}_{[\omega_{1}(\beta)\vert_{N}] \setminus [\omega_{1}(\beta)\vert_{n}]})\\
&= \lim_{n \to +\infty} \ln(n) \sum_{k = N+1}^{n-1} \widehat{T}_{1}^{n-k}\left(\widehat{T}_{1}^{k}\left(v_{\beta, \alpha} \cdot \mathds{1}_{[\overline{\omega_{1}(\beta)}\vert_{k}]} / h_{1} \right)\right)\\
&= \lim_{n \to +\infty} \ln(n) \sum_{k = N+1}^{n-1} \widehat{T}_{1}^{n-k}\left(\left\lvert f_{\overline{\omega_{1}(\beta)}\vert_{k}}' \right\rvert \frac{1}{\lvert \beta - f_{\overline{\omega_{1}(\beta)}\vert_{k}} \rvert^{\alpha}} \frac{1}{h_{1}} \right)\\
&\leq \lim_{n \to +\infty} \frac{C'}{2 \cdot a^{2} \cdot (1 - z_{k})} \sum_{k = N + 1}^{n - 1} \frac{\ln(n)}{\ln(n-k+1)} \frac{1}{((r(k) +1)\cdot q_{m(k)} + q_{m(k)-1})^{2\cdot(1-\alpha)}}\\
&\leq \lim_{n \to +\infty} \frac{C'}{2 \cdot a^{2} \cdot (1 - z_{k})} 
\sum_{k = N + 1}^{\lfloor n/2 \rfloor} \frac{\ln(n)}{\ln(n/2)} \frac{1}{((r(k) +1)\cdot q_{m(k)} + q_{m(k)-1})^{2\cdot(1-\alpha) - \epsilon}}\\
& \phantom{\leq} +
\lim_{n \to +\infty} \frac{C'}{2 \cdot a^{2} \cdot (1 - z_{k})} \sum_{k = \lfloor n /2 \rfloor + 1}^{n - 1} \frac{2 \cdot \ln(n)}{n^{\epsilon}} \frac{1}{((r(k) +1)\cdot q_{m(k)} + q_{m(k)-1})^{2\cdot(1-\alpha) - \epsilon}}\\
&\leq \frac{C'}{a^{2} \cdot (1 - z_{k})} \sum_{k = N + 1}^{+\infty} \frac{1}{((r(k) +1)\cdot q_{m(k)} + q_{m(k)-1})^{2\cdot(1-\alpha) - \epsilon}}
\leq \frac{C'}{a^{2} \cdot (1 - z_{k})} \sum_{k = \widetilde{N}}^{+\infty} \sum_{j = 1}^{a_{k}} t_{k, j}^{2\cdot(\alpha - 1) + \epsilon}.
\end{align*}
This completes the proof.
\end{proof}

\subsubsection{Convergence of the $1$-tail}
\mbox{ }

The aim of this section is to provide an analogous result (Lemma~\ref{lem:tail_r_=_1}) for $r = 1$ of Lemma~\ref{lem:tail_r_neq_1}.  The idea behind the proofs of Lemmata~\ref{lem:tail_r_neq_1} and \ref{lem:tail_r_=_1} are similar, however, in the case that $r = 1$, several technical difficulties arise and thus need to be taken care off.

\begin{lemma}\label{lem:tail_r_=_1}
For $\alpha \in (0, 1)$, $\beta \in [0, 1]$ irrational, $n \in \mathbb{N}$ and $\eta > 0$, we have that
\begin{align*}
\mathrm{dim}_{\mathcal{H}}\left(\limsup_{n \to +\infty} A_{n, 1, \eta}\right) = 0.
\end{align*}
\end{lemma}

\begin{proof}
It is sufficient to prove, for all $k \in \mathbb{N}$, $\eta > 0$ and $\epsilon \in (0, (2k(k+1))^{-1})$, that
\begin{align*}
\mathrm{dim}_{\mathcal{H}}\left(\limsup_{n \to +\infty} A_{n, 1, \eta} \cap ( 1/(k+1) + \epsilon, 1/k - \epsilon ) \right) = 0.
\end{align*}
To this end, for $n \in \mathbb{N}$, set $z = z(n) \coloneqq T_{1}^{n}(\beta)$ and observe that $z$ is the unique real number in $[0, 1]$ such that $f_{1, \omega_{1}(\beta)\vert_{n}}(z) = \beta$.  If $z \in (1/(k+1), 1/k)$, then, for all $x \in (1/(k+1)+\epsilon, 1/k - \epsilon)$, by the mean value theorem and Lemma~\ref{lem:fomegan}, there exists $u \in (1/(k+1), 1/k)$ such that
\begin{align*}
\lvert \beta - f_{1, \omega_{1}(\beta)\vert_{n}}(x)\rvert
= \lvert f_{1, \omega_{1}(\beta)\vert_{n}}(z) - f_{1, \omega_{1}(\beta)\vert_{n}}(x)\rvert
&= \lvert x - z \rvert \cdot \lvert f_{1, \omega_{1}(\beta)\vert_{n}}'(u)\rvert\\
&= \lvert x - T_{1}^{n}(\beta) \rvert \cdot \lvert (r(n) u + 1 ) q_{m(n)} +  q_{m(n)-1} u \rvert^{-2}\\
&\geq k^{2} \cdot \lvert x - T_{1}^{n}(\beta) \rvert \cdot \lvert (r(n) + k ) q_{m(n)} +  q_{m(n)-1} \rvert^{-2}.
\end{align*}
If $z \not\in (1/(k+1), 1/k)$, then, for all $x \in (1/(k+1)+\epsilon, 1/k - \epsilon)$, since $f_{1, \omega_{1}(\beta)\vert_{n}}$ is order preserving or order reversing, we have that 
\begin{align*}
\lvert \beta - f_{1, \omega_{1}(\beta)\vert_{n}}(x)\rvert
&= \lvert f_{1, \omega_{1}(\beta)\vert_{n}}(z) - f_{1, \omega_{1}(\beta)\vert_{n}}(x)\rvert\\
&\geq \min\{ \lvert f_{1, \omega_{1}(\beta)\vert_{n}}(1/k) - f_{1, \omega_{1}(\beta)\vert_{n}}(x)\rvert, \lvert f_{1, \omega_{1}(\beta)\vert_{n}}(1/(k+1)) - f_{1, \omega_{1}(\beta)\vert_{n}}(x)\rvert \}
\end{align*}
and so by the mean value theorem and Lemma~\ref{lem:fomegan}, there exists $u \in (1/(k+1), 1/k)$ such that
\begin{align*}
\lvert \beta - f_{1, \omega_{1}(\beta)\vert_{n}}(x)\rvert
\geq \epsilon \cdot \lvert f_{1, \omega_{1}(\beta)\vert_{n}}'(u)\rvert
&= \epsilon \cdot \rvert (r(n) \cdot u + 1 ) \cdot q_{m(n)} +  q_{m(n)-1} \cdot u \rvert^{-2}\\
&\geq \epsilon \cdot k^{2} \cdot \lvert (r(n) + k ) \cdot q_{m(n)} +  q_{m(n)-1} \rvert^{-2}\!.
\end{align*}
Hence, for $x \in (1/(k+1)+\epsilon, 1/k - \epsilon)$, we have that
\begin{align*}
\ln(n) \cdot v_{n, 1}(x)
&= \frac{\ln(n)}{( (r(n) \cdot x + 1 ) \cdot q_{m(n)} +  q_{m(n)-1} \cdot x )^{2}} \, \frac{1}{\lvert \beta - f_{1, \omega_{1}(\beta)\vert_{n}}(x)\rvert^{\alpha}}\\
&\leq 
\begin{cases}
\displaystyle{\frac{(k + 1)^{2} \cdot \ln(n)}{\lvert T_{1}^{n}(\beta) - x \rvert^{\alpha} \cdot k^{2\cdot\alpha} \cdot( (r(n) + k )\cdot q_{m(n)} + q_{m(n)-1} )^{2\cdot(1-\alpha)}}} & \text{if} \; T_{1}^{n}(\beta) \in (1/(k+1), 1/k),\\
\displaystyle{\frac{(k + 1)^{2} \cdot \ln(n)}{\epsilon^{\alpha}\cdot k^{2\cdot\alpha}\cdot( (r(n) + k ) \cdot q_{m(n)} + q_{m(n)-1} )^{2\cdot(1-\alpha)}}} & \text{if} \; T_{1}^{n}(\beta) \not\in (1/(k+1), 1/k).
\end{cases}
\end{align*}
Since, 
\begin{align*}
&\lim_{n \to +\infty}\frac{(k+1)^{2} \cdot \ln(n)}{\epsilon^{\alpha} \cdot k^{2\cdot\alpha} \cdot( (r(n) + k ) \cdot q_{m(n)} + q_{m(n)-1} )^{2\cdot(1-\alpha)}}\\
&\leq \lim_{n \to +\infty} \frac{(k+1)^{2} \cdot \ln((r(n) + k ) \cdot q_{m(n)} + q_{m(n)-1})}{\epsilon^{\alpha} \cdot k^{2\cdot\alpha}\cdot( (r(n) + k ) \cdot q_{m(n)} + q_{m(n)-1} )^{2\cdot(1-\alpha)}} = 0,
\end{align*}
there exists $M \in \mathbb{N}$ such that, for all $x \in (1/(k+1)+\epsilon, 1/k - \epsilon)$ and $n \geq M$, if $T_{1}^{n}(\beta) \not\in (1/(k+1), 1/k)$, then $\ln(n) \cdot v_{n, 1}(x) < \eta$.  Therefore, for all $n \geq M$, if $T_{1}^{n}(\beta) \not\in (1/(k+1), 1/k)$, then
\begin{align*}
A_{n, 1, \eta} \cap (1/(k+1) + \epsilon, 1/k - \epsilon)
= \emptyset;
\end{align*}
otherwise, if $T_{1}^{n}(\beta) \in (1/(k+1), 1/k)$, then
\begin{align*}
&A_{n, 1, \eta} \cap (1/(k+1) + \epsilon, 1/k - \epsilon)\\
&= \left\{ x \in (1/(k+1) + \epsilon, 1/k - \epsilon) \colon \ln(n) \cdot v_{n, 1}(x) \geq \eta \right\}\\
&\subseteq \left\{ x \in (1/(k+1) + \epsilon, 1/k - \epsilon) \colon \frac{(k+1)^{2} \cdot \ln(n)}{\lvert T_{1}^{n}(\beta) - x \rvert^{\alpha} \cdot k^{2\cdot\alpha} \cdot( (r(n) + k )\cdot q_{m(n)} + q_{m(n)-1} )^{2\cdot(1-\alpha)}} \geq \eta \right\}\\
&\subseteq B\left(T_{1}^{n}(\beta), \frac{(k+1)^{2/\alpha} \cdot \ln(n)^{1/\alpha}}{\eta^{1/\alpha} \cdot k^{2} \cdot( (r(n) + k )\cdot q_{m(n)} + q_{m(n)-1} )^{2\cdot(1/\alpha -1)}} \right) \cap (1/(k+1) + \epsilon, 1/k - \epsilon).
\end{align*}
Hence, given $\delta > 0$, there exists a natural number $K = K(\delta) \geq M$ such that
\begin{align*}
\left\{ 
B\left(T_{1}^{n}(\beta), \frac{(k+1)^{2/\alpha} \cdot \ln(n)^{1/\alpha}}{\eta^{1/\alpha} \cdot k^{2} \cdot( (r(n) + k ) \cdot q_{m(n)} + q_{m(n)-1} )^{2\cdot(1/\alpha - 1)}} \right) 
\colon
n \geq K \; \text{and} \; \exists \; l \in \mathbb{N} \; \text{so that} \; n = - k + \sum_{i = 1}^{l} a_{i} 
\right\}
\end{align*}
is an open $\delta$-cover of
\begin{align*}
\limsup_{n \to +\infty} A_{n, 1, \eta} \cap ( 1/(k+1) + \epsilon, 1/k - \epsilon ).
\end{align*}
Therefore, for $s > 0$ and $\delta > 0$, letting $\mathcal{H}_{\delta}^{s}$ denote the $\delta$-approximation to the $s$-dimensional Hausdorff measure, we have that 
\begin{align*}
&\mathcal{H}_{\delta}^{s}\left( \limsup_{n \to +\infty} A_{\eta, n} \cap ( 1/(k+1) + \epsilon, 1/k - \epsilon ) \right)\\
&\leq \sum_{n = M}^{+\infty} \lambda \left( 
B\left(T_{1}^{n}(\beta),
\frac{2^{2\cdot(1/\alpha - 1)} \cdot (k+1)^{2/\alpha} \cdot \ln(n)^{1/\alpha}}{\eta^{1/\alpha} \cdot k^{2} \cdot( (r(n) + k + 1 ) \cdot q_{m(n)} + q_{m(n)-1} )^{2\cdot(1/\alpha - 1)}}
\right) \cap (1/(k+1) + \epsilon, 1/k - \epsilon)\right)^{s}\\
&\leq
\frac{2^{s + 2\cdot(1/\alpha - 1)} \cdot (k + 1)^{2\cdot s/\alpha}}{\eta^{s/\alpha} \cdot k^{2\cdot s}}
\sum_{m= m(K)}^{+\infty}
\frac{\ln\left(\sum_{\ell =1}^{m+1}a_{\ell}\right)^{s/\alpha}}{ (q_{m+1}) ^{2\cdot s \cdot(1/\alpha - 1)}}\\
&\leq
\frac{2^{s + 2\cdot(1/\alpha - 1)} \cdot (k + 1)^{2\cdot s/\alpha}}{\eta^{s/\alpha} \cdot k^{2\cdot s}}
\sum_{m= m(K)}^{+\infty}
\frac{\ln (q_{m+1})^{s/\alpha}}{ (q_{m+1})^{2\cdot s \cdot(1/\alpha - 1)}}.
\end{align*}
(In the above we have used that if $y \in [1/( \ell + 2), 1/(\ell + 1)]$, for some $\ell \in \mathbb{N}$, then $T_{1}(y) \in  [1/( \ell + 1), 1/\ell]$.)  This latter infinite sum is finite for all $s > 0$ and $\delta > 0$ since, by the recursive definition of $q_{n}$, we have that $q_{n}$ grows at least at an exponential rate as $n \to +\infty$.  Thus $\mathcal{H}^{s}( \limsup_{n \to +\infty} A_{n, 1, \eta})$ is finite for all $s > 0$. This yields that  $\mathrm{dim}_{\mathcal{H}} ( \limsup_{n \to +\infty} A_{n, 1, \eta} ) = 0$ as required.  (Here $\mathcal{H}^{s}$ denotes the $s$-dimensional Hausdorff measure.)
\end{proof}

\section{Proof of main results}\label{sec:proofs}

\subsection{Proof of Theorem~\ref{thm:Convergence_Equilibrum}}

\begin{proof}[Proof of Theorem~\ref{thm:Convergence_Equilibrum}]
By linearity of the Perron-Frobenius operator we have that
\begin{align*}
\mathcal{P}_{r}^{n}(v) = \mathcal{P}_{r}^{n}(v \cdot \mathds{1}_{[0,1] \setminus [W_{r, n}(\beta)]}) +  \mathcal{P}_{r}^{n}(v \cdot \mathds{1}_{[W_{r, n}(\beta)]})
\end{align*}
where $[W_{r, n}(\beta)]$ is as defined in \eqref{eq_tail_set}.  Further, by Lemma~\ref{lem:bv_part} we have that
\begin{align*}
\lim_{n \to +\infty} \mathcal{P}_{r}^{n}(v \cdot \mathds{1}_{[0,1] \setminus [W_{r, n}(\beta)]}) = \int v \, \mathrm{d}\lambda \cdot h_{r}
\end{align*}
uniformly on $[0, 1]$.  By the facts that $v$ is non-negative and $\mathcal{P}_{r}$ is a positive operator, we have that
\begin{align*}
0
\leq \lim_{n \to \infty} \mathcal{P}_{r}^{n}(v \cdot \mathds{1}_{[W_{r, n}(\beta)]})
\leq \lim_{n \to \infty} \mathcal{P}_{r}^{n}(v_{n, r}),
\end{align*}
where $v_{n, r}$ is as defined in \eqref{eq:vnr}.  By Lemma~\ref{lem:tail_r_neq_1}, this latter limit is equal to zero outside a set of Hausdorff dimension zero.

All that remains to show is that if $\beta \in [0, 1]$ is pre-periodic with respect to $T_{r}$ and has period length strictly greater than one, then on $\Omega_{r}(\beta)$ we have that
\begin{align*}
\qquad \quad
\liminf_{n \to +\infty} \mathcal{P}_{r}^{n}(v)
= \int v \, \mathrm{d}\lambda \cdot h_{r}
\quad \text{and} \quad
\limsup_{n \to +\infty} \mathcal{P}_{r}^{n}(v);
= +\infty;
\end{align*}
and in the case that $\beta \in [0, 1]$ is pre-periodic with respect to $T_{r}$ and has period length equal to one then on the singleton $\Omega_{r}(\beta)$ we have that the limit in \eqref{eq:convergence} is equal to $+\infty$.

By linearity of $\mathcal{P}_{r}^{n}$ and Lemma~\ref{lem:bv_part}, it suffices to show, if $\beta \in [0, 1]$ is pre-periodic with respect to $T_{r}$ and has period length strictly greater than one, then on $\Omega_{r}(\beta)$
\begin{align*}
\liminf_{n \to +\infty} v_{n, r} = 0
\quad \text{and} \quad
\limsup_{n \to +\infty} v_{n, r} = +\infty;
\end{align*}
and in the case that $\beta \in [0, 1]$ is pre-periodic with respect to $T_{r}$ and has period length equal to one, then on the singleton $\Omega_{r}(\beta)$
\begin{align*}
\lim_{n \to +\infty} v_{n, r} = +\infty.
\end{align*}
Indeed if $\beta$ is pre-periodic with respect to $T_{r}$ and has period length $m \geq 1$, then letting $n \in \mathbb{N}_{0}$, be the minimal integer so that $T_{r}^{n + k}(\beta) = T_{r}^{n + k + m}(\beta)$, for all $k \in \mathbb{N}_{0}$, we have that
\begin{align*}
f_{r, (\omega_{r, n + j + 1}(\beta), \dots, \omega_{r, n + j + m}(\beta))}(T_{r}^{n + j}(\beta)) = T^{n + j}_{r}(\beta),
\end{align*}
for all $j \in \{ 0, 1, \dots, m-1 \}$.  Further, $\Omega_{r}(\beta) = \{ T^{n}_{r}(\beta), \dots, T^{n + m -1}_{r}(\beta) \}$, and hence, for $j \in \{ 0, 1, \dots, m-1 \}$, it follows that
\begin{align*}
v_{n + j + k \cdot m, r}(T^{n + j}_{r}(\beta)) = + \infty,
\end{align*}
for all $k \in \mathbb{N}_{0}$.  To complete the proof we will show, for $m > 1$ and $i, j \in \{ 0, 1, \dots, m-1 \}$ with $i \neq j$, that
\begin{align*}
\lim_{k \to +\infty} v_{n + j + k \cdot m, r}(T^{n + i}_{r}(\beta)) = 0.
\end{align*}
To this end set $L \coloneqq \min \{ \lvert T^{n + j}_{r}(\beta) - T^{n + i}_{r}(\beta) \colon i, j \in \{ 0, 1, \dots, m-1 \} \; \text{and} \; i \neq j \}$.  By \eqref{eq:derivative_bounds} and Lemmata~\ref{lem:BD} and \ref{lem:BD_neighbouring_cylinders}, there exists a positive constant $\varrho \in \mathbb{R}$ such that the following chain of inequalities hold.
\begin{align*}
&\lim_{k \to +\infty} v_{n + j + k \cdot m, r}(T^{n + i}_{r}(\beta))\\
&= \lim_{k \to +\infty} \sum_{\omega \in \mathfrak{W}_{r, n+j+k\cdot m}(\beta)} \lvert f_{r, \omega}'(T^{n + i}_{r}(\beta)) \rvert \cdot \lvert \beta - f_{r, \omega}(T^{n + i}_{r}(\beta)) \rvert^{-\alpha}\\
&\leq \lim_{k \to +\infty} 3 \cdot \varrho \cdot \lvert f_{r, \omega_{1}(\beta)\vert_{n+j+k\cdot m}}'(T^{n + i}_{r}(\beta)) \rvert \cdot \lvert \beta - f_{r, \omega_{1}(\beta)\vert_{n+j+k\cdot m}}(T^{n + i}_{r}(\beta)) \rvert^{-\alpha}\\
&\leq \lim_{k \to +\infty} 3 \cdot \varrho \cdot \lvert f_{r, \omega_{1}(\beta)\vert_{n+j+k\cdot m}}'(T^{n + i}_{r}(\beta)) \rvert \cdot \lvert f_{r, \omega_{1}(\beta)\vert_{n+j+k\cdot m}}(T^{n+j+k\cdot m}_{r}(\beta)) - f_{r, \omega_{1}(\beta)\vert_{n+j+k\cdot m}}(T^{n + i}_{r}(\beta)) \rvert^{-\alpha}\\
&\leq \lim_{k \to +\infty} 3 \cdot \varrho^{1+\alpha} \cdot \lvert f_{r, \omega_{1}(\beta)\vert_{n+j+k\cdot m}}'(T^{n + i}_{r}(\beta)) \rvert^{1-\alpha} \cdot \lvert T_{r}^{n+j+k\cdot m}(\beta) - T^{n + i}_{r}(\beta) \rvert^{-\alpha}\\
&= 3 \cdot \varrho^{1+\alpha} \cdot \lvert T_{r}^{n+j}(\beta) - T^{n + i}_{r}(\beta) \rvert^{-\alpha} \lim_{k \to \infty} (2 - r)^{(\alpha-1) \cdot (n + j + k \cdot m)}\\
&= 3 \cdot \varrho^{1+\alpha} \cdot L \lim_{k \to \infty} (2 - r)^{(\alpha-1) \cdot (n + j + k \cdot m)}
= 0.
\end{align*}
This completes the proof.
\end{proof}

\subsection{Proof of Theorems~\ref{thm:Infinite_1} and \ref{thm:Infinite_2}}

\subsubsection{Proof of Theorem~\ref{thm:Infinite_1}}
\mbox{ }

We divide the proof of Theorem~\ref{thm:Infinite_1} into two cases; the first case is when $\beta$ is a rational number and the second case is when $\beta$ is an irrational of intermediate $\alpha$-type.  We emphasise that when $\beta$ is an irrational of intermediate $\alpha$-type, then the method of proof of Theorem~\ref{thm:Infinite_1} is the same as Theorem~\ref{thm:Convergence_Equilibrum}, whereas in the case that $\beta$ is a rational, this method is no longer applicable.

\begin{proof}[Proof of Theorem~\ref{thm:Infinite_1} for $\beta$ rational]
Let $\alpha \in (0, 1)$, $\beta \in (0, 1]$ be a rational number and $v \in \mathfrak{U}_{\beta, \alpha}$.  As $\beta$ is a rational number, there exists a minimal $n \in \mathbb{N}$ such that $T^{n}(\beta) = 0$, let $n$ be fixed as such.  Further, we have that $\Omega_{1}(\beta) = \{ 0 \}$.  We will first prove the result for $\beta \neq 1$.  By definition of the Farey map, there exist exactly two finite words $\eta, \eta' \in \Sigma^{n}$ such that
\begin{enumerate}[label=(\alph*)]
\item $f_{1, \eta}(0) = \beta = f_{1, \eta'}(0)$,
\item $f_{1, \eta}(x) < \beta < f_{1, \eta'}(x)$, for all $x \in (0, 1]$, and
\item $f_{1, \xi}(x) \neq \beta$, for all words $\xi \in \Sigma^{n} \setminus \{ \eta, \eta' \} $ and all $x \in [0, 1]$.
\end{enumerate}
By definition, we have, for $k \in \mathbb{N}$, that
\begin{align*}
\mathcal{P}_{1}^{k}(v)(x) = \sum_{\xi \in \Sigma^{k}} \lvert f_{1, \xi}' \rvert \cdot v \circ f_{1, \xi}.
\end{align*}
Hence, by linearity of the operator $\mathcal{P}_{1}$, we have, for all natural numbers $k > n$, that 
\begin{align*}
\mathcal{P}^{k}_{1}(v)
= \mathcal{P}_{1}^{k-n}(\mathcal{P}_{1}^{n}(v))
&= \mathcal{P}_{1}^{k-n}(\mathcal{P}_{1}^{n}(v \cdot \mathds{1}_{[0, 1] \setminus [\eta] \cap [0, 1] \setminus [\eta']})) + \mathcal{P}_{1}^{k-n}(\mathcal{P}_{1}^{n}(v \cdot \mathds{1}_{[\eta] \cup [\eta']}))\\
&= \mathcal{P}^{k-n}_{1}\left(\textstyle{\sum_{\xi \in \Sigma^{k} \setminus \{\eta, \eta'\}} \lvert f_{1, \xi}' \rvert \cdot v \circ f_{1, \xi}}\right) + \mathcal{P}_{1}^{k-n}(\mathcal{P}_{1}^{n}(v \cdot \mathds{1}_{[\eta] \cup [\eta']})).
\end{align*}
If $\xi \in \{0, 1\}^{n - 1} \setminus \{ \eta, \eta' \}$, then since $\beta \not\in f_{1,\xi}([0, 1])$, since the functions $f_{1,\xi}$, $f_{1, \xi}'$, $1/h_{1}$ are all of bounded variation, since $v \in \mathfrak{U}_{\beta, \alpha}$ and since $[\xi]$ is a compact interval bounded away from $\beta$, by Proposition~\ref{Prop:propertiesBV}, it follows that the function
\begin{align*}
[0, 1] \ni x \mapsto \frac{1}{h_{1}(x)} \sum_{\xi \in \Sigma^{k} \setminus \{\eta, \eta'\}} \lvert f_{1, \xi}'(x) \rvert \cdot v \circ f_{1, \xi}(x)
\end{align*}
is of bounded variation.  Hence, by Proposition~\ref{prop:rmk1} and Theorems~\ref{thm:MT:2001:sec10} and~\ref{prop:jk2011:th7} together with Remark~\ref{rmk:thm:MT:2001:sec10}, we have that
\begin{align*}
\lim_{k \to \infty} \ln(k) \cdot \mathcal{P}_{1}^{k}(v \cdot \mathds{1}_{[0, 1] \setminus [\eta] \cap [0, 1] \setminus [\eta']})
&= \int \mathcal{P}_{1}^{n}(v \cdot \mathds{1}_{[0, 1] \setminus [\eta] \cap [0, 1] \setminus[\eta']}) \, \mathrm{d}\lambda \cdot h_{1}\\
&= \int v \cdot \mathds{1}_{[0, 1] \setminus [\eta] \cap [0, 1] \setminus [\eta']} \, \mathrm{d}\lambda \cdot h_{1}.
\end{align*}
Therefore, to complete the proof we need to show that 
\begin{align*}
\lim_{k \to +\infty} \ln(k) \cdot \mathcal{P}_{1}^{k}(v \cdot \mathds{1}_{[\eta] \cup [\eta']}) = \int v \cdot \mathds{1}_{[\eta] \cup [\eta']} \, \mathrm{d}\lambda \cdot h_{1}.
\end{align*}
To this end let $m > n$ be a fixed natural number satisfying $\lambda([\xi]) \leq \min \{ \lvert a - \beta \rvert, \lvert b - \beta \rvert \}$ for all $\xi \in \Sigma^{m}$, where $U = (a, b)$ is the open connected set such that $C_{1} v_{\beta, \alpha} \leq v \leq C_{2} v_{\beta, \alpha}$ on $U$, for some constants $C_{1}, C_{2}$.  Let $\nu, \nu' \in \Sigma^{m}$ be the unique words satisfying
\begin{align*}
[\nu] \cap [\nu'] = \{ \beta \}, \quad [\nu] \subset [\eta] \quad \text{and} \quad [\nu'] \subset [\eta'].
\end{align*}
Indeed, we necessarily have that $f_{1, \nu}(0) = \beta = f_{1, \nu'}(0)$.  Using identical arguments to those above, we can conclude that
\begin{align*}
\lim_{k \to +\infty} \ln(k) \cdot \mathcal{P}_{1}^{k}(v \cdot \mathds{1}_{[\eta] \setminus [\nu] \cup [\eta'] \setminus [\nu']}) = \int v \cdot \mathds{1}_{[\eta]\setminus [\nu] \cup [\eta'] \setminus [\nu']} \, \mathrm{d}\lambda \cdot h_{1}.
\end{align*}
Moreover, by positivity of the operator $\mathcal{P}_{1}$ we have that
\begin{align*}
C_{1} \mathcal{P}_{1}^{k}(v_{\beta, \alpha} \cdot \mathds{1}_{[\nu] \cup [\nu']})
\leq \mathcal{P}_{1}^{k}(v \cdot \mathds{1}_{[\nu] \cup [\nu']})
\leq C_{2} 
\mathcal{P}_{1}^{k}(v_{\beta, \alpha} \cdot \mathds{1}_{[\nu] \cup [\nu']}).
\end{align*}
We claim (and will shortly prove) that 
\begin{align}\label{eq:rational_remainder}
\lim_{k \to +\infty} \mathcal{P}_{1}^{k}(v_{\beta, \alpha} \cdot \mathds{1}_{[\nu] \cup [\nu']}) = \int v_{\beta, \alpha} \cdot \mathds{1}_{[\nu] \cup [\nu']} \, \mathrm{d} \lambda \cdot h_{1}.
\end{align}
Assuming this, we may conclude, for all $m \in \mathbb{N}$, that 
\begin{align}\label{eq:rational_remainder(i)}
\liminf_{k \to +\infty} \, \mathcal{P}^{k}_{1}(v) &\geq C_{1} \int v_{\beta, \alpha} \cdot \mathds{1}_{[\nu] \cup [\nu']} \, \mathrm{d} \lambda \cdot h_{1} + \int v \cdot \mathds{1}_{[0, 1] \setminus [\nu] \cap [0, 1] \setminus [\nu']} \, \mathrm{d} \lambda \cdot h_{1}
\end{align}
and
\begin{align}\label{eq:rational_remainder(ii)}
\limsup_{k \to +\infty} \, \mathcal{P}_{1}^{k}(v) &\leq
C_{2} \int v_{\beta, \alpha} \cdot \mathds{1}_{[\nu] \cup [\nu']} \, \mathrm{d} \lambda \cdot h_{1} + \int v \cdot \mathds{1}_{[0, 1] \setminus [\nu] \cap [0, 1] \setminus [\nu']} \, \mathrm{d} \lambda \cdot h_{1}.
\end{align}
(Note that the words $\nu, \nu'$ are dependent on $m$.)  Since the LHS of \eqref{eq:rational_remainder(i)} and \eqref{eq:rational_remainder(ii)} are independent of $m$ and since $\lambda(\nu), \lambda(\nu')$ both converge to zero as $n \to +\infty$, the result follows.

We now prove the equality given in \eqref{eq:rational_remainder}.  By Proposition~\ref{prop:rmk1} and Theorems~\ref{thm:MT:2001:sec10} and~\ref{prop:jk2011:th7} together with Remark~\ref{rmk:thm:MT:2001:sec10} it is sufficient to show that
\begin{align*}
[0, 1] \ni x \mapsto \widehat{T}_{1}^{m}(v_{\beta, \alpha} \cdot \mathds{1}_{[\nu] \cup [\nu']}/h_{1})(x)
\end{align*}
is of bounded variation.  In order to show this, recall that $f_{1, \nu}$ and $f_{1,\nu'}$ are M\"obius transformations and observe that 
\begin{align*}
\widehat{T}_{1}^{m}(v_{\beta, \alpha} \cdot \mathds{1}_{[\nu] \cup [\nu']}/h_{1})(x)
= \sum_{i = 1}^{2} \frac{x}{(c_{i} \cdot x + d_{i})^{2}} \left(\frac{(-1)^{i+1}}{\beta - \frac{a_{i} \cdot x + b_{i}}{c_{i}\cdot x + d_{i}}}\right)^{\alpha},
\end{align*}
where $a_{i}, b_{i}, c_{i}, d_{i} \in \mathbb{Z}$, for $i \in \{ 1, 2\}$, are such that
\begin{align*}
f_{\nu}(x) = \frac{a_{1} \cdot x + b_{1}}{c_{1} \cdot x + d_{1}}
\quad \text{and} \quad
f_{\nu'}(x) = \frac{a_{2} \cdot x + b_{2}}{c_{2} \cdot x + d_{2}}.
\end{align*}
The desired conclusion, namely that $\widehat{T}_{1}^{m}(v_{\beta, \alpha} \cdot \mathds{1}_{[\nu] \cup [\nu']}/h_{1})$ is of bounded variation follows from the following four observations.
\begin{enumerate}
\item For all $t \in (0, 1]$, we have that $V_{[t, 1]}(\widehat{T}_{1}^{m}(v_{\beta, \alpha} \cdot \mathds{1}_{[\nu] \cup [\nu']}/h_{1})) < + \infty$.
\item For $i \in \{ 1, 2\}$, by L'H\^opital's rule we have that
\begin{align*}
\lim_{x \to 0} \frac{(-1)^{i+1} \cdot x}{\beta - \frac{a_{i}\cdot x + b_{i}}{c_{i}\cdot x + d_{i}}} = d_{i}^{2}.
\end{align*}
\item By L'H\^opital's rule, we have that
\begin{align*}
\qquad
\lim_{x \to 0} \widehat{T}_{1}^{m}(v_{\beta, \alpha} \cdot \mathds{1}_{[\nu] \cup [\nu']}/h_{1})(x)
= \sum_{i = 1}^{2} \lim_{x \to 0} \frac{x}{(c_{i} \cdot x + d_{i})^{2}} \left(\frac{(-1)^{i+1}}{\beta - \frac{a_{i}\cdot x + b_{i}}{c_{i}\cdot x + d_{i}}}\right)^{\alpha}
= 0.
\end{align*}
\item We have that
\begin{align*}
\qquad \quad
&\frac{\mathrm{d}}{\mathrm{dx}} \widehat{T}_{1}^{m}(v_{\beta, \alpha} \cdot \mathds{1}_{[\nu] \cup [\nu']}/h_{1})(x)\\
&= \sum_{i = 1}^{2} \frac{\mathrm{d}}{\mathrm{dx}} \frac{x}{(c_{i} \cdot x + d_{i})^{2}} \left(\frac{(-1)^{i+1}}{\beta - \frac{a_{i} \cdot x + b_{i}}{c_{i} \cdot x + d_{i}}}\right)^{\alpha}\\
&= \sum_{i = 1}^{2} \frac{-c_{i} \cdot x + d_{i}}{(c_{i} \cdot x + d_{i})^{3}} \left(\frac{(-1)^{i+1}}{\beta - \frac{a_{i} \cdot x + b_{i}}{c_{i} \cdot x + d_{i}}}\right)^{\alpha}
-
\frac{(-1)^{i+1} \cdot \alpha \cdot x}{(c_{i} \cdot x + d_{i})^{4}}\left(\frac{(-1)^{i+1}}{\beta - \frac{a_{i} \cdot x + b_{i}}{c_{i} \cdot x + d_{i}}}\right)^{\alpha + 1}
\end{align*}
which is non-negative on an open neighbourhood of zero.
\end{enumerate}
The case when $\beta = 1$ is a simplification of the above case.
\end{proof}

\begin{proof}[Proof of Theorem~\ref{thm:Infinite_1} for $\beta$ irrational of intermediate $\alpha$-type]
By linearity of the Perron-Frobenius operator we have that
\begin{align*}
\ln(n) \cdot \mathcal{P}_{1}^{n}(v)
= \ln(n) \cdot \mathcal{P}_{1}(v \cdot \mathds{1}_{[0, 1] \setminus [\omega_{1}(\beta)\vert_{n}]})
+ \ln(n) \cdot \mathcal{P}_{1}(v \cdot \mathds{1}_{[\omega_{1}(\beta)\vert_{n}]}).
\end{align*}
Further, by Lemma~\ref{lem:bv_part_r=1} and the fact that $h_{1} \cdot \widehat{T}_{1}(f) = \mathcal{P}_{1}(f \cdot h_{1})$, we have that
\begin{align*}
\lim_{n \to \infty} \ln(n) \cdot \widehat{T}_{1}^{n}(v \cdot \mathds{1}_{[0, 1] \setminus [\omega_{1}(\beta)\vert_{n}]} / h_{1}) = \int v \; \mathrm{d}\lambda \cdot h_{1}
\end{align*}
uniformly on compact subsets of $(0, 1)$.  Moreover, by the facts that $v \in \mathfrak{U}_{\beta, \alpha}$ is non-negative and $\mathcal{P}_{1}$ is a positive linear operator, there exists a positive constant $C$ with
\begin{align*}
0
\leq \lim_{n \to \infty} \ln(n) \cdot \mathcal{P}_{1}^{n}(v \cdot \mathds{1}_{[\omega_{1}(\beta)\vert_{n}]})
\leq \lim_{n \to \infty} \ln(n) \cdot C \cdot \mathcal{P}_{r}^{n}(v_{n, 1}),
\end{align*}
where we recall that $v_{n, 1} = v_{\beta, \alpha} \cdot \mathds{1}_{[\omega_{1}(\beta)\vert_{n}]}$.  By Lemma~\ref{lem:tail_r_=_1}, this latter limit is equal to zero outside a set of Hausdorff zero.

All that remains to show is that if $\beta \in (0, 1]$ is irrational, pre-periodic with respect to $T_{1}$ and has period length strictly greater than one, then on $\Omega_{1}(\beta)$ we have that
\begin{align*}
\qquad \quad
\liminf_{n \to +\infty} \ln(n) \cdot \mathcal{P}_{1}^{n}(v)
= \int v \, \mathrm{d}\lambda \cdot h_{1}
\quad \text{and} \quad
\limsup_{n \to +\infty} \ln(n) \cdot \mathcal{P}_{1}^{n}(v)
= +\infty;
\end{align*}
and in the case that $\beta \in (0, 1]$ is pre-periodic with respect to $T_{1}$ and has period length equal to one then on the singleton $\Omega_{1}(\beta)$ we have that the limit in \eqref{eq:convergence} is equal to $+\infty$.

By positivity and linearity of $\mathcal{P}_{1}^{n}$ and Lemma~\ref{lem:bv_part_r=1}, it suffices to show, if $\beta \in (0, 1]$ is irrational, pre-periodic with respect to $T_{1}$ and has period length strictly greater than one, then on $\Omega_{1}(\beta)$,
\begin{align*}
\liminf_{n \to +\infty} \ln(n) \cdot v_{n, 1} = 0
\quad \text{and} \quad
\limsup_{n \to +\infty} \ln(n) \cdot v_{n, 1} = +\infty;
\end{align*}
and in the case that $\beta \in (0, 1]$ is pre-periodic with respect to $T_{1}$ and has period length equal to one, then on the singleton $\Omega_{1}(\beta)$,
\begin{align*}
\lim_{n \to +\infty} \ln(n) \cdot v_{n, 1} = +\infty.
\end{align*}
Indeed if $\beta$ is pre-periodic with respect to $T_{1}$ and has period length $l \geq 1$, then letting $n \in \mathbb{N}_{0}$, be the minimal integer so that $T_{1}^{n + k}(\beta) = T_{1}^{n + k + l}(\beta)$, for all $k \in \mathbb{N}_{0}$, we have that
\begin{align*}
f_{1, (\omega_{1, n + j + 1}(\beta), \dots, \omega_{1, n + j + l}(\beta))}(T_{1}^{n + j}(\beta)) = T^{n + j}_{1}(\beta),
\end{align*}
for all $j \in \{ 0, 1, \dots, l-1 \}$.  Further, $\Omega_{1}(\beta) = \{ T^{n}_{1}(\beta), \dots, T^{n + l -1}_{1}(\beta) \}$, and hence, for $j \in \{ 0, 1, \dots, l-1 \}$, it follows that
\begin{align*}
v_{n + j + k \cdot l, 1}(T^{n + j}_{1}(\beta)) = + \infty,
\end{align*}
for all $k \in \mathbb{N}_{0}$.  To complete the proof we will show, for $l > 1$ and $i, j \in \{ 0, 1, \dots, l-1 \}$ with $i \neq j$, that
\begin{align*}
\lim_{k \to +\infty} v_{n + j + k \cdot l, 1}(T^{n + i}_{1}(\beta)) = 0.
\end{align*}
To this end set $L \coloneqq \min \{ \lvert T^{n + j}_{1}(\beta) - T^{n + i}_{1}(\beta) \colon i, j \in \{ 0, 1, \dots, l-1 \} \; \text{and} \; i \neq j \}$ and set
\begin{align*}
a \coloneqq \min \{ T^{n + j}_{1}(\beta) \colon j \in \{ 0, 1, \dots, l-1 \} \}
\quad \text{and} \quad
b \coloneqq \max \{ T^{n + j}_{1}(\beta) \colon j \in \{ 0, 1, \dots, l-1 \} \}.
\end{align*}
Since $\beta$ is irrational and pre-periodic with period $m > 1$, it follows that $0 < a < b < 1$ and therefore,
\begin{align*}
\lvert f_{1, \omega_{1}(\beta)\vert_{n + j + k \cdot l}}'(T^{n + i}_{1}(\beta)) \rvert \leq a^{-2} ((r(n + j + k \cdot l)+1) q_{m(n + j + k \cdot l)} + q_{m(n + j + k \cdot l) - 1})^{-2}
\end{align*}
for all $i, j \in \{ 0, 1, \dots, l-1 \}$ and $k \in \mathbb{N}$.  Further, we have that
\begin{align*}
\lvert \beta - f_{1, \omega_{1}(\beta)\vert_{n + j + k \cdot l}}(T^{n + i}_{1}(\beta)) \rvert
&\geq \lvert f_{1, \omega_{1}(\beta)\vert_{n + j + k \cdot l}}(T^{n + j + k \cdot l}_{1}(\beta)) - f_{1, \omega_{1}(\beta)\vert_{n + j + k \cdot l}}(T^{n + i}_{1}(\beta)) \rvert\\
&\geq \inf_{u \in [a, b]} \lvert f_{1, \omega_{1}(\beta)\vert_{n + j + k \cdot l}}'(u) \rvert \cdot \rvert T^{n + j}_{1}(\beta) - T^{n + i}_{1}(\beta) \rvert\\
&\geq ((r(n + j + k \cdot l)+1) q_{m(n + j + k \cdot l)} + q_{m(n + j + k \cdot l) - 1})^{-2} \cdot L,
\end{align*}
for all $i, j \in \{ 0, 1, \dots, l-1 \}$ with $i \neq j$ and $k \in \mathbb{N}$.  Hence, for all $i, j \in \{ 0, 1, \dots, l-1 \}$ with $i \neq j$, we have
\begin{align*}
0 \leq \lim_{l \to +\infty} v_{n + j + l \cdot m, 1}(T^{n + i}_{1}(\beta))
&\leq \lim_{l \to +\infty} \lvert f_{1, \omega_{1}(\beta)\vert_{n + j + l \cdot m}}'(T^{n + i}_{1}(\beta)) \rvert \cdot \lvert \beta - f_{1, \omega_{1}(\beta)\vert_{n + j + l \cdot m}}(T^{n + i}_{1}(\beta)) \rvert^{-\alpha}\\
&\leq \lim_{l \to +\infty} a^{-2} \cdot L^{-\alpha} \cdot ((r(n + j + k \cdot l)+1) q_{m(n + j + k \cdot l)} + q_{m(n + j + k \cdot l) - 1})^{2\cdot(\alpha - 1)}\\
&= 0.
\end{align*}
This completes the proof.
\end{proof}

\subsubsection{Proof of Theorem~\ref{thm:Infinite_2}}

\begin{proof}[Proof of Theorem~\ref{thm:Infinite_2}\ref{3.3.a}]
Within this proof set
\begin{align*}
\beta = [0; \! \underbrace{1, 1,}_{2 \cdot1} \! 2, \underbrace{1, 1, 1, 1,}_{2 \cdot 2} 2, \underbrace{1, 1, 1, 1, 1, 1,}_{2\cdot3} 2, \dots]
\quad \text{and} \quad
\kappa = [0; \! \underbrace{1, 1,}_{2^{1}} \! 2, \underbrace{1, 1, 1, 1,}_{2^{2}} 2, \underbrace{1, 1, \dots, 1,}_{2^{3}} 2, \dots]
\end{align*}
and, for $n \in \mathbb{N}$, set
\begin{align*}
\Lambda(n, \tau) \coloneqq \begin{cases}
n\cdot(n+2) & \text{if} \; \tau = \beta,\\
2^{n}+n-2 & \text{if} \; \tau = \kappa.
\end{cases}
\end{align*}
Observe that $\beta, \kappa \in [1/2, 1]$.  Letting $a_{n}(\beta)$ and $a_{n}(\kappa)$ denote the $n$-th continued fraction entry of $\beta$ and $\kappa$ respectively, an elementary calculation yields that $a_{\Lambda(n, \beta)-1}(\beta) = a_{\Lambda(n, \kappa)-1}(\kappa) = 2$.  Further, one can show that
\begin{align*}
\Omega_{1}(\beta) = \Omega_{1}(\kappa) = \{ [0; \underbrace{1, 1, \dots, 1,}_{k} 2, \overline{1}] \colon k \in \mathbb{N}_{0} \} \cup \{ \gamma \coloneqq (\sqrt{5}-1)/2 = [0; \overline{1}]\}.
\end{align*}
Recall from \eqref{eq:vnr} that $v_{\tau, \alpha, n, 1} = \lvert f_{1, \omega_{1}(\tau)\vert_{n}}' \rvert \cdot \lvert \tau - f_{1, \omega_{1}(\tau)\vert_{n}} \rvert^{-\alpha}$.  Following the same arguments as in beginning of the proof of Theorem~\ref{thm:Infinite_1}, it is sufficient to show, on $\Omega_{1}(\beta) = \Omega_{1}(\kappa)$, that
\begin{align}\label{eq:aim_bounded}
\limsup_{n \to +\infty} \ln(n) \cdot v_{\beta, \alpha, n, 1} = 0
\quad \text{and} \quad
\limsup_{n \to +\infty} \ln(n) \cdot v_{\kappa, \alpha , n, 1} = \begin{cases}
0 & \text{if} \; \alpha \in (0, 1/2),\\
+\infty & \text{if} \; \alpha \in (1/2, 1].
\end{cases}
\end{align}
To this end fix $k \in \mathbb{N}_{0}$ and set 
\begin{align*}
\zeta_{k} \coloneqq [0; \underbrace{1, 1, \dots, 1,}_{k} 2, \overline{1}] \in [1/3, 1].
\end{align*}
We will show that the equalities given in \eqref{eq:aim_bounded} hold for $\zeta_{k}$, the result for $\gamma$ is a simplification of this case.  To this end let $\tau \in \{ \beta, \kappa \}$.  By the mean value theorem, for each $n \in \mathbb{N}$, there exists $u_{n}(\tau) \in (1/3, 1)$ such that
\begin{align*}
\lvert \tau - f_{1, \omega_{1}(\tau)\vert_{n}}(\zeta_{k})\rvert
&= \lvert T_{1}^{n}(\tau) - \zeta_{k} \rvert \cdot \lvert f_{1, \omega_{1}(\tau)\vert_{n}}'(u_{n}(\tau))\rvert\\
&= \lvert T_{1}^{n}(\tau) - \zeta_{k} \rvert \cdot ( (r(n,\tau) u_{n}(\tau) + 1 ) q_{m(n, \tau)}(\tau) +  q_{m(n, \tau)-1}(\tau) u_{n}(\tau) )^{-2}\\
&\begin{cases}
\geq 5^{-2} \cdot (q_{m(n, \tau)}(\tau))^{-2} \cdot \lvert T^{n}_{1}(\tau) - \zeta_{k} \rvert,\\
\leq (q_{m(n, \tau)}(\tau))^{-2} \cdot \lvert T^{n}_{1}(\tau) - \zeta_{k} \rvert,
\end{cases}
\end{align*}
where $m(n, \tau)$ and $r(n, \tau)$ are as defined in \eqref{eq:def_m_r_k} and where, for $l \in \mathbb{N}_{0}$, the integers $p_{l}(\tau)$ and $q_{l}(\tau)$ are as defined in \eqref{eq:definition_p_q}.  Thus, for $\tau \in \{ \beta, \kappa \}$ and $k \in \mathbb{N}_{0}$, we have that
\begin{align*}
&\limsup_{n \to \infty} \ln(n) \cdot v_{\tau, \alpha, n, 1}(\zeta_{k})\\
&= \limsup_{n \to \infty} \frac{\ln(n)}{((r(n, \tau) \cdot \zeta_{k} +1) \cdot q_{m(n, \tau)}(\tau) + q_{m(n, \tau) - 1} \cdot \zeta_{k})^{2}} \frac{1}{\lvert \tau - f_{1, \omega_{1}(\tau)\vert_{n}}(\zeta_{k}) \rvert^{\alpha}}\\
&\begin{cases}
\geq \displaystyle{\limsup_{n \to \infty} \frac{\ln(n)}{5^{2} \cdot (q_{m(n, \tau)}(\tau))^{2\cdot(1-\alpha)}} \frac{1}{\lvert T_{1}^{n}(\tau) - \zeta_{k} \rvert^{\alpha}}}\\[1em]
\leq \displaystyle{\limsup_{n \to \infty} \frac{5^{2\cdot \alpha} \cdot \ln(n)}{(q_{m(n, \tau)}(\tau))^{2\cdot(1-\alpha)}} \frac{1}{\lvert T_{1}^{n}(\tau) - \zeta_{k} \rvert^{\alpha}}}
\end{cases}\\
&\begin{cases}
\geq \displaystyle{\limsup_{n \to \infty} \frac{\ln(n)}{5^{2} \cdot (q_{m(n,\tau)}(\tau))^{2\cdot(1-\alpha)}} \frac{1}{\lvert T_{1}^{n - (k+1)}(\tau) - \gamma \rvert^{\alpha}} \frac{1}{\lvert (f_{1,1}^{k} \circ f_{1,0} \circ f_{1,1})'(0) \rvert^{\alpha}}}\\[1em]
\leq \displaystyle{\limsup_{n \to \infty} \frac{5^{2 \cdot \alpha} \cdot \ln(n)}{(q_{m(n,\tau)}(\tau))^{2\cdot(1-\alpha)}} \frac{1}{\lvert T_{1}^{n - (k+1)}(\tau) - \gamma \rvert^{\alpha}} \frac{1}{\lvert (f_{1,1}^{k} \circ f_{1,0} \circ f_{1,1})'(1) \rvert^{\alpha}}}.
\end{cases}
\end{align*}
Hence it is sufficient to show that, for $\alpha \in (0, 1)$,
\begin{align}\label{eq:beta-con}
\limsup_{n \to +\infty} \frac{\ln(n)}{(q_{m(n,\beta)}(\beta))^{2\cdot(1-\alpha)}} \frac{1}{\lvert T_{1}^{n - (k+1)}(\beta) - \gamma \rvert^{\alpha}} = 0
\end{align}
and 
\begin{align}\label{eq:varrho-con}
\limsup_{n \to +\infty} \frac{\ln(n)}{(q_{m(n,\kappa)}(\kappa))^{2\cdot(1-\alpha)}} \frac{1}{\lvert T_{1}^{n - (k+1)}(\kappa) - \gamma \rvert^{\alpha}} = \begin{cases}
0 & \text{if} \; \alpha \in (0, 1/2),\\
+\infty & \text{if} \; \alpha \in (1/2, 1).
\end{cases}
\end{align}
We will first show the equality given in \eqref{eq:beta-con} after which we will show the equality given in \eqref{eq:varrho-con}.  For this observe that if $n - (k + 1) = \Lambda(l, \beta) + (l -1)$, for some $l \in \mathbb{N}$, then
\begin{align*}
T_{1}^{n - (k+1)}(\beta) = [0; 2, \underbrace{1, 1, \dots, 1,}_{2 \cdot (l+1)} 2, \underbrace{1, 1, \dots, 1,}_{2\cdot(l+2)} 2, \underbrace{1, 1, \dots, 1,}_{2\cdot(l+3)} 2, \dots] \in [1/3, 1/2],
\end{align*}
and hence,
\begin{align}\label{eq:bounded_1}
\begin{aligned}
\frac{\ln(n)}{(q_{m(n,\beta)}(\beta))^{2\cdot(1-\alpha)}} \frac{1}{\lvert T_{1}^{n - (k+1)}(\beta) - \gamma \rvert^{\alpha}}
&\leq \frac{\ln(\Lambda(l, \beta) + (l -1)+(k+1))}{(q_{\Lambda(l, \beta)}(\beta))^{2\cdot(1-\alpha)}} \frac{1}{\lvert (1/2) - \gamma \rvert^{\alpha}}\\
&\sim \frac{2 \cdot \ln(l)}{(q_{l \cdot (l+2)}(\beta))^{2\cdot(1-\alpha)}} \frac{1}{\lvert (1/2) - \gamma \rvert^{\alpha}}.
\end{aligned}
\end{align}
Since the sequence $(q_{j})_{j \in \mathbb{N}}$ grows exponentially, this latter term converges to zero as $l \to \infty$.  (Here we have used the fact that $n - (k + 1) = \Lambda(l, \beta) + (l -1)$.)

In the case that $n - (k + 1) \not\in \{ \Lambda(j, \beta) + (j -1) \colon j \in \mathbb{N}\}$, set $l = l(n) \in \mathbb{N}$ to be the maximal integer such that $n - (k + 1) > \Lambda(l, \beta) + (l -1)$, in which case
\begin{align*}
T_{1}^{n - (k+1)}(\beta) = [0; \underbrace{1, 1, \dots\dots\dots\dots\dots, 1,}_{\substack{3 \cdot (l+1) + (k+1) + \Lambda(l,\beta) - n \\[0.25em] \leq 2 \cdot (l+1) + 1}} 2, \underbrace{1, 1, \dots, 1,}_{2\cdot(l+2)} 2, \underbrace{1, 1, \dots, 1,}_{2\cdot(l+3)} 2, \dots],
\end{align*}
and hence,
\begin{align}\label{eq:bounded_2}
\begin{aligned}
&\frac{\ln(n)}{(q_{m(n,\beta)}(\beta))^{2\cdot(1-\alpha)}} \frac{1}{\lvert T_{1}^{n - (k+1)}(\beta) - \gamma \rvert^{\alpha}}\\
&= \frac{\ln(n)}{(q_{m(n,\beta)}(\beta))^{2\cdot(1-\alpha)}} \frac{1}{\lvert f_{1, 1}^{3 \cdot (l+1) + (k+1) + \Lambda(l,\beta) - n}(T_{1}^{\Lambda(l+1,\beta) + l}(\beta)) - f_{1, 1}^{3 \cdot (l+1) + (k+1) + \Lambda(l,\beta) - n}(\gamma) \rvert^{\alpha}}\\
&\leq \frac{\ln((l+2)\cdot(l+5))}{(q_{l\cdot(l+2)}(\beta))^{2\cdot(1-\alpha)}} \frac{1}{\inf_{u \in [0, 1]} \lvert (f_{1, 1}^{3 \cdot (l+1) + (k+1) + \Lambda(l,\beta) - n})'(u) \rvert^{\alpha}} \frac{1}{\lvert (1/2) - \gamma \rvert^{\alpha}}\\
&= \frac{\ln((l+2)\cdot(l+5))}{(q_{l\cdot(l+2)}(\beta))^{2\cdot(1-\alpha)}} \frac{(q_{3 \cdot (l+1) + (k+1) + \Lambda(l,\beta) - n}(\gamma))^{\alpha}}{\lvert (1/2) - \gamma \rvert^{\alpha}}\\
&= \frac{\ln((l+2)\cdot(l+5))}{(q_{l\cdot(l+2)}(\beta))^{2\cdot(1-\alpha)}} \frac{(q_{2\cdot(l+1)+1}(\beta))^{\alpha}}{\lvert (1/2) - \gamma \rvert^{\alpha}}.
\end{aligned}
\end{align}
Since the sequence $(q_{j}(\beta))_{j \in \mathbb{N}}$ grows exponentially, this latter term converges to zero as $l = l(n) \to \infty$.  The equality stated in \eqref{eq:beta-con} now follows from \eqref{eq:bounded_1} and \eqref{eq:bounded_2}.

We will now prove the equality given in \eqref{eq:varrho-con}.  The result for, $\alpha \in (0, 1/2)$, follows in a similar manner to the previous case.  Indeed, observe that if if $n - (k + 1) = \Lambda(l, \kappa) + (l -1)$, for some $l \in \mathbb{N}$, then
\begin{align*}
T_{1}^{n - (k+1)}(\kappa) = [0; 2, \underbrace{1, 1, \dots, 1,}_{2^{l+1}} 2, \underbrace{1, 1, \dots, 1,}_{2^{l+2}} 2, \underbrace{1, 1, \dots, 1,}_{2^{l+3}} 2, \dots] \in [1/3, 1/2],
\end{align*}
and hence, for $n$ sufficiently large,
\begin{align}\label{eq:bounded_3}
\begin{aligned}
\frac{\ln(n)}{(q_{m(n,\kappa)}(\kappa))^{2\cdot(1-\alpha)}} \frac{1}{\lvert T_{1}^{n - (k+1)}(\kappa) - \gamma \rvert^{\alpha}}
\leq \frac{(l + 1) \cdot \ln(2)}{(q_{2^{l}}(\kappa))^{2\cdot(1-\alpha)}} \frac{1}{\lvert (1/2) - \gamma \rvert^{\alpha}}.
\end{aligned}
\end{align}
The sequence $(q_{j}(\kappa))_{j \in \mathbb{N}}$ grows exponentially, in particular there exists a positive constant $c$ so that $\kappa^{-j}/c \leq q_{j}(\kappa) \leq c \cdot \kappa^{-j}$. Therefore, the latter term in \eqref{eq:bounded_3} converges to zero as $l \to \infty$.  (Here we have used the fact that $n - (k + 1) = \Lambda(l, \kappa) + (l -1)$.)

In the case that $n - (k + 1) \not\in \{ \Lambda(j, \kappa) + (j -1) \colon j \in \mathbb{N}\}$, set $l = l(n) \in \mathbb{N}$ to be the maximal integer such that $n - (k + 1) > \Lambda(l, \kappa) + (l -1)$, in which case
\begin{align*}
T_{1}^{n - (k+1)}(\kappa) = [0; \!\underbrace{1, 1, \dots\dots\dots\dots\dots, 1,}_{\substack{2^{l+1} + (l+1) + (k+1) + \Lambda(l,\kappa) - n\\[0.25em] \leq 2^{l+1} + 1}} \! 2, \underbrace{1, 1, \dots, 1,}_{2^{l+2}} 2, \underbrace{1, 1, \dots, 1,}_{2^{l+3}} 2, \dots].
\end{align*}
We also observe that $q_{i}(\gamma) \leq q_{i}(\kappa)$, for all $i \in \mathbb{N}_{0}$.  Therefore, it follows that
\begin{align}\label{eq:bounded_4}
\begin{aligned}
\frac{\ln(n)}{(q_{m(n,\kappa)}(\kappa))^{2\cdot(1-\alpha)}} \frac{1}{\lvert T_{1}^{n - (k+1)}(\kappa) - \gamma \rvert^{\alpha}}
&\leq \frac{(l+2) \cdot \ln(2)}{(q_{2^{l}}(\kappa))^{2\cdot(1-\alpha)}} (q_{2\cdot2^{l} +2}(\gamma))^{\alpha}\\
&\leq \frac{(l+2) \cdot \ln(2)}{(q_{2^{l}}(\gamma))^{2\cdot(1-\alpha)} \cdot (q_{2\cdot2^{l} +2}(\gamma))^{-\alpha}}.
\end{aligned}
\end{align}
Since there exists a positive constant $c$ so that $\gamma^{-j}/c \leq q_{j}(\gamma) \leq c \cdot \gamma^{-j}$, if $\alpha \in (0, 1/2)$, this latter term converges to zero as $l = l(n) \to \infty$.  The equality in \eqref{eq:varrho-con} for $\alpha \in (0, 1/2)$ follows from \eqref{eq:bounded_3} and \eqref{eq:bounded_4}.

Let us now examine the case  that $\alpha \in (1/2, 1)$.  It follows from an inductive argument that, for all $n \in \mathbb{N}$, $q_{l}(\kappa) \leq 2^{n} \cdot q_{l}(\gamma)$ for all integers $l \in [\Lambda(n, \kappa), \Lambda(n+1, \kappa))$.  Further, for all $n \in \mathbb{N}$ we have that
\begin{enumerate}
\item $\displaystyle{\lvert \gamma - T_{1}^{\Lambda(n, \kappa)+n-1}(\kappa) \rvert
= \lvert \gamma - [0; 2, \underbrace{1, \dots, 1,}_{2^{n+1}} 2, \underbrace{1, \dots, 1,}_{2^{n+2}} 2\dots] \rvert \geq \lvert \gamma - (1/2) \rvert} \;$
and
\item $\displaystyle{\lvert \gamma - T_{1}^{\Lambda(n, \kappa)+n+1}(\kappa) \rvert
= \lvert \gamma - [0; \underbrace{1, \dots, 1,}_{2^{n}} 2, \underbrace{1, \dots, 1,}_{2^{n+1}} 2, \dots] \rvert
\leq \left\lvert \gamma - \frac{p_{2^{n}}(\gamma)}{q_{2^{n}}(\gamma)} \right\rvert
\leq \frac{1}{(q_{2^{n}}(\gamma))^{2}}}$.
\end{enumerate}
Therefore, if $\alpha \in (1/2, 1)$, since there exists a positive constant $\mathfrak{s}$ so that $\gamma^{-n} / \mathfrak{s} \leq q_{n}(\gamma) \leq \mathfrak{s} \cdot \gamma^{-n}$, for all $n \in \mathbb{N}$, we have that
\begin{align*}
\limsup_{n \to +\infty} \frac{\ln(\Lambda(n,\kappa)+n+1)}{(q_{\Lambda(n,\kappa)}(\kappa))^{2\cdot(1-\alpha)}} \frac{1}{\lvert T_{1}^{\Lambda(n,\kappa)+n+1}(\kappa) - \gamma \rvert^{\alpha}}
&\geq \limsup_{n \to +\infty} \frac{n \cdot \ln(2) \cdot (q_{2^{n}}(\gamma))^{2\cdot\alpha}}{2^{2\cdot n\cdot(1-\alpha)}\cdot(q_{2^{n}+n-2}(\gamma))^{2\cdot(1-\alpha)}}\\
&\geq \limsup_{n \to +\infty} \frac{n \cdot \ln(2)}{\mathfrak{s}^{2} \cdot \gamma^{2^{n+1} \cdot (1-2\cdot\alpha) + 2\cdot(3\cdot n - 2)\cdot(1 - \alpha)}}
= + \infty.
\end{align*}
Moreover, since the sequence $(q_{j}(\kappa))_{j \in \mathbb{N}}$ grows exponentially, it follows that
\begin{align*}
\liminf_{n \to +\infty} \frac{\ln(\Lambda(n,\kappa)+n-1)}{(q_{\Lambda(n,\kappa)-1}(\kappa))^{2\cdot(1-\alpha)}} \frac{1}{\lvert T_{1}^{\Lambda(n,\kappa)+n-1}(\kappa) - \gamma \rvert^{\alpha}}
\leq \liminf_{n \to +\infty} \frac{\ln(\Lambda(n,\kappa)+n-1)}{(q_{\Lambda(n,\kappa)-1}(\kappa))^{2\cdot(1-\alpha)}} \frac{1}{\lvert \gamma - (1/2) \rvert^{\alpha}} = 0.
\end{align*}
This completes the proof.
\end{proof}

\begin{proof}[Proof of Theorem~\ref{thm:Infinite_2}\ref{3.3.b}]
Since $\lim_{n \to + \infty} a_{n} = +\infty$, we have that $\Omega_{1}(\beta) = \{ 1/k \colon k \in \mathbb{N} \} \cup \{ 0 \}$.  Let $v_{\beta, \alpha, n, 1}$ be as in \eqref{eq:vnr}.  Following the same arguments as in beginning of the proof of Theorem~\ref{thm:Infinite_1}, it is sufficient to show, for a fixed $k \in \mathbb{N}$, that
\begin{align*}
\limsup_{n \to +\infty} \ln(n) \cdot v_{\beta, \alpha, n, 1}(1/k)
\begin{cases}
= 0 & \text{if} \; \displaystyle{\limsup_{j \to \infty}  \mathscr{S}_{k, j} = 0},\\
> 0 & \text{if} \; \displaystyle{\limsup_{j \to \infty} \mathscr{S}_{k, j} > 0}.
\end{cases}
\quad \text{and} \quad 
\liminf_{n \to +\infty} \ln(n) \cdot v_{\beta, \alpha, n, 1}(1/k) = 0.
\end{align*}
To this end fix $k \in \mathbb{N}$ and, for $n \in \mathbb{N}$, set $z = z(n) \coloneqq T_{1}^{n}(\beta)$.  (Note, $z$ is the unique real number in $[0, 1]$ such that $f_{1, \omega_{1}(\beta)\vert_{n}}(z) = \beta$.)  If $z \in (1/(k+1), 1/k)$, then, by the mean value theorem, there exists $u = u(n) \in (1/(k+1), 1/k)$ such that
\begin{align*}
\lvert \beta - f_{1, \omega_{1}(\beta)\vert_{n}}(1/k)\rvert
&= \lvert f_{1, \omega_{1}(\beta)\vert_{n}}(z) - f_{1, \omega_{1}(\beta)\vert_{n}}(1/k)\rvert\\
&= \lvert 1/k - z \rvert \cdot \lvert f_{1, \omega_{1}(\beta)\vert_{n}}'(u)\rvert\\
&= \lvert 1/k - T^{n}_{1}(\beta) \rvert \cdot \lvert (r(n) \cdot u + 1 ) \cdot q_{m(n)} +  q_{m(n)-1} \cdot u \rvert^{-2}\\
&\begin{cases}
\geq k^{2} \cdot \lvert 1/k - T^{n}_{1}(\beta) \rvert \cdot \lvert (r(n) + k) \cdot q_{m(n)} +  q_{m(n)-1} \rvert^{-2},\\
\leq (k+1)^{2} \cdot \lvert 1/k - T^{n}_{1}(\beta) \rvert \cdot \lvert (r(n) + k + 1) \cdot q_{m(n)} +  q_{m(n)-1} \rvert^{-2}.
\end{cases}
\end{align*}
If $z \not\in (1/(k+1), 1/k)$, then, since $f_{1, \omega_{1}(\beta)\vert_{n}}$ is either order preserving or order reversing, we have for $n \in \mathbb{N}$ sufficiently large that
\begin{align*}
\lvert \beta - f_{1, \omega_{1}(\beta)\vert_{n}}(1/k)\rvert
&= \lvert f_{1, \omega_{1}(\beta)\vert_{n}}(z) - f_{1, \omega_{1}(\beta)\vert_{n}}(1/k)\rvert\\
&\geq \begin{cases}
\lvert f_{1, \omega_{1}(\beta)\vert_{n}}(1/2) - f_{1, \omega_{1}(\beta)\vert_{n}}(1)\rvert & \quad\text{if} \; k = 1,\\
\min\{ \lvert f_{1, \omega_{1}(\beta)\vert_{n}}(1/(k+1)) - f_{1, \omega_{1}(\beta)\vert_{n}}(1/k)\rvert, &\\
\hspace{1.9em}\lvert f_{1, \omega_{1}(\beta)\vert_{n}}((2k-1)/(2k(k-1))) - f_{1, \omega_{1}(\beta)\vert_{n}}(1/k)\rvert \} & \quad \text{otherwise}.
\end{cases}
\end{align*}
By the mean value theorem there exists $u \in (1/(k+1), (2k-1)/(2k(k-1)))$ if $k \neq 1$ and $u \in (1/2, 1)$ if $k = 1$ such that
\begin{align*}
\lvert \beta - f_{1, \omega_{1}(\beta)\vert_{n}}(x)\rvert
&\geq (2 \cdot k \cdot (k + 1))^{-1} \cdot \lvert f_{1, \omega_{1}(\beta)\vert_{n}}'(u)\rvert\\
&= (2 \cdot k \cdot (k+1))^{-1} \cdot \rvert (r(n) \cdot u + 1 ) \cdot q_{m(n)} +  q_{m(n)-1} \cdot u \rvert^{-2}\\
&\geq (3 \cdot 2\cdot k)^{-1} \cdot \lvert (r(n) + \max\{ k-1, 1\})\cdot q_{m(n)} +  q_{m(n)-1} \rvert^{-2}.
\end{align*}
We now consider the following two cases $z \not\in (1/(k+1), 1/k)$ and $z \in (1/(k+1), 1/k)$.
\begin{enumerate}
\item If $z \not\in (1/(k+1), 1/k)$, then
\begin{align*}
0
\leq \ln(n) \cdot v_{\beta, \alpha, n, 1}(1/k)
&= \frac{\ln(n)}{( (r(n)/k + 1 ) \cdot q_{m(n)} +  q_{m(n)-1}/k )^{2}} \, \frac{1}{\lvert \beta - f_{1, \omega_{1}(\beta)\vert_{n}}(1/k) \rvert^{\alpha}}\\
&\leq \frac{6^{2 \cdot \alpha} \cdot k^{2\cdot(1-\alpha)} \cdot \ln(n)}{((r(n) + 1) \cdot q_{m(n)} + q_{m(n)-1})^{2\cdot(1-\alpha)}}.
\end{align*}
Since $(r(n) + 1) \cdot q_{m(n)} + q_{m(n)-1} > n$, for all $n \in \mathbb{N}$, it follows that
\begin{align*}
\qquad \quad
\liminf_{n \to +\infty} \ln(n) \cdot v_{\beta, \alpha, n,1}(1/k) = 0.
\end{align*}
\item If $z \in (1/(k+1), 1/k)$, then $z = T_{1}^{n}(\beta) = [0; k, a_{m(n)}, a_{m(n)+1}, \dots ]$; that is $n = n_{k, m(n)}$. Thus, we have that
\begin{align*}
\qquad \quad
&\limsup_{j \to +\infty} \ln(n_{k, j}) \cdot v_{\beta, \alpha, n_{k, j},1}(1/k)\\
&= \limsup_{j \to +\infty} \frac{k^{2} \cdot \ln(n_{k, j})}{( (r(n_{k, j}) + k ) \cdot  q_{m(n_{k, j})} +  q_{m(n_{k, j})-1} )^{2}} \, \frac{1}{\lvert \beta - f_{1, \omega_{1}(\beta)\vert_{n_{k, j}}}(1/k) \rvert^{\alpha}}\\[0.5em]
&\begin{cases}
\displaystyle{\leq \limsup_{j \to +\infty} \frac{k^{2\cdot(1-\alpha)} \cdot \ln(n_{k, j})}{\lvert 1/k - T_{1}^{n_{k, j}}(\beta) \rvert^{\alpha} \cdot  ( (r(n_{k, j}) + k ) \cdot  q_{m(n_{k, j})} +  q_{m(n_{k, j})-1} )^{2\cdot(1-\alpha)}}}\\[1em]
\displaystyle{\geq \limsup_{j \to +\infty} \frac{k^{2\cdot(1+\alpha)} \cdot \ln(n_{k, j})}{2^{2\cdot \alpha} \cdot \lvert 1/k - T_{1}^{n_{k, j}}(\beta) \rvert^{\alpha} \cdot  ( (r(n_{k, j}) + k) \cdot  q_{m(n_{k, j})} +  q_{m(n_{k, j})-1} )^{2\cdot(1-\alpha)}}
}\\[1em]
\end{cases}\\[0.5em]
&\begin{cases}
\displaystyle{\leq \limsup_{j \to +\infty} \frac{k^{2} \cdot (a_{j+1} + 1)^{\alpha} \cdot \ln(n_{k, j})}{(q_{j})^{2\cdot(1-\alpha)}}}\\[1em]
\displaystyle{\geq \limsup_{j \to +\infty} \frac{k^{2(1+2\cdot\alpha)} \cdot (a_{j+1})^{\alpha} \cdot \ln(n_{k, j})}{2^{2\cdot\alpha}\cdot (q_{j})^{2\cdot(1-\alpha)}}}\\[1em]
\end{cases}\\[0.5em]
&\begin{cases}
= \displaystyle{\limsup_{j \to +\infty} k^{2} \cdot \mathscr{S}_{k, j}}\\[1em]
= \displaystyle{\limsup_{j \to +\infty} k^{2\cdot(1+2\cdot\alpha)} \cdot 4^{-\alpha} \cdot \mathscr{S}_{k, j}.}
\end{cases}
\end{align*}
\end{enumerate}
This completes the proof.
\end{proof}

\section{Proof of Proposition~\ref{prop:rmk1}}\label{sec:appendix}

\begin{proof}[Proof of Proposition~\ref{prop:rmk1} - Conditions (R1) and (R2)]
By definition of the norm $\lVert \cdot \rVert_{\mathrm{BV}(Y)}$ and since the support of any function $f \in \mathrm{BV}(Y)$ is a subset of $Y = [1/2, 1]$, we have that if $f \in \mathrm{BV}(Y)$ then $f \in \mathcal{L}^{1}_{1}([0, 1])$ and that $\lVert \cdot \rVert_{\mathcal{L}^{\infty}} \leq \lVert \cdot \rVert_{\infty} \leq \lVert \cdot \rVert_{\mathrm{BV}(Y)}$.  Thus it remains to show that $R(1)(f) \in \mathrm{BV}(Y)$, for all $f \in \mathrm{BV}(Y)$.  To this end let $f \in \mathrm{BV}(Y)$ be fixed. By \cite[Proposition 1]{S:2002} and \cite[Proposition 1 (p.\ 33)]{JK:2011}, we have that $R(1)$ is a positive linear operator and that
\begin{align}\label{eq:induced_operator}
\int R(1)(w) \cdot u \, \mathrm{d}\mu_{1} = \int w \cdot u \circ T_{1}^{\phi_{Y}} \, \mathrm{d}\mu_{1},
\end{align}
for all $w \in \mathcal{L}_{1}^{1}(\mu_{1}\vert_{Y})$ and $u \in \mathcal{L}_{1}^{\infty}(Y)$.  Hence, by Propositions~\ref{Prop:propertiesBV}(\ref{BV:value}), for all $x \in [0, 1]$ we have that
\begin{align}\label{eq:DF1}
\begin{aligned}
\lvert R(1)(f) (x) \rvert 
\leq 2\int R(1)( \lvert f \rvert ) \, \mathrm{d}\mu_{1} + V_{Y}(R(1)(f))
= 2  \cdot \lVert f \rVert_{1, 1} + V_{Y}(R(1)(f)).
\end{aligned}
\end{align}
Thus, it suffice to show that the variation of $R(1)(f)$ is bounded. Moreover by Proposition~\ref{Prop:propertiesBV-Complex} we may assume that $f$ is $\mathbb{R}$-valued.  In order to do this we will use the fact that $R(1)$ is a positive linear operator, \eqref{eq:induced_operator} and Proposition~\ref{Prop:propertiesBV}.  Observe, for $k \in \mathbb{N}$, that $U_{k} \coloneqq \overline{\{ y \in Y \colon \phi_{Y}(y) = k \}} = [k/(k+1), (k+1)/(k+2)]$ and let $\tau = \tau_{Y}$ be as in Propositions~\ref{Prop:propertiesBV}(\ref{BV:alt}).   For $\psi \in \tau$ and, for $k \in \mathbb{N}$, let $g_{k}, \psi_{k} \colon [0, 1] \to \mathbb{R}$ denote the functions
\begin{align*}
g_{k}(x) &\coloneqq 
\begin{cases}
-k \cdot x +2  \cdot k - 1 - (k-1) / x & \text{if} \; x \in U_{k},\\
0 & \text{otherwise},
\end{cases}
\quad \text{and} \quad
\psi_{k}(x) &\coloneqq 
\begin{cases}
-\psi \circ T_{1}^{k}(x) & \text{if} \; x \in U_{k} \setminus \partial U_{k},\\
0 & \text{otherwise.}
\end{cases}
\end{align*}
Indeed, on the interior of $U_{k}$, we have that $g_{k} = -T_{1}^{k} \cdot h_{1}/(T_{1}^{k})'$.  Via elementary calculations, one can conclude, for $k \in \mathbb{N}$, that $\psi_{k} \in \tau_{_{U_{k}}}$, that the function $g_{k}$ is continuous on $U_{k}$ and that
\begin{align*}
\left\lVert g_{k}\lvert_{U_{k}} \right\rVert_{\infty} =
\begin{cases}
1/2 & \text{if} \, k = 1,\\
3 - 2^{3/2} & \text{if} \, k = 2,\\
2/((k+1) \cdot (k+2)) & \text{otherwise,}
\end{cases}
\, \text{and} \;\,
V_{U_{k}}(g_{k}) = \begin{cases}
1/6 & \text{if} \, k = 1,\\
17/3 - 2^{5/2} & \text{if} \, k = 2,\\
(k-2)/(k \cdot (k+1) \cdot (k+2)) & \text{otherwise}.
\end{cases}
\end{align*}
Hence, we have that
\begin{align*}
\int R(1)(f)(x)  \cdot \psi'(x) \, \mathrm{d}\lambda(x) 
&= \int f(x) \cdot T_{1}^{\phi_{Y}}(x) \cdot \psi' \circ T_{1}^{\phi_{Y}}(x) \, \mathrm{d}\mu(x)\\
&= \sum_{k = 1}^{+\infty} \int \mathds{1}_{U_{k}}(x) \cdot f(x) \cdot g_{k}(x) \cdot \psi_{k}'(x) \, \mathrm{d}\lambda(x)\\
&\leq \sum_{k = 1}^{+\infty} V_{U_{k}}( g_{k} ) \cdot \lVert f \rVert_{\infty} + V_{U_{k}}( f ) \cdot \lVert g_{k} \rVert_{\infty}\\
&\leq (6 - 2^{5/2})  \cdot \lVert f \rVert_{\infty} + (V_{Y}( f ))/2.
\end{align*}
In particular, setting $c \coloneqq 6 - 2^{5/2}$, we have that
\begin{align}\label{eq:DF}
V_{Y}(R(1)(f)) \leq c \cdot \lVert f \rVert_{\infty} + (V_{Y}(f))/2,
\end{align}
for all $f \in \mathrm{BV}(Y)$.  Combing this with \eqref{eq:DF1} yields that
\begin{align*}
\lVert R(1)(f)\rVert_{\mathrm{BV}(Y)}
= \lVert R(1)(f)\rVert_{\infty} + V_{Y}(R(1)(f))
&\leq 2 \cdot \lVert f \rVert_{1, 1} + 2 \cdot V_{Y}(R(1)(f))\\
&\leq 2 \cdot \lVert f \rVert_{1, 1} + \lVert f \rVert_{\mathrm{BV}(Y)}
< + \infty,
\end{align*}
which completes the proof.
\end{proof}

\begin{remark}
Since $R(1) \coloneqq \sum_{n = 1}^{+\infty} R_{n}$, as a corollary to Condition (R3), we obtain an alternative proof to the fact that $R(1)(f) \in \mathrm{BV}(Y)$ for all $f \in \mathrm{BV}(Y)$.  However, the above calculations will be extremely useful in the proof of Condition (R4).
\end{remark}

\begin{proof}[Proof of Proposition~\ref{prop:rmk1} - Condition (R3)]
Since $\widehat{T}_{1}$ is a linear operator, powers of $\widehat{T}_{1}$ are linear operators and so, $R_{n}$ is a linear operator, for all $n \in \mathbb{N}$.  We will now show that the operator norm of $R_{n}\vert_{\mathrm{BV}(Y)}$ is bounded above by $8  \cdot \mu_{1}( \{ y \in Y \colon \phi_{Y}(y) = n \} )$.  We will prove the result for integers $n \geq 3$, an explicit calculation will yield the result for $n \in \{ 1, 2 \}$.  To this end let $n \geq 3$ denote a fixed integer.  Recall, for $n \in \mathbb{N}$, that $U_{n} \coloneqq \{ y \in Y \colon \phi_{Y}(y) = n \} = f_{1, 1} \circ f_{1, 0}^{n-1}([0, 1))$.  The representation of $\widehat{T}_{1}$ given in \eqref{eqn:FareyDual} together with an inductive argument yields, for $f \in \mathrm{BV}(Y)$, that
\begin{align*}
\widehat{T}_{1}^{n}(\mathds{1}_{U_{n}} \cdot f)
= f_{1, 0}^{n} \cdot \prod_{k = 0}^{n-2} f_{1, 1}\circ f_{1, 0}^{k} \cdot \mathds{1}_{[1/2, 1)} \cdot \left(f \circ f_{1, 1} \circ f_{1, 0}^{n-1}\right).
\end{align*}
Since, for $k \in \mathbb{N}$ and $x \in [0, 1]$,
\begin{align}\label{eq:contractions_induced}
f_{1, 0}^{k}(x) = \frac{x}{1 + k \cdot x}
\quad \text{and} \quad
f_{1, 1}\circ f_{1, 0}^{k}(x) = \frac{1+k \cdot x}{1+(k+1) \cdot x},
\end{align}
it follows that 
\begin{align*}
\rVert \mathds{1}_{[1/2, 1)} \cdot f_{1, 0}^{k} \rVert_{\infty} = \frac{1}{1 + k}
\quad \text{and} \quad
\rVert \mathds{1}_{[1/2, 1)} \cdot f_{1, 1}\circ f_{1, 0}^{k} \rVert_{\infty} = \frac{2 + k}{2 + (k+1)},
\end{align*}
and hence, that
\begin{align*}
\left\lVert \mathds{1}_{[1/2, 1)}\cdot f_{0}^{n} \cdot  \prod_{k = 0}^{n-2} f_{1}\circ f_{0}^{k} \right\rVert_{\infty}
\leq \frac{1}{1+n} \prod_{k = 0}^{n-2} \frac{2 + k}{2 + (k+1)}
= \frac{2}{(n+1)^{2}}.
\end{align*}
Moreover, since $f_{1, 1} \circ f_{1, 0}^{k-1}$ is a positive monotonic decreasing contracting $C^{1}$-function, we have that 
\begin{align*}
V_{Y}(f \circ f_{1, 1} \circ f_{1, 0}^{n-1}) \leq V_{Y}(f)
\quad \text{and} \quad
V\left(\prod_{k = 0}^{n-2} f_{1, 1}\circ f_{1, 0}^{k}\right) \leq \left\lVert \prod_{k = 0}^{n-2} f_{1, 1} \circ f_{1, 0}^{k} \right\rVert_{\infty} \leq \frac{2}{n+1}.
\end{align*}
This in tandem with Proposition~\ref{Prop:propertiesBV}(\ref{BV:Lin+prod}) implies, for a $\mathbb{R}$-valued function $f \in \mathrm{BV}(Y)$, that
\begin{align*}
&\rVert R_{n} (f)\lVert_{\mathrm{BV}(Y)}\\
&= \left \lVert \mathds{1}_{Y} \cdot \widehat{T}^{n}(\mathds{1}_{U_{k}} \cdot f) \right\rVert_{\infty} + V\left(\mathds{1}_{Y} \cdot\widehat{T}^{n}(\mathds{1}_{U_{k}} \cdot f) \right)\\
&\leq \left\lVert \mathds{1}_{[1/2, 1)} \cdot f_{1, 0}^{n} \cdot  \prod_{k = 0}^{n-2} f_{1, 1}\circ f_{1, 0}^{k} \cdot \left(f \circ f_{1, 1} \circ f_{1, 0}^{n-1}\right) \right\rVert_{\infty} +V \left( \mathds{1}_{[1/2, 1)} \cdot f_{1, 0}^{n} \cdot  \prod_{k = 0}^{n-2} f_{1, 1}\circ f_{1, 0}^{k} \cdot \left(f \circ f_{1, 1} \circ f_{1, 0}^{n-1}\right)\right)\\
&\leq \lVert \mathds{1}_{[1/2, 1)} \cdot f_{1, 0}^{n} \rVert_{\infty} \cdot \left\lVert \mathds{1}_{[1/2, 1)} \cdot  \prod_{k = 0}^{n-2} f_{1, 1}\circ f_{1, 0}^{k} \right\rVert_{\infty} \cdot \left( 2 \cdot \lVert f \rVert_{\infty} + V_{Y}(f) \right)\\
&\leq 4 \cdot (n+1)^{-3} \cdot \lVert f \rVert_{\mathrm{BV}(Y)}.
\end{align*}
It now follows from linearity of the operator $R(1)$, the triangle inequality and Proposition~\ref{Prop:propertiesBV-Complex}(\ref{BV:Complex3}), that  $\rVert R_{n} (f)\lVert_{\mathrm{BV}(Y)} \leq 8 \cdot (n+1)^{-3} \lVert f \rVert_{\mathrm{BV}(Y)}$, for all $f \in \mathrm{BV}(Y)$.  Finally, observe that
\begin{align*}
\mu_{1}(U_{n})
= \int \mathds{1}_{U_{n}}(x) \cdot x^{-1} \, \mathrm{d}\lambda(x)
= \ln\left( 1 + \frac{1}{n  \cdot (n+2)} \right)
\geq \frac{1}{n \cdot (n+2)} - \frac{1}{2  \cdot n^{2}  \cdot (n+2)^{2}} \geq \frac{1}{(n+1)^{3}}.
\end{align*}
This completes the proof.
\end{proof}

In order to prove condition (R4) we will use the following theorem (a generalisation of earlier results by Doeblin and Fortet \cite{DF:1937} and Ionescu-Tulcea and Marinescu \cite{ITM:1950}), which gives sufficient criterion for an operator to be quasi-compact.

\begin{definition}[Quasi-compact]
A bounded linear operator $L$ on a Banach space $\mathfrak{L}$ with spectral radius $\rho(L)$ is called \textit{quasi-compact} if there is a direct sum decomposition $\mathfrak{L} = \mathfrak{F} \oplus \mathfrak{H}$ and $0 < \rho < \rho(L)$ where
\begin{enumerate}
\item $\mathfrak{F}, \mathfrak{H}$ are closed and $L$-invariant, that is, $L(\mathfrak{H}) \subseteq \mathfrak{H}$ and $L(\mathfrak{F}) \subset \mathfrak{F}$,
\item $\mathfrak{F}$ is finite dimensional and all eigenvalues of $L\lvert_{\mathfrak{F}} \colon \mathfrak{F} \circlearrowleft$ have modulus larger than $\rho$ and
\item the spectral radius of $L\lvert_{\mathfrak{H}} \colon \mathfrak{H} \circlearrowleft$ is smaller than $\rho$.
\end{enumerate}
\end{definition}

\begin{theorem}[{\cite[Theorem XIV.3]{HH:2001}}]\label{thm:Herrion}
Suppose that $(\mathfrak{L}, \lVert \cdot \rVert_{\mathfrak{L}})$ is a Banach space and $L \colon \mathfrak{L} \circlearrowleft$ is a bounded linear operator with spectral radius $\rho(L)$.  Assume that there exists a semi-norm $\lVert \cdot \rVert'_{\mathfrak{L}}$ with the following properties.
\begin{description}
\item[\hspace{1em}Continuity] The semi-norm $\lVert \cdot \rVert'_{\mathfrak{L}}$ is continuous on $\mathfrak{L}$.
\item[\hspace{1em}Pre-compactness] For a sequence $( f_ {n})_{n \in \mathbb{N}}$ in $\mathfrak{L}$, if $\sup_{n \in \mathbb{N}} \lVert f_{n} \rVert_{\mathfrak{L}} < +\infty$, then there exists a subsequence $( n_{k} )_{k \in \mathbb{N}}$ of $\mathbb{N}$ and $g \in \mathfrak{L}$ with $\lim_{k \to +\infty} \lVert L(f_{n_{k}}) - g \rVert'_{\mathfrak{L}} = 0$.
\item[\hspace{1em}Boundedness] There exists $M > 0$ such that $\lVert L(f) \rVert'_{\mathfrak{L}} \leq \lVert f \rVert'_{\mathfrak{L}}$, for all $f \in \mathfrak{L}$.
\item[\hspace{1em}Doeblin-Fortet Inequality] There exist $k \in \mathbb{N}$, $r \in (0, \rho(L))$ and $R \geq 0$ so that, for all $f \in \mathfrak{L}$,
\begin{align*}
\lVert L^{k}(f) \rVert_{\mathfrak{L}} \leq r^{k} \cdot \lVert f \rVert_{\mathfrak{L}} + R \cdot \lVert f \rVert'_{\mathfrak{L}}.
\end{align*}
\end{description}
Under these conditions the operator $L \colon \mathfrak{L} \circlearrowleft$ is quasi-compact.
\end{theorem} 

\begin{proof}[Proof of Proposition~\ref{prop:rmk1} - Condition (R4)]
Recal that $h_{1}(x) = 1/x$ and that, for all $k \in \mathbb{N}$,
\begin{align*}
U_{k} \coloneqq \overline{\{ y \in Y \colon \phi_{Y}(y) = k \}} = [k/(k+1), (k+1)/(k+2)].
\end{align*}
Let $\mathrm{int}(Y)$ denote the interior of $Y$.  By definition and utilising \eqref{eq:contractions_induced}, we conclude that 
\begin{align*}
R(1)(\mathds{1}_{\mathrm{int}(Y)})(x)
&= \lim_{m \to + \infty} \sum_{k = 1}^{m} \mathds{1}_{\mathrm{int}(Y)}(x) \cdot \widehat{T}_{1}^{k}(\mathds{1}_{U_{k}})(x)\\
&= \lim_{m \to + \infty} \sum_{k = 1}^{m} \mathds{1}_{\mathrm{int}(Y)}(x) \cdot x \cdot \mathcal{P}_{1}^{k}(\mathds{1}_{U_{k}} \cdot h_{1})(x)\\
&= \lim_{m \to + \infty} \sum_{k = 1}^{m} \mathds{1}_{\mathrm{int}(Y)}(x) \frac{x}{(1 + (k-1)\cdot x) \cdot (1 + k \cdot x)}\\
&= \lim_{m \to + \infty} \sum_{k = 1}^{m} \mathds{1}_{\mathrm{int}(Y)}(x) \cdot x \cdot \left( \frac{1-k}{(1 + (k-1)\cdot x)} + \frac{k}{(1 + k \cdot x)} \right)
= \mathds{1}_{\mathrm{int}(Y)}(x).
\end{align*}
Hence, the function $\mathds{1}_{\mathrm{int}(Y)}$ is an eigenfunction of the operator $R(1)$ with eigenvalue one and therefore the spectral radius $\rho(R(1)\vert_{\mathrm{BV}(Y)})$ of $R(1)$ restricted to the Banach space $\mathrm{BV}(Y)$ is equal to $1$.  In order to show that $1$ is an isolated eigenvalue it is sufficient to show that $R(1)$ is quasi-compact.  By Theorem~\ref{thm:Herrion}, this follows from the following four properties.
\begin{description}
\item[\hspace{1em}Continuity] Let $(f_{n})_{n \in \mathbb{N}}$ denote a convergent sequence in $\mathrm{BV}(Y)$ and denote its limit by $f \in \mathrm{BV}(Y)$.  By the definition of $\lVert \cdot \rVert_{\mathrm{BV}(Y)}$, we have that $\lim_{n \to +\infty} \lVert f_{n} - f \rVert_{\infty} = 0$ and hence
\begin{align*}
\qquad
\lim_{n \to +\infty} \lVert f_{n} - f \rVert_{1, 1}
\leq \lim_{n \to +\infty} \int \lVert f_{n} - f \rVert_{\infty} \; \mathrm{d}\mu_{1}
= \lim_{n \to +\infty} \ln(2) \cdot \lVert f_{n} - f \rVert_{\infty}
= 0.
\end{align*}
\item[\hspace{1em}Pre-compactness] From \eqref{eq:induced_operator} one can deduce that
\begin{align*}
\qquad
\lVert R(1)(f) \rVert_{\mathcal{L}_{1}^{1}(Y)} = \lVert f \rVert_{\mathcal{L}_{1}^{1}(Y)}
\end{align*}
Therefore, by linearity of the operator $R(1)$, Egrov's theorem \cite[Theorem 2.2.1]{Bo:2007}, Proposition~\ref{Prop:propertiesBV}(\ref{BV:difference}) and Proposition~\ref{Prop:propertiesBV-Complex}(\ref{BV:Complex2}) and (\ref{BV:Complex3}), it is sufficient to show the following.  Given a sequence $( f_{n} \colon Y \to \mathbb{R} )_{n \in \mathbb{N}}$ of non-decreasing, non-negative functions which are bounded everywhere such that there exists a constant $M$ with $\lVert f_{n} \rVert_{\mathrm{BV}(Y)} = 2 \lVert f_{n} \rVert_{\infty} \leq M$, then there exists a monotonic subsequence $( n_{k} )_{k \in \mathbb{N}}$ of $\mathbb{N}$ such that the sequence $( f_{n_{k}} )_{n_{k} \in \mathbb{N}}$ converges to a function $f$, with finite $\mathrm{BV}(Y)$-norm, point-wise almost everywhere.  (We recall, by the definition of $\mathrm{BV}(Y)$, that the functions $f_{n}$ and $f$ are right-continuous.)  To this end let $R$ denote a countable dense subset of $Y$ and let $\{ r_{k} \}_{k \in \mathbb{N}}$ be an enumeration of $R$.  Since the sequence $\{ f_{n}(r_{1}) \}_{n \in \mathbb{N}}$ is a bounded subsequence, by the Bolzano-Weierstra{\ss} theorem, there exists an accumulation point $j_{1} \in [0, M/2]$ and a monotonic sequence of natural numbers $( n^{_{(1)}}_{k} )_{k \in \mathbb{N}}$ so that $\lim_{k \to +\infty} f_{n^{_{(1)}}_{k}}(r_{1}) = j_{1}$.  The same argument applied to the sequence $(f_{n^{_{(1)}}_{k}}(r_{2}) )_{k \in \mathbb{N}}$ produces an accumulation point $j_{2} \in [0, M/2]$ and a monotonic sequence $( n^{_{(2)}}_{k} )_{k \in \mathbb{N}}$ of natural numbers so that $\lim_{k \to +\infty} f_{n^{_{(2)}}_{k}}(r_{2}) = j_{2}$.  Continuing this procedure \textit{ad infinitum} leads to a sequence of points $( j_{k} )_{k \in \mathbb{N}}$, which belong to the interval $[0, M/2]$, and a nested sequence of monotonic subsequences $( (n^{_{(m)}}_{k})_{k \in \mathbb{N}} )_{m \in \mathbb{N}}$ of the natural numbers such that for all $m \in \mathbb{N}$,
\begin{align*}
\qquad
\lim_{k \to +\infty} f_{n^{_{(m)}}_{k}}(r_{i}) = j_{i},
\end{align*}
for all $i \in \{ 1, 2, 3, \dots, m \}$.  We will show that there exists a positive function $f \colon Y \to \mathbb{R}$ with $\lVert f \rVert_{\mathrm{BV}(Y)} \leq M$ which is the almost everywhere point-wise limit of the sequence of functions $( f_{n^{_{(k)}}_{k}} )_{k \in \mathbb{N}}$.  Define
\begin{align*}
\qquad
f(x) \coloneqq
\begin{cases}
\displaystyle{\lim_{k \to +\infty} f_{n^{_k}_{k}}(x)} & \text{if} \; x \in R\\
\displaystyle{\lim_{r \downarrow x; \; r \in R} f(r)} & \text{if} \; x \in Y \setminus R.
\end{cases} 
\end{align*}
This is well defined since, for all $k \in \mathbb{N}$, the function $f_{n^{_{(k)}}_{k}}$ is right-continuous, non-decreasing, non-negative and bounded above by $M/2$ everywhere, and so, on $R$ the function $f$ is right-continuous, non-decreasing, non-negative and bounded above by $M/2$.  Therefore, we have that $\lVert f \rVert_{BV} = 2 \, \lVert f \rVert_{\infty} \leq M$, in particular that $f$ is of bounded variation and so differentiable almost everywhere, and hence, continuous almost everywhere.  Let $U$ denote the set of points where $f$ is discontinuous.  If $x \in R\setminus U$, then the point-wise convergence follows by construction.  If $x \in Y \setminus (R \cup U)$, then since $f$ is continuous on this set, we have that
\begin{align*}
\qquad
f(x)
= \lim_{\substack{y \uparrow x; \\ y \in Y \setminus (U \cup R)}} f(y)
= \lim_{\substack{y \uparrow x; \\ y \in Y \setminus (U \cup R)}} \; \lim_{\substack{r \downarrow y; \\ r \in R }} \;\; \liminf_{k \to +\infty} f_{n^{_{(k)}}_{k}}(r)
\leq \liminf_{k \to +\infty} f_{n^{_{(k)}}_{k}}(x)
\end{align*}
and that
\begin{align*}
\qquad
f(x)
= \lim_{\substack{y \downarrow x; \\ y \in Y \setminus (U \cup R)}} f(y)
= \lim_{\substack{y \downarrow x; \\ y \in Y \setminus (U \cup R)}} \; \lim_{\substack{r \downarrow y; \\ r \in R }} \;\; \limsup_{k \to +\infty} f_{n^{_{(k)}}_{k}}(r)
\geq \limsup_{k \to +\infty} f_{n^{_{(k)}}_{k}}(x).
\end{align*}
Thus the limit $\lim_{k \to +\infty} f_{n^{_{(k)}}_{k}}(x)$ exists and equals $f(x)$ for all $x \in Y \setminus U$.
\item[\hspace{1em}Boundedness] Indeed, as mentioned above, from \eqref{eq:induced_operator} one can deduce that $\lVert R(1) \rVert_{\mathcal{L}_{1}^{1}(Y)} = 1$.
\item[\hspace{1em}Doeblin-Fortet Inequality] By \eqref{eq:DF1} and \eqref{eq:DF}, setting $c = 6 - 2^{5/2}$, for a $\mathbb{R}$-valued $f \in \mathrm{BV}(Y)$,
\begin{align*}
\qquad
\lVert R(1)^{2}(f) \rVert_{\mathrm{BV}(Y)} &\leq 2 \cdot \lVert f \rVert_{1, 1} + 2 \cdot V_{Y}(R(1)^{2}(f))\\
&\leq 2 \cdot \lVert f \rVert_{1, 1} + 2 \cdot c \cdot \lVert R(1)(f) \rVert_{\infty} + V_{Y}(R(1)(f))\\
&\leq 2 \cdot (1 + c) \cdot \lVert f \rVert_{1, 1} + (c + 1) \cdot V_{Y}(R(1)(f)\\
&\leq 2 \cdot (1 + c) \cdot \lVert f \rVert_{1, 1} + c \cdot (c + 1) \cdot \lVert f \rVert_{\infty} + (1/2) \cdot (c + 1) \cdot V_{Y}(f)\\
&\leq 2 \cdot (1 + c) \cdot \lVert f \rVert_{1, 1} + (1/2) \cdot (c + 1) \cdot \lVert f \rVert_{\mathrm{BV}(Y)},
\end{align*}
and hence,
\begin{align*}
\qquad
\lVert R(1)^{4}(f) \rVert_{\mathrm{BV}(Y)}
\leq (2 \cdot (1 + c) + (1+c)^{2})\cdot  \lVert f \rVert_{1,1} + (1/4) \cdot (c + 1)^{2} \cdot \lVert f \rVert_{\mathrm{BV}(Y)}.
\end{align*}
Using Proposition~\ref{Prop:propertiesBV-Complex}(\ref{BV:Complex3}), if $f \in \mathrm{BV}(Y)$ is $\mathbb{C}$-valued, then 
\begin{align*}
\qquad
\lVert R(1)^{4}(f) \rVert_{\mathrm{BV}(Y)}
\leq 2\cdot(2\cdot(1 + c) + (1+c)^{2}) \cdot \lVert f \rVert_{1, 1} + (1/2)\cdot(c + 1)^{2} \cdot \lVert f \rVert_{\mathrm{BV}(Y)}.
\end{align*}
Noting that $(1/2)(c + 1)^{2} < 1$ yields the required inequality.
\end{description}
\end{proof}

\begin{proof}[Proof of Proposition~\ref{prop:rmk1} - Condition (R5)]
For $z \in \mathbb{D}\setminus \mathbb{S}$, we define the operator $T(z) \colon \mathcal{L}^{1}_{1}(Y) \circlearrowleft$ by
\begin{align*}
T(z)(f) \coloneqq \sum_{n = 1}^{+\infty} z^{n} \cdot \mathds{1}_{Y} \cdot \widehat{T}^{n}(\mathds{1}_{Y} \cdot f).
\end{align*}
By \cite[Proposition 1]{S:2002} we have that
\begin{align*}
R(z) \circ T(z)(f) = T(z)(f) - f   = T(z) \circ R(z)(f).
\end{align*}
This implies that $1$ does not belong to the spectrum of the operator $R(z)$.  Hence, it is sufficient to show the result for $z \in \mathbb{S} \setminus \{ 1 \}$.  For this, we will follow the arguments given in the proof of \cite[Lemma 6.7]{G:2004}.  To this end let $t \in (0, 2\pi)$ and let $z = \mathrm{e}^{i \cdot t}$ be fixed.  Suppose, by way of contradiction, that $R(z)(f) = f$ for some non-zero $f \in \mathrm{BV}(Y)$.  Let $\mathcal{L}_{1}^{2}(Y)$ denote the space of $\mathbb{C}$-valued square integrable functions with respect to the measure $\mu_{1}$ that have domain $[0, 1]$ and are supported on $Y$.  Further, let $\langle \cdot, \cdot \rangle$ denote the associated bilinear form.  Define the operator $W \colon \mathcal{L}^{\infty}(Y) \circlearrowleft$, by
\begin{align*}
W(u) \coloneqq \mathrm{e}^{-i \cdot t \cdot \phi_{Y}} \cdot u \circ T_{1}^{\phi_{Y}} 
\end{align*}
for $u \in \mathcal{L}^{\infty}(Y)$.  Using the fact that $R(z)(v) = R(1)(\mathrm{e}^{i \cdot t \cdot \phi_{Y}} \cdot v)$ with \eqref{eq:induced_operator}, for all $v \in \mathrm{BV}(Y)$ and $u \in \mathcal{L}^{\infty}(Y)$,
\begin{align*}
\langle u, R(z)(v) \rangle
= \int \overline{u} \cdot R(z)(v) \, \mathrm{d}\mu_{1}
= \int \overline{u} \cdot R(1)(\mathrm{e}^{i \cdot t \cdot \phi_{Y}} \cdot v) \, \mathrm{d}\mu_{1}
= \int \overline{u} \circ T_{1}^{\phi_{Y}} \cdot \mathrm{e}^{i\cdot t\cdot \phi_{Y}} \cdot v \, \mathrm{d}\mu_{1}
= \langle W(u), v \rangle,
\end{align*}
and thus,
\begin{align}\label{eq:Wf-f}
\begin{aligned}
\lVert W(f) - f \rVert_{\mathcal{L}^{2}_{1}(Y)}^{2}
&= \lVert W(f) \rVert_{\mathcal{L}^{2}_{1}(Y)}^{2} - 2 \cdot \mathfrak{Re}\, \langle W(f), f \rangle + \lVert f \rVert_{\mathcal{L}^{2}_{1}(Y)}^{2}\\
&= \lVert W(f) \rVert_{\mathcal{L}^{2}_{1}(Y)}^{2} - 2 \cdot \mathfrak{Re}\, \langle f, R(z)(f) \rangle + \lVert f \rVert_{\mathcal{L}^{2}_{1}(Y)}^{2}\\
&= \lVert W(f) \rVert_{\mathcal{L}^{2}_{1}(Y)}^{2} - 2 \cdot \mathfrak{Re}\, \langle f, f \rangle + {\lVert f \rVert_{2}}^{2}\\
&= \lVert W(f) \rVert_{\mathcal{L}^{2}_{1}(Y)}^{2} - \lVert f \rVert_{\mathcal{L}^{2}_{1}(Y)}^{2},
\end{aligned}
\end{align}
By another application of \eqref{eq:induced_operator}, we also have that
\begin{align}\label{eq:Wf=f}
\lVert W(f) \rVert_{\mathcal{L}^{2}_{1}(Y)}^{2}
= \int \lvert f \rvert^{2} \circ T_{1}^{\phi_{Y}} \, \mathrm{d}\mu_{1}
= \int \lvert f \rvert^{2} \, \mathrm{d}\mu_{1}
= \lVert f \rVert_{\mathcal{L}^{2}_{1}(Y)}^{2}.
\end{align}
From \eqref{eq:Wf-f} and \eqref{eq:Wf=f}, we obtain that $W(f) - f$ is zero $\mu_{1}$-almost everywhere.  Since by definition of $\mathrm{BV}(Y)$, we have that $f$ is right-continuous, $W(f)$ is right-continuous, and so the function $W(f) - f$ is zero everywhere.

We now have a right-continuous function $f$ so that $\mathrm{e}^{- i\cdot t\cdot \phi_{Y}} \cdot f \circ T_{1}^{\phi_{Y}} = f$.  Since the $T_{1}$ is ergodic with respect to $\mu_{1}$ by \cite[Proposition 1.4.8, 1.5.1 and 1.5.3]{JA:1997} we have that $T_{1}^{\phi_{Y}}$ is ergodic with respect to $\mu_{1}$.  Thus, by \cite[Theorem 1.6]{PW:2000}, we obtain that $\lvert f \rvert$ is constant everywhere.  As $f$ is non-zero, this constant is non-zero, and so, we obtain that $\mathrm{e}^{-i\cdot t\cdot \phi_{Y}} = f / (f \circ T_{1}^{\phi_{Y}})$.  However, since for each $n \in \mathbb{N}$, there exists an $x \in Y$ such that $T_{1}^{\phi_{Y}}(x) = x$ and such that $\phi_{Y}(x) = n$, we have that $\mathrm{e}^{-i\cdot t\cdot n} = 1$ for all $n \in \mathbb{N}$.  This contradicts the choice of $t$, namely that $t$ belongs to the open interval $(0, 2\pi)$.
\end{proof}

\end{document}